\documentclass[a4paper, 12pt, english, leqno]{article}
\usepackage[inner=3.5cm,outer=3.1cm,top=4.0cm,bottom=3.9cm]{geometry}
\usepackage{amsfonts,mathrsfs}
\usepackage{amsmath,amsthm,amssymb} 
\usepackage{babel}   
\usepackage{verbatim}
\usepackage[latin1]{inputenc}  
\usepackage{hyperref}
\hypersetup{pdftoolbar=true, pdftitle={Drinfeld},
pdffitwindow=true, colorlinks=true, citecolor=black,
filecolor=black, linkcolor=black, urlcolor=black,
hypertexnames=false}
\usepackage[all]{xy}
\usepackage{macros}     

\makeatletter 
\def\@tocline#1#2#3#4#5#6#7{\relax
  \ifnum #1>\c@tocdepth 
  \else
    \par \addpenalty\@secpenalty\addvspace{#2}%
    \begingroup \hyphenpenalty\@M
    \@ifempty{#4}{%
      \@tempdima\csname r@tocindent\number#1\endcsname\relax
    }{%
      \@tempdima#4\relax
    }%
    \parindent\z@ \leftskip#3\relax \advance\leftskip\@tempdima\relax
    \rightskip\@pnumwidth plus4em \parfillskip-\@pnumwidth
    #5\leavevmode\hskip-\@tempdima #6\nobreak\relax
    \ifnum#1<1\hfill\else\dotfill\fi\hbox to\@pnumwidth{\@tocpagenum{#7}}\par
    \nobreak
    \endgroup
  \fi}
\makeatother

\makeindex
\begin{document}
\pagestyle{plain}


\title{Adelic Openness for Drinfeld Modules\\ in Special Characteristic\\[12pt]}

 \author{Anna Devic\footnote{
 Email: anna.devic@alumni.ethz.ch}
 \hspace{1cm}
 \addtocounter{footnote}{5} 
 Richard Pink\footnote{Dept. of Mathematics,
 ETH Z\"urich,
 CH-8092 Z\"urich,
 Switzerland,
 pink@math.ethz.ch}
}

\date{January 28, 2012}

\maketitle

\begin{abstract}
For any Drinfeld module of special characteristic $\pp_0$ over a finitely generated field, we study the associated adelic Galois representation at all places different from $\pp_0$ and~$\infty$ and determine the images of the arithmetic and geometric Galois groups up to commensurability.
%
%
\end{abstract}

\bigskip
\begin{flushleft}
\\

\end{flushleft}

\parskip=\smallskipamount




\newpage
{\parskip=0pt
\tableofcontents
}


\newpage


\section{Introduction and overview}\label{notation}

\subsection{Main result} 

Let $K$ be a field that is finitely generated over a finite field $\kappa$ of characteristic~$p$. Let $K^{\text{sep}}$ be a fixed separable closure of $K$,
and let $\overline{\kappa}$ be the algebraic closure of $\kappa$ in $K^\text{sep}$. Let $G_K:=\Gal(K^{\text{sep}}/K)$ denote the absolute Galois group and $G_K^{\text{geom}}:=\Gal(K^{\text{sep}}/K \overline{\kappa})$ the geometric Galois group of~$K$. 

Let $F$ be a finitely generated field of transcendence degree $1$ over~$\mathbb{F}_p$. Let $A$ be the ring of elements of $F$ which are regular outside a fixed place $\infty$ of~$F$. Let $\phi: A \rightarrow K\{\tau\}$ be a Drinfeld $A$-module of rank $r$ over $K$ of special characteristic~$\pp_0$. For any prime $\pp\not=\pp_0$ of $A$ let $\rho_\pp: G_K \to \GL_r(A_\pp)$ denote the homomorphism describing the Galois action on the Tate module $T_\pp(\phi)$. We are interested in the image of the associated adelic Galois representation
$$\rho_{\text{ad}} := (\rho_\pp)_\pp:\ 
  G_K \longrightarrow \prod_{\pp\not=\pp_0}\GL_r(A_\pp).$$
By Anderson \cite{Anderson_t_Motives}, \S4.2, it is known that the composite of $\rho_{\text{ad}}$ with the determinant map is the adelic Galois representation associated to some Drinfeld module of rank $1$ of the same characteristic~$\pp_0$. 
Thus the image of $\rho_{\text{ad}}(G_K^{\text{geom}})$ under the determinant is finite: see Proposition \ref{finite} below. Consequently, the image of $\rho_{\text{ad}}(G_K)$ under the determinant is an extension of a finite group and a pro-cyclic group and therefore far from open. Also, the main problem in determining $\rho_{\text{ad}}(G_K)$ lies in determining $\rho_{\text{ad}}(G_K^{\text{geom}}) \cap \prod_{\pp\not=\pp_0}\SL_r(A_\pp)$.

Recall that two subgroups of a group are called commensurable if their intersection has finite index in both. We will show that $\rho_{\text{ad}}(G_K^{\text{geom}})$ is commensurable to an explicit subgroup of $\prod_{\pp\not=\pp_0}\SL_r(A_\pp)$ whose definition depends only on information on certain endomorphism rings associated to~$\phi$. We will also determine $\rho_{\text{ad}}(G_K)$ up to commensurability.

First, since the Galois representation commutes with the endomorphisms of $\phi$ over~$K$, the image of $\rho_{\text{ad}}$ must be contained in the centralizer of $\End_K(\phi)$ in $\prod_{\pp\not=\pp_0}\GL_r(A_\pp)$. Second, enlarging $K$ does not change the image of Galois up to commensurability, but may increase the endomorphism ring. Since all endomorphisms of $\phi$ over any extension of $K$ are defined over a finite separable extension, the relevant endomorphism ring is therefore $\End_{K^{\text{sep}}}(\phi)$.

For a Drinfeld module in generic characteristic it turns out that the image of $\rho_{\text{ad}}$ up to commensurability, which was determined in \cite{PR2}, indeed depends only on $\End_{K^{\text{sep}}}(\phi)$. But in special characteristic this cannot be so, due to a phenomenon described in \cite{PinII}. The problem is that the endomorphism ring of a Drinfeld module in special characteristic can be non-commutative. As a consequence, it is possible that for some integrally closed infinite subring $B\subset A$, the endomorphism ring of the Drinfeld $B$-module $\phi|B$ is larger than that of~$\phi$. The Galois representation associated to $\phi$ must then commute with the additional operators coming from endomorphisms of $\phi|B$, forcing the image of $\rho_{\text{ad}}$ to be smaller. But using the results of \cite{PinII} one can reduce the problem to the case where this phenomenon of growing endomorphism rings does not occur.

For the following results let $a_0$ be any element of $A$ that generates a positive power of~$\pp_0$. View $a_0$ as a scalar element of $\prod_{\pp\not=\pp_0}\GL_r(A_\pp)$ via the diagonal embedding $A\hookrightarrow \prod_{\pp\not=\pp_0} A_\pp$, and let $\smash{\overline{\langle a_0\rangle}}$ denote the pro-cyclic subgroup that is topologically generated by it.

In the simplest case, where the endomorphism ring of $\phi$ over $K^{\text{sep}}$ is $A$ and does not grow under restriction, our main result is the following:

\begin{thm} \label{main_theorem_1}
Let $\phi$ be a Drinfeld $A$-module of rank $r$ over a finitely generated field $K$ of special characteristic $\pp_0$. Assume that for every integrally closed infinite subring $B \subset A$ we have $\End_{K^{\text{sep}}}(\phi|B)=A$. 
Then
\begin{enumerate}
\item[(a)] $\rho_{\text{ad}}( G_K^{\geom})$ is commensurable to $\prod_{\pp\not=\pp_0}\SL_r(A_\pp)$, and
\item[(b)] $\rho_{\text{ad}}( G_K)$ is commensurable to $\overline{\langle a_0\rangle} \cdot \prod_{\pp\not=\pp_0}\SL_r(A_\pp)$.
\end{enumerate}
\end{thm} 

More generally, set $R:= \End_{K^{\text{sep}}}(\phi)$ and $F := \mathop{\rm Quot}(A)$. Assume for the moment that the center of $R$ is~$A$. Then $R \otimes_A F$ is a central division algebra over~$F$ of dimension $d^2$ for some $d$ dividing~$r$. For any prime $\pp\not=\pp_0$ of~$A$, the Tate module $T_\pp(\phi)$ is a module over $R_\pp := R\otimes_AA_\pp$, which is an order in a semisimple algebra over~$F_\pp$. Let $D_\pp$ denote the commutant of $R_\pp$ in $\End_{A_\pp}(T_\pp(\phi))$, which is an order in another semisimple algebra over $F_\pp$. Let $D^1_\pp$ denote the multiplicative group of elements of $D_\pp$ of reduced norm~$1$. This is isomorphic to $\SL_{r/d}(A_\pp)$ for almost all $\pp$ by Proposition \ref{labacu}, and equal to $\SL_r(A_\pp)$ for all $\pp$ if $R=A$. 

In this situation a version of our main result is the following:

\begin{thm} \label{main_theorem}
Let $\phi$ be a Drinfeld $A$-module over a finitely generated field $K$ of special characteristic $\pp_0$. Assume that $R:=\End_{K^{\text{sep}}}(\phi)$ has center $A$ and that for every integrally closed infinite subring $B \subset A$ we have $\End_{K^{\text{sep}}}(\phi|B)=R$. Let $D^1_\pp$ and $\smash{\overline{\langle a_0\rangle}}$ denote the subgroups defined above. Then
\begin{enumerate}
\item[(a)] $\rho_{\text{ad}}( G_K^{\geom})$ is commensurable to $\prod_{\pp\not=\pp_0} D^1_\pp$, and 
\item[(b)] $\rho_{\text{ad}}( G_K)$ is commensurable to $\overline{\langle a_0\rangle} \cdot \prod_{\pp\not=\pp_0} D^1_\pp$.
\end{enumerate}
\end{thm} 

Theorem \ref{main_theorem} is the central result of this article; its special case $R=A$ is just Theorem \ref{main_theorem_1}. Sections \ref{LAG} to \ref{sect_surjective} are dedicated to proving Theorem \ref{main_theorem}. In Section \ref{generalcase} we deduce corresponding results without any assumptions on $\End_{K^{\text{sep}}}(\phi)$ that are somewhat more complicated to state.



\subsection{Outline of the proof}

In this outline we explain the key steps in the proof of Theorem \ref{main_theorem} in the case $R=A$; the general case follows the same principles. So we assume that for every integrally closed infinite subring $B \subset A$ we have $\End_{K^{\text{sep}}}(\phi|B)=\nobreak A$. 
After replacing $K$ by a finite extension, we may assume that $\rho_{\text{ad}}( G_K^{\geom}) \subset \prod_{\pp\not=\pp_0}\SL_r(A_\pp)$. Let $\Gamma_\pp^{\geom}$ denote its image in $\SL_r(A_\pp)$ for any single prime $\pp\not=\pp_0$ of $A$, and let $\Delta_\pp^{\geom}$ denote its image in $\SL_r(k_\pp)$ over the residue field $k_\pp := A/\pp$. 

A large part of the effort goes into proving that $\Delta_\pp^{\geom} = \SL_r(k_\pp)$ for almost all $\pp$. The key arithmetic ingredients for this are the absolute irreducibility of the residual representation from \cite{PinIII}, the Zariski density of $\Gamma^{\geom}_\pp$ in $\SL_{r,F_\pp}$ from \cite{PinI}, and the characterization of $k_\pp$ by the traces of Frobenius elements in the adjoint representation from \cite{PinII}.

In fact, the absolute irreducibility combined with a strong form of Jordan's theorem on finite subgroups of $\GL_r$ from \cite{L-P} shows that $\Delta_\pp^{\geom}$ is essentially a finite group of Lie type in characteristic $p:=\car(F)$. Let $H_\pp$ denote the ambient connected semisimple linear algebraic group over an algebraic closure $\bar k_\pp$ of $k_\pp$. If $H_\pp$ is a proper subgroup of $\SL_{r,\bar k_\pp}$, the eigenvalues of any element of $H_\pp$ must satisfy one of finitely many explicit multiplicative relations that depend only on~$r$. In this case we show that the eigenvalues of any Frobenius element in the residual representation satisfy a similar relation. If this happens for infinitely many $\pp$, the fact that the adelic Galois representation is a compatible system implies that the eigenvalues of Frobenius over any single $F_\pp$ satisfy the same kind of relation. But that is impossible, because $\Gamma^{\geom}_\pp$ is Zariski dense in $\SL_{r,F_\pp}$. Therefore $H_\pp$ is equal to $\SL_{r,\bar k_\pp}$ for almost all $\pp$.

This means that $\Delta_\pp^{\geom}$ is essentially the group of $k'_\pp$-rational points of a model of $\SL_{r,\bar k_\pp}$ over a subfield $k'_\pp \subset\bar k_\pp$. To identify this subfield we observe that the trace in the adjoint representation for any automorphism of the model is an element of $k'_\pp$. We show that this holds in particular for the images of Frobenius elements. But by \cite{PinII} the images of the traces of all Frobenius elements in the adjoint representation of $\SL_r$ generate $k_\pp$ for almost all $\pp$. It follows that $k_\pp\subset k'_\pp$ for almost all $\pp$, and then the inclusion $\Delta_\pp^{\geom} \subset \SL_r(k_\pp)$ must be an equality for cardinality reasons.

We also need to prove that the homomorphism $G_K^{\geom} \to \SL_r(k_{\pp_1}) \times \SL_r(k_{\pp_2})$ is surjective for any distinct $\pp_1,\pp_2$ outside some finite set of primes. This again relies on traces of Frobenius elements. Indeed, if the homomorphism is not surjective, the surjectivity to each factor and Goursat's lemma imply that its image is essentially the graph of an isomorphism $\SL_r(k_{\pp_1}) \stackrel{\sim}{\to} \SL_r(k_{\pp_2})$. This isomorphism must come from an isomorphism of algebraic groups over an isomorphism of the residue fields $\sigma: k_{\pp_1} \stackrel{\sim}{\to} k_{\pp_2}$. Using this we show that the traces of Frobenius in the adjoint representation of $\SL_r$ map to the subring $\mathop{\rm graph}(\sigma) \subset k_{\pp_1}\times k_{\pp_2}$. But that again contradicts the result from \cite{PinII} unless $\pp_1$ or $\pp_2$ is one of finitely many exceptional primes.

Next we prove that $\Gamma_\pp^{\geom} = \SL_r(A_\pp)$ for almost all $\pp$. For this we may already assume that $\Delta_\pp^{\geom} = \SL_r(k_\pp)$. That alone does not imply much, because $A_\pp$ is a local ring of equal characteristic, and so the Teichm\"uller lift of the residue field $k_\pp\hookrightarrow A_\pp$ induces a lift $\SL_r(k_\pp) \hookrightarrow \SL_r(A_\pp)$. But using successive approximation in $\SL_r(A_\pp)$ we reduce the problem to showing that $\Gamma_\pp^{\geom}$ surjects to $\SL_r(A/\pp^2)$. This in turn we can guarantee for almost all $\pp$ using traces of Frobenius elements again.

Indeed, suppose first that $(p,r)\not=(2,2)$. Then the result from \cite{PinII} implies that the images of the traces of all Frobenius elements in the adjoint representation of $\SL_r$ generate $A/\pp^2$ for almost all $\pp$. In particular these traces do not all lie in the Teichm\"uller lift $k_\pp\subset A/\pp^2$, and so the images of Frobenius elements in $\GL_r(A_\pp)$ cannot all lie in the lift of $\GL_r(k_\pp)$. The desired surjectivity $\Gamma_\pp^{\geom} \twoheadrightarrow \SL_r(A/\pp^2)$ follows from this using some group theory.

In the remaining case $p=r=2$ it may happen that the traces of Frobenius in the adjoint representation do not generate the field $F$, but the subfield of squares $F^2 := \{x^2\mid x\in F\}$, of which $F$ is an inseparable extension of degree $2$. This phenomenon stems from the fact that the adjoint representation of $\SL_2$ on $\psll_2$ in characteristic $2$ factors through the Frobenius ${\rm Frob}_2: x\mapsto x^2$. In that case, the result from \cite{PinII} implies that the images of the traces of all Frobenius elements in the adjoint representation of $\SL_r$ generate the subring $k_\pp\oplus \pp^2/\pp^3$ of $A/\pp^3$ for almost all~$\pp$, where $k_\pp$ denotes the canonical Teichm\"uller lift of the residue field $k_\pp$ of~$\pp$. By digging deeper into the structure of $\SL_2(A/\pp^3)$, and replacing $K$ by a finite extension at a crucial step in the argument, we can again show that $\Gamma_\pp^{\geom}$ surjects to $\SL_r(A/\pp^2)$.

Finally, using group theory alone the above results about $\SL_r(k_{\pp_1}) \times \SL_r(k_{\pp_2})$ and $\SL_r(A_\pp)$ imply that the homomorphism $G_K^{\geom} \to \prod_{\pp\not\in P_3}\SL_r(A_\pp)$ is surjective for some finite set of primes $P_3$. On the other hand, the homomorphism $G_K^{\geom} \to \prod_{\pp\in P_3}\SL_r(A_\pp)$ has open image by the main result of \cite{PinII}. While this does not directly imply that the image of the product homomorphism $G_K^{\geom} \to \prod_{\pp\not=\pp_0}\SL_r(A_\pp)$ is open, because the image of a product map may be smaller than the product of the images, some variant of the argument can be made to work, thereby finishing the proof of Theorem \ref{main_theorem}~(a).

Theorem \ref{main_theorem}~(b) is deduced from this as follows. Since $\rho_{\text{ad}}(G_K^{\geom})$ is already open in $\prod_{\pp\not=\pp_0}\SL_r(A_\pp)$, it suffices to show that $\det\rho_{\text{ad}}(G_K)$ is commensurable to $\smash{\overline{\langle a_0\rangle}}$ within $\prod_{\pp\not=\pp_0} A_\pp^\times$. As the determinant of $\rho_{\text{ad}}$ is the adelic Galois representation associated to some Drinfeld module of rank $1$ of the same characteristic~$\pp_0$, this reduces the problem to the case that $r=1$ and that $\phi$ is defined over a finite field, say over $\kappa$ itself. Then $\Frob_\kappa$ acts through multiplication by an element $a\in A$ which is a unit at all primes $\pp\not=\pp_0$ but not at~$\pp_0$. 
It follows that $(a)=\pp_0^i$ for some positive integer~$i$. The same properties of $a_0$ show that $(a_0)=\pp_0^j$ for some positive integer~$j$. Together it follows that $(a^j) = \pp_0^{ij} = (a_0^i)$, and so $a^j/a_0^i$ is a unit in~$A^\times$. As the group of units is finite, we deduce that $a^{j\ell} = a_0^{i\ell}$ for some positive integer~$\ell$. In particular $\rho_{\text{ad}}(G_K) = \smash{\overline{\langle a\rangle}}$ is commensurable to $\smash{\overline{\langle a_0\rangle}}$, as desired. This finishes the proof of Theorem \ref{main_theorem}~(b).


\subsection{Structure of the article}

Section \ref{notation} is the present introduction and overview. Sections \ref{LAG} and \ref{approx} deal with subgroups of $\SL_n$ and $\GL_n$. They are independent of Drinfeld modules, of the rest of the article, and of each other. 

Section \ref{LAG} deals with subgroups of $\SL_n$ and $\GL_n$ over a field and establishes suitable conditions for such a subgroup to be equal to $\SL_n$. It is based on some calculations in root systems, on known results on finite groups of Lie type, and on a strong form of Jordan's theorem from \cite{L-P}.

Section \ref{approx} deals with closed subgroups of $\SL_n$ and $\GL_n$ over a complete discrete valuation ring $R$ of equal characteristic $p$ with finite residue field, and establishes suitable conditions for such a subgroup to be equal to $\SL_n(R)$. The method uses successive approximation over the congruence filtration of $\SL_n(R)$, respectively of $\GL_n(R)$, whose subquotients are related to the adjoint representation. Curiously, the case $p=n=2$ presents special subtleties here, too, because the Lie bracket on $\sll_2$ in characteristic $2$ is not surjective. 

In Section \ref{known_results} we list known results about Drinfeld modules in special characteristic or adapt them slightly to the situation at hand. This includes properties of endomorphism rings, Galois representations on Tate modules, characteristic polynomials of Frobenius, and bad reduction. We also create the setup in which the proof of Theorem \ref{main_theorem} takes place, and list the main arithmetic ingredients from \cite{PinI}, \cite{PinII}, and \cite{PinIII} with their immediate consequences.

Section \ref{sect_surjective} then contains (what remains of) the proof of Theorem \ref{main_theorem}, following the outline expained above.

In Section \ref{generalcase} we determine $\rho_{\text{ad}}(G_K^{\geom})$ and $\rho_{\text{ad}}(G_K)$ up to commensurability for arbitrary Drinfeld modules in special characteristic. The main ingredients for this are the special case of Theorem \ref{main_theorem} and some reduction steps from \cite{PinII}.

\medskip
This article is based on the doctoral thesis of the first author \cite{DevicThesis}; its results are roughly the same as the results there. We are grateful to Florian Pop for pointing out Theorem~\ref{Pop}.

%
%
%

\newpage


\section{Subgroups of $\SL_n$ over a field}\label{LAG}

In a nutshell, the main goal of this section is to establish suitable conditions for subgroups of $\SL_n$ over a field to be equal to $\SL_n$. We first give conditions for root systems to be simple of type $A_\ell$, and then deal with the case of connected semisimple linear algebraic groups over a field. Based on this we treat the case of finite groups of Lie type, which must also take inner forms of $\SL_n$ into account. The main results are Theorems \ref{LAG_main}, \ref{finite_main}, and \ref{finite_main1}. We also recall a strong form of Jordan's theorem from \cite{L-P}.


\subsection{Root systems}\label{21root}

Let $\Phi$ be a non-trivial root system generating a euclidean vector space $E$. Let $W$ be the associated Weyl group, and let $S$ be a $W$-orbit in $E$. We are interested in the conditions:
\begin{enumerate}
\item[(a)] $S$ generates $E$ as a vector space.
\item[(b)] There are no distinct elements $\lambda_1,\ldots,\lambda_4 \in S$ such that $\lambda_1+ \lambda_2=\lambda_3+\lambda_4$.
\item[(c)] There are no distinct elements $\lambda_1,\ldots,\lambda_6 \in S$ such that $\lambda_1+ \lambda_2+\lambda_3=\lambda_4+\lambda_5+\lambda_6$. 
\end{enumerate}

\begin{thm}\label{mainroot}
Assume (a) and (b). Then $\Phi$ is simple of type $A_\ell$ for some $\ell\geq 1$. Moreover, if 
$$\Phi =\{\pm(e_i-e_j) \mid 0\leq i <j \leq \ell \} \subset E= \mathbb{R}^{\ell+1}/\diag(\mathbb{R})$$ 
in standard notation, and if $\ell\not=2$ or in addition (c) is satisfied, then
$$S=\{ c e_i \mid 0 \leq i \leq \ell\}$$
for some constant $c \neq 0$.
\end{thm}

The proof of this result extends over the rest of this subsection. Throughout we assume conditions (a) and (b). Note that (a) implies that $0\not\in S$.

\begin{lem}\label{A1A1}
Let $\lambda \in S$ and $\alpha_1$, $\alpha_2$ be two orthogonal roots in $\Phi.$ Then $\lambda\perp \alpha_1$ or $\lambda\perp \alpha_2.$
\end{lem}

\begin{proof}
Let $s_{\alpha_i}\in W$ denote the simple reflection associated to $\alpha_i$. The fact that $\alpha_1\perp\alpha_2$ implies that
\begin{eqnarray*}
s_{\alpha_i}(\lambda) 
&=& \lambda - \frac{2(\lambda,\alpha_i)}{(\alpha_i,\alpha_i)}\cdot\alpha_i, 
\qquad\hbox{and}\\
s_{\alpha_1}s_{\alpha_2}(\lambda)
&=& \lambda - \frac{2(\lambda,\alpha_1)}{(\alpha_1,\alpha_1)}\cdot\alpha_1
            - \frac{2(\lambda,\alpha_2)}{(\alpha_2,\alpha_2)}\cdot\alpha_2,
\end{eqnarray*}
and hence
$$\lambda + s_{\alpha_1}s_{\alpha_2}(\lambda) =
  s_{\alpha_1}(\lambda)+ s_{\alpha_2}(\lambda).$$ 
But if $\lambda$ is not orthogonal to $\alpha_1$ or $\alpha_2$, these are four distinct elements of $S$, contradicting condition (b).
\end{proof}

\begin{lem}\label{simple}
The root system $\Phi$ is simple.
\end{lem}

\begin{proof}
Assume that $\Phi=\Psi_1 +\Psi_2$ is decomposable and let $\lambda \in S$. Since $\Phi$ generates $E$, there exists an $\alpha \in \Phi$ which is not orthogonal to $\lambda$. Suppose without loss of generality that $\alpha \in \Psi_2$. Then, by Lemma \ref{A1A1}, the vector $\lambda$ is orthogonal to all roots that are orthogonal to $\alpha$; in particular $\lambda \perp \Psi_1$ . Then $w(\lambda) \perp \Psi_1$ for all $w \in W$ and therefore $S \perp \Psi_1.$ However, this contradicts condition (a). 
\end{proof}

\begin{lem}\label{noB2}
The root system $\Phi$ does not contain a root subsystem of type $B_2$.
\end{lem}

\begin{proof}
Assume that $\Psi\subset\Phi$ is a root subsystem of type $B_2$. Then the subspace ${\mathbb R}\Psi$ possesses a basis $\{e_1,e_2\}$ such that $\Psi$ consists of eight roots $\pm e_1$, $\pm e_2$, and $\pm e_1\pm e_2$, and where $e_1\perp e_2$ and $e_1+e_2\perp e_1-e_2$. Thus for any $\lambda\in S$, Lemma \ref{A1A1} implies that $\lambda\perp e_i$ for some $i=1$, $2$, and that $\lambda\perp e_1\pm e_2$ for some choice of sign. Together this gives four cases, in each of which we deduce that $\lambda\perp e_1$. As $\lambda$ was arbitrary, this shows that $S\perp e_1$, contradicting condition (a). 
\end{proof}

\begin{lem}\label{noG2}
The root system $\Phi$ is not of type $G_2$.
\end{lem}

\begin{proof}
Choose simple roots $\alpha_1$, $\alpha_2$ of $\Phi$ such that $\alpha_1$ is the shorter one. Then $\Phi$ contains the root $2\alpha_1+\alpha_2$ which is orthogonal to $\alpha_2$, and the root $3\alpha_1+2\alpha_2$ which is orthogonal to $\alpha_1$. Thus for any $\lambda\in S$, Lemma \ref{A1A1} implies that $\lambda\perp\alpha_2$ or $\lambda\perp 2\alpha_1+\alpha_2$, and that $\lambda\perp\alpha_1$ or $\lambda\perp3\alpha_1+2\alpha_2$. By a simple calculation, each of these four cases implies that $\lambda=0$, contradicting condition (a). 
\end{proof}

\begin{lem}\label{noD4}
The root system $\Phi$ does not contain a root subsystem of type $D_4$.
\end{lem}

\begin{proof}
Assume that $\Psi\subset\Phi$ is a root subsystem of type $D_4$. Then, up to scaling the inner product on $E$, the subspace ${\mathbb R}\Psi$ possesses an orthonormal basis $\{e_1,e_2,e_3,e_4\}$ such that $\Psi$ consists of the roots $\pm e_i\pm e_j$ for all $1\leq i<j\leq4$ and all choices of signs. In particular, the roots $e_1+e_i$ and $e_1-e_i$ are orthogonal for every $2\leq i\leq4$. Thus for any $\lambda\in S$, Lemma \ref{A1A1} implies that $\lambda\perp e_1+\epsilon_ie_i$ for some $\epsilon_i=\pm1$. Since the roots $e_1-\epsilon_2e_2$ and $\epsilon_3e_3+\epsilon_4e_4$ are also orthogonal, Lemma \ref{A1A1} implies that $\lambda\perp e_1-\epsilon_2e_2$ or $\lambda\perp \epsilon_3e_3+\epsilon_4e_4$. Since
$$(e_1+\epsilon_2e_2)+(e_1-\epsilon_2e_2) \ =\ 
(e_1+\epsilon_3e_3)+(e_1+\epsilon_4e_4)-(\epsilon_3e_3+\epsilon_4e_4) \ =\ 2e_1,$$
in both cases we deduce that $\lambda\perp 2e_1$. As $\lambda$ was arbitrary, this shows that $S\perp 2e_1$, contradicting condition (a). 
\end{proof}

Combining Lemmas \ref{simple} through \ref{noD4}, it follows that $\Phi$ is a simple root system of type $A_\ell$ for some $\ell\ge1$. Using standard notation we may identify $E$ with the vector space $\mathbb{R}^{\ell+1}/\diag(\mathbb{R})$, let $e_0,\ldots,e_\ell\in E$ denote the images of the standard basis vectors of $\mathbb{R}^{\ell+1}$, and assume that $\Phi$ consists of the roots $e_i-e_j$ for all distinct $0\leq i,j \leq\ell$. Then its Weyl group is the symmetric group $S_{\ell+1}$ on $l+1$ letters, acting on $E$ by permuting the coefficients. 

Consider any $\lambda\in S$ and write $\lambda = (a_0,\ldots,a_\ell)$ modulo $\diag(\mathbb{R})$. Since $\lambda\not=0$ in~$E$, the coefficients $a_i$ are not all equal. 

\begin{lem}\label{no4distinct}
Suppose that $\ell\ge3$, and consider indices $i$ and $j$ satisfying $a_i\not=a_j$. Then for all indices $i'$ and $j'$ that are distinct from $i$ and $j$ we have $a_{i'}=a_{j'}$.
\end{lem}

\begin{proof}
The assumption implies that $i\not=j$, and the assertion is trivial unless also $i'\not=j'$. Then $e_i-e_j$ and $e_{i'}-e_{j'}$ are orthogonal roots, and so Lemma \ref{A1A1} implies that $\lambda\perp e_i-e_j$ or $\lambda\perp e_{i'}-e_{j'}$. This means that $a_i=a_j$ or $a_{i'}=a_{j'}$; but by assumption only the second case is possible.
\end{proof}

\begin{lem}\label{no3distinct}
If $\ell\ge3$, there exists an index $i$ such that the $a_j$ for all $j\not=i$ are equal.
\end{lem}

\begin{proof}
Since $\ell\geq3$ and the $a_i$ are not all equal, Lemma \ref{no4distinct} implies that the $a_i$ are also not all distinct. Therefore there exist distinct indices $i,j,j'$ satisfying $a_i\not=a_j=a_{j'}$. Then for any $i'\not=i,j$, Lemma \ref{no4distinct} shows that $a_{i'}=a_{j'}$. Thus $i$ has the desired property.
\end{proof}

\begin{lem}\label{no3distinct2}
If $\ell=2$ and in addition condition (c) is satisfied, then the $a_i$ are not all distinct.
\end{lem}

\begin{proof}
Being an orbit under the Weyl group, the set $S$ consists of the vectors
\begin{eqnarray*}
  (a_0,a_1,a_2), & (a_1,a_2,a_0), & (a_2,a_0,a_1), \\
  (a_0,a_2,a_1), & (a_1,a_0,a_2), & (a_2,a_1,a_0)
\end{eqnarray*}
modulo $\diag(\mathbb{R})$.
If the $a_i$ are all distinct, these six vectors are all distinct in $E$, for instance because the positions of the greatest and the smallest coefficient of a vector in $\mathbb{R}^3$ depend only on its residue class modulo $\diag(\mathbb{R})$. As the sum of the three vectors in the first row is equal to the sum of those in the second row, that contradicts condition (c).
\end{proof}

We can now prove Theorem \ref{mainroot}. The statement about $\Phi$ has already been established. To show the statement about $S$, we may assume condition (c) if $\ell=2$. If $\ell\ge2$, using the action of the Weyl group $S_{\ell+1}$, Lemma \ref{no3distinct} or \ref{no3distinct2} implies that $S$ contains an element of the form $(a,b,\ldots,b) \mod \diag(\mathbb{R})$ with $a\not=b$. The same is trivially true if $\ell=1$, because then any non-zero element of $E$ has this form. But for any $\ell\geq1$, the indicated element of $E$ is equal to $ce_1$ with $c=a-b\not=0$. Since $S$ is an orbit under $S_{\ell+1}$, it follows that $S=\{ c e_i \mid 0 \leq i\leq\ell\}$, as desired. This finishes the proof of Theorem \ref{mainroot}. 


\subsection{Some algebraic relations}\label{alg_relations}

From here until the end of this section we fix an integer $n\geq2$. Consider the expression
\begin{myequation}\label{Rprod}
\vcenter{\hbox{$\quad
\begin{array}{l}\displaystyle
     \!\!\prod_{\substack{i_1,i_2 \\ \text{distinct}}} \!
           (\alpha_{i_1}-\alpha_{i_2})
\;\cdot\!\!\prod_{\substack{i_1,i_2,i_3 \\ \text{distinct}}} \!
           (\alpha_{i_1}\alpha_{i_2}-\alpha_{i_3}^2) 
\;\cdot\!\! \\[25pt]
\displaystyle\hskip1.3cm
\;\cdot\!\!\prod_{\substack{i_1,\ldots, i_4 \\ \text{distinct}}} \!
           (\alpha_{i_1}\alpha_{i_2}-\alpha_{i_3}\alpha_{i_4})
\;\cdot\!\!\prod_{\substack{i_1,\ldots, i_6 \\ \text{distinct}}} \!
           (\alpha_{i_1}\alpha_{i_2}\alpha_{i_3}-\alpha_{i_4}\alpha_{i_5}\alpha_{i_6}),
\end{array}
$}}
\end{myequation}%
where the products are extended over all tuples of distinct indices in $\{1,\ldots,n\}$. 
(Note that some of these products are empty for small $n$, but this will not cause any problems.) Clearly this is a symmetric polynomial with integral coefficients in the variables $\alpha_1,\ldots,\alpha_n$. It can therefore be written uniquely as a polynomial with integral coefficients in $\beta_1,\ldots,\beta_n$, where
$$\prod_{i=1}^n \;(T-\alpha_i)\ =\ T^n+ \beta_1 T^{n-1}+\cdots+\beta_n.$$
In particular, we can apply it to the coefficients of the characteristic polynomial $\det(T\cdot{\rm Id}_n-\gamma)$ of a matrix $\gamma \in\GL_n$ over any field $L$ and obtain an algebraic morphism 
\begin{myequation}\label{Rdef}
f: \GL_{n,L} \rightarrow \mathbb{A}_L^1.
\end{myequation}%
By construction, this morphism has the following property:

\begin{lem}\label{Rprop}
For any algebraically closed field $L$ and any matrix $\gamma\in\GL_n(L)$ with eigenvalues $\alpha_1,\ldots,\alpha_n \in L$, listed with their respective multiplicities, we have $f(\gamma)=0$ if and only if one of the following holds:
\begin{enumerate}
\item[(a)] There exist distinct indices $i_1,i_2$ such that $\alpha_{i_1}=\alpha_{i_2}$.
\item[(b)] There exist distinct indices $i_1,i_2,i_3$ such that $\alpha_{i_1}\alpha_{i_2}=\alpha_{i_3}^2$.
\item[(c)] There exist distinct indices $i_1,\ldots,i_4$ such that $\alpha_{i_1}\alpha_{i_2}=\alpha_{i_3}\alpha_{i_4}$.
\item[(d)] There exist distinct indices $i_1,\ldots,i_6$ such that $\alpha_{i_1}\alpha_{i_2}\alpha_{i_3}=\alpha_{i_4}\alpha_{i_5}\alpha_{i_6}$.
\end{enumerate}
\end{lem}

\begin{lem}\label{R_nontrivial}
For any field $L$ and any integer $N \geq 1$, the morphism 
$$\GL_{n,L} \to \mathbb{A}_L^1,\ \gamma \mapsto f(\gamma^N)$$ 
is not identically zero.
\end{lem}

\begin{proof}
Let $T\subset\GL_{n,L}$ be the subgroup of diagonal matrices. Then none of the factors in (\ref{Rprod}) is identically zero on $T$; hence $f$ is not identically zero on $T$. Since the morphism $T\to T$, $\gamma\mapsto \gamma^{nN}$ is surjective, it follows that $\gamma\mapsto f(\gamma^{nN})$ is not identically zero on $T$, and hence not on $\GL_{n,L}$. This implies the desired conclusion.
\end{proof}

%
%


\subsection{Linear algebraic groups}

\begin{thm} \label{LAG_main}
Let $n\geq2$, let $L$ be an algebraically closed field, and let $G$ be a connected semisimple linear algebraic subgroup of $\SL_{n,L}$. Assume that the tautological representation of $G$ on $L^n$ is irreducible and the morphism $f$ from (\ref{Rdef}) does not vanish identically on $G$. Then $G = \SL_{n,L}$.
\end{thm}

\begin{proof}
Let $T$ be a maximal torus of $G$, let $E=X^*(T)\otimes{\mathbb R}$ be the associated character space, let $\Phi\subset E$ be the root system of $G$ with respect to $T$, and let $W$ denote the Weyl group of $\Phi$. The assumption $n\geq2$ and the irreducibility implies that $G$ and hence $\Phi$ is non-trivial.

Let $S\subset E$ be the set of weights of $T$ in the given representation on $L^n$. The fact that the representation is faithful implies that $S$ generates $E$. Let $\lambda\in S$ denote the highest weight of the representation, and let $W\lambda\subset S$ denote its orbit under $W$. Then $S$ is contained in the convex closure of $W\lambda$;
hence $W\lambda$ also generates $E$. 

Next, since the conjugates of $T$ form a Zariski dense subset of $G$, and $f$ does not vanish identically on $G$, it follows that $f$ does not vanish identically on $T$. From this we conclude that
\begin{enumerate}
\item[(a)] none of the weights $\lambda\in S$ has multiplicity $>1$;
\item[(b)] there are no distinct elements $\lambda_1,\lambda_2,\lambda_3 \in S$ such that $\lambda_1+\nobreak\lambda_2\allowbreak=2\lambda_3$;
\item[(c)] there are no distinct elements $\lambda_1,\ldots,\lambda_4 \in S$ such that $\lambda_1+ \lambda_2=\lambda_3+\lambda_4$;
\item[(d)] there are no distinct elements $\lambda_1,\ldots,\lambda_6 \in S$ such that $\lambda_1+ \lambda_2+\lambda_3=\lambda_4+\lambda_5+\lambda_6$;
\end{enumerate}
because by Lemma \ref{Rprop} any one of these relations would imply that $f|T=0$.

In particular, the assumptions of Theorem \ref{mainroot} are satisfied for $\Phi$ and the orbit $W\lambda$. It follows that $\Phi$ is simple of type $A_\ell$ for some $\ell\geq 1$ and that $W\lambda = \{ c e_i \mid 0 \leq i \leq \ell\}$ in standard notation for some constant $c \neq 0$.
Since $W\lambda$ consists of weights, $c$ is an integer. Let us use the standard ordering of $A_\ell$, where the simple roots are $e_{i-1}-e_i$ for all $1\leq i\leq\ell$. The fact that $\lambda$ is dominant then implies that $\lambda=ce_0$ with $c >0$, or $\lambda=ce_\ell$ with $c<0$. These two cases correspond to dual representations which are interchanged by the outer automorphism of $A_\ell$; hence we may assume that $\lambda=ce_0$ and $c>0$.

\begin{lem} \label{premet}
Suppose that $L$ has characteristic $p>0$. Then $0 < c \leq p-1$.
\end{lem}

\begin{proof}
For any integer $d\ge0$ let $V_{d}$ denote the irreducible representation of $\SL_{\ell+1,L}$ with highest weight $de_0$. We know already that there exists a central isogeny $\SL_{\ell+1,L}\twoheadrightarrow G$, such that the pullback of the given representation on $L^n$ is isomorphic to $V_{c}$. Write $c=a+pb$ with integers $0\leq a\leq p-1$ and $b\geq0$. Then by Steinberg's Tensor Product Theorem (cf.\ \cite{HumphMod}, Theorem 2.7) we have $V_{c} \cong V_{a} \otimes V_{b}^{(p)}$, where $(\ )^{(p)}$ denotes the pullback under the absolute Frobenius morphism ${\rm Frob}_p$, which on coordinates is given by $x\mapsto x^p$.

If $a=0$, it follows that the homomorphism $\SL_{\ell+1,L} \twoheadrightarrow G \subset \SL_{n,L}$ factors through ${\rm Frob}_p$, which is not a central isogeny. We must therefore have $a>0$. Suppose that $b>0$. Then the $ae_i$ for $0\leq i\leq\ell$ are distinct weights of $V_{a}$, and the $be_j$ for $0\leq j\leq\ell$ are distinct weights of $V_{b}$; hence the $\lambda_{ij} := ae_i+pbe_j$ for $0\leq i,j\leq\ell$ are distinct weights of $V_{c}$. In other words, the $\lambda_{ij}$ for $0\leq i,j\leq\ell$ are distinct elements of $S$. Since $\lambda_{00}+\lambda_{11}=\lambda_{01}+\lambda_{10}$, this contradicts the property (c) above. We must therefore have $b=0$ and so $0<c\leq p-1$, as desired.
\end{proof}

\begin{lem} \label{postmet}
For all $0\leq i\leq c$ we have $(c-i)e_0+ie_1\in S$.
\end{lem}

\begin{proof}
Consider the simple root $\alpha:=e_0-e_1$, and let $U_{\pm\alpha}\subset G$ denote the two root subgroups, isomorphic to ${\mathbb G}_{a,L}$ and normalized by $T$, corresponding to $\pm\alpha$. Let $H_\alpha\subset G$ denote the subgroup generated by $T$ and $U_{\alpha}$ and $U_{-\alpha}$, whose semisimple part has root system $\{\pm\alpha\}$ of type $A_1$. For any weight $\mu$ let $V_\mu \subset L^n$ denote the associated weight space, and recall that the highest weight is $\lambda=ce_0$. Then the subspace $\bigoplus_{i \in \IZ} V_{ce_0-i\alpha}$ is $H_\alpha$-invariant and irreducible with highest weight $ce_0$ by \cite{Jan}, Part II, Proposition 2.11. 

If $L$ has characteristic $0$, by classical results the representation of the Lie algebra of $H_\alpha$ on this subspace is irreducible with highest weight $ce_0$. If $L$ has characteristic $p>0$, we have $0< c \leq p-1$ by Lemma \ref{premet}, and so the same conclusion holds by \cite{Premet}, Theorem 1.
{}From \cite{HumLie}, Proposition 21.3, it follows that the set of weights of this representation is saturated; in other words these weights are $ce_0-i\alpha$ for all $0\leq i\leq 2(ce_0,\alpha)/(\alpha,\alpha) = c$. They therefore appear in $S$, as desired.
\end{proof}

In particular, Lemma \ref{postmet} implies that $S$ contains the elements
\begin{eqnarray*}
\lambda_1 := ce_0, && \lambda_3 := (c-1)e_0+e_1, \\
\lambda_2 := ce_1, && \lambda_4 := e_0+(c-1)e_1.
\end{eqnarray*}
If $c\ge3$, these elements are all distinct. If $c=2$, we have $\lambda_3=\lambda_4$, but $\lambda_1,\lambda_2,\lambda_3$ are all distinct. Since $\lambda_1+ \lambda_2=\lambda_3+\lambda_4$, we obtain a contradiction to the above property (c) if $c\geq3$, respectively to (b) if $c=2$. We must therefore have $c=1$.

But then $G\cong\SL_{\ell+1,L}$ and the given representation is isomorphic to the standard representation on $L^{\ell+1}$. Thus $\ell+1=n$ and so $G=\SL_{n,L}$, as desired. This finishes the proof of Theorem \ref{LAG_main}.
\end{proof}


\subsection{Finite groups of Lie type}

In this subsection $L$ denotes an algebraically closed field of characteristic $p > 0$.

Let $G$ be a simply connected simple semisimple linear algebraic group over $L$. A surjective endomorphism $F: G \rightarrow G$ whose group of fixed points $G^F$ is finite is called a \emph{Frobenius map} on $G$. For any such $F$, any non-abelian finite simple group isomorphic to a Jordan-H\"{o}lder constituent of $G^F$ is called a \emph{finite simple group of Lie type in characteristic $p$}.

A few small finite simple groups of Lie type have idiosyncrasies that we avoid with the following ad hoc definition. Denote the center of a group $H$ by $Z(H)$.

\begin{defn}\label{regular}
Let $\Gamma$ be a finite simple group of Lie type in characteristic $p$. We call $\Gamma$ \emph{regular} if there exist $G$ and $F$ as above such that 
\begin{enumerate}
\item[(a)] $\Gamma \cong G^F/Z(G^F)$,
\item[(b)] $G^F$ is perfect, and 
\item[(c)] $G^F$ is the universal central covering of $\Gamma$ as an abstract group. 
\end{enumerate}
\end{defn}

\begin{prop}\label{Struct_fgolt} 
Up to isomorphism, there are only finitely many finite simple groups of Lie type, in any characteristic, that are not regular.
\end{prop}

\begin{proof} 
Suppose that $\Gamma$ is a non-abelian Jordan-H\"{o}lder constituent of $G^F$.
Since $G$ is simply connected, by \cite{GoLySo}, \nolinebreak Theorem 2.2.6 \nolinebreak (f) the group $G^F$ is generated by elements whose order is a power of $p$. We can therefore apply \cite{GoLySo}, Theorem 2.2.7, to $G^F$. The first part of this theorem says that, with finitely many exceptions up to isomorphism, $G^F/Z(G^F)$ is non-abelian simple. It is therefore isomorphic to $\Gamma$. The second part says that, with the same exceptions as in the first part, the group $G^F$ is perfect.

As $\Gamma$ is simple and hence perfect, by \cite{GoLySo}, Theorem 5.1.2, it possesses a universal central covering $\Gamma^c\twoheadrightarrow\Gamma$ which is unique up to isomorphism. Its kernel $M(\Gamma)$ is called the \emph{Schur multiplier} of $\Gamma$. 
By \cite{GoLySo}, Theorem 6.1.4, after removing another finite number of exceptions up to isomorphism (these are listed in Table 6.1.3), the Schur multiplier $M(\Gamma)$ is isomorphic to $Z(G^F)$. Since $G^F$ is already perfect with $G^F/Z(G^F)\cong\Gamma$, this implies that $G^F$ is the universal central covering of $\Gamma$. 
Then $\Gamma$ is regular, and the proposition follows.
\end{proof}

The next result is a direct consequence of the stronger statements of \cite{HumphMod}, Theorems 2.11 and 20.2.

\begin{prop}\label{Irred_fgolt}
Let $G$ and $F$ be as above, and let $\rho: G^F \rightarrow \SL_n(L)$ be an irreducible representation on the vector space $L^n$. Then $\rho$ is the restriction to $G^F$ of an irreducible algebraic representation $\rho_G : G \rightarrow \SL_{n,L}$.
\end{prop}

Now we can state our analogues of Theorem \ref{LAG_main}.

\begin{thm}\label{finite_main}
Let $n\geq2$, and let $\Gamma$ be a finite subgroup of $\SL_n(L)$ that acts irreducibly on $L^n$. Assume that $\Gamma$ is perfect and that $\Gamma/Z(\Gamma)$ is a direct product of finite simple groups of Lie type in characteristic $p$ that are regular in the sense of Definition \ref{regular}. Assume moreover that the map $\Gamma\to L$, $\gamma\mapsto f(\gamma)$ is not identically zero. Then there exist a finite subfield $k'$ of $L$ and a model $G'$ of $\SL_{n,L}$ over $k'$ such that $\Gamma=G'(k').$
\end{thm}

\begin{proof}
Let $\overline{\Gamma}_1, \ldots, \overline{\Gamma}_m$ denote the simple factors of $\Gamma/Z(\Gamma)$ and $\Gamma_1,\ldots,\Gamma_m$ their inverse images in~$\Gamma$.
Then the natural homomorphism $\Gamma_1\times\ldots\times\Gamma_m \twoheadrightarrow\Gamma$ is a central extension. By \cite{GoFinite}, Theorem 3.7.1, the pullback of the given irreducible representation on $L^n$ is the exterior tensor product of irreducible representations $\Gamma_i \to \GL_{n_i}(L)$ for certain ${n_i\geq1}$. In fact every $n_i\ge2$, because the corresponding projective representation of $\overline{\Gamma}_1\times\ldots\times\overline{\Gamma}_m$ is faithful.

For each $1 \leq i \leq m$ choose a simply connected simple semisimple linear algebraic group $G_i$ over $L$, a Frobenius map $F_i : G_i \rightarrow G_i$, and an isomorphism $G_i^{F_i}/Z(G_i^{F_i}) \cong \overline{\Gamma}_i$, such that $G_i^{F_i}$ is perfect and the universal central covering of~$\overline{\Gamma}_i$. By the last property the isomorphism lifts to a homomorphism $G_i^{F_i} \to\Gamma_i$. By Proposition \ref{Irred_fgolt} the composite homomorphism $G_i^{F_i} \to\Gamma_i \to \GL_{n_i}(L)$ is the restriction of some irreducible algebraic representation $\rho_i : G_i \rightarrow \GL_{n_i,L}$. Since $G_i$ is simple and $n_i\geq2$, the kernel of this homomorphism is finite.

Set $G := G_1\times\ldots\times G_m$. Then the exterior tensor product of the above $\rho_i$ defines an irreducible algebraic representation $\rho: G \to \GL_{n,L}$. By construction its kernel is finite; in other words it induces an isogeny $G\twoheadrightarrow\rho(G)$. Moreover, with the Frobenius map $F := F_1\times\ldots\times F_m$ on $G$ the homomorphism $\rho$ induces a homomorphism $G^F\to\Gamma$ lifting the given isomorphism
$$G^F/Z(G^F) \ =\ \prod_{i=1}^m G_i^{F_i}/Z(G_i^{F_i}) 
\ \cong\ \prod_{i=1}^m \overline{\Gamma}_i \ \cong\ \Gamma/Z(\Gamma).$$
As $\Gamma$ is perfect, it follows that $G^F\to\Gamma$ is surjective.

Since $G$ is a connected semisimple algebraic group, so is its image $\rho(G)$, which is therefore contained in $\SL_{n,L}$. Moreover, the tautological representation of $\rho(G)$ on $L^n$ is again irreducible. Furthermore, since by assumption the morphism $f$ does not vanish identically on the subgroup $\Gamma \subset \rho(G)$, it does not vanish identically on $\rho(G)$. By Theorem \ref{LAG_main} we therefore have $\rho(G) = \SL_{n,L}$.

In particular $\rho(G)$ is simple of type $A_{n-1}$, and so $G$ itself is simple of type~$A_{n-1}$. 
As $G$ is simply connected, it is therefore isomorphic to $\SL_{n,L}$. Consider the resulting isogeny $\rho: \SL_{n,L}\cong G\twoheadrightarrow\rho(G) = \SL_{n,L}$. 
Its scheme-theoretic kernel is contained in the scheme-theoretic kernel of $\rho\circ F$; hence 
there exists an isogeny $F': \SL_{n,L} \to \SL_{n,L}$ satisfying $F'\circ\rho = \rho\circ F$. On the other hand $\rho$ is bijective; hence it induces an isomorphism from $G^F$ to $\SL_{n,L}^{F'}$. Together it follows that $\Gamma = \SL_{n,L}^{F'}$.

Finally, by known classification results such as \cite{CarterShort}, Proposition 4.5, the Frobenius map $F'$ is standard. This means that there is a finite subfield $k' \subset L$ and a model $G'$ of $\SL_{n,L}$ over $k'$ such that $\SL_{n,L}^{F'}=G'(k')$. Thus Theorem \ref{finite_main} is proved.
\end{proof}

For the next theorem let $c$ denote the least common multiple of the orders of all finite simple groups of Lie type that are not regular in the sense of Definition \ref{regular}, which is finite by Proposition \ref{Struct_fgolt}. Let $\Gamma^{\rm der}$ denote the derived group of $\Gamma$.

\begin{thm}\label{finite_main1}
Let $n\geq2$, and let $\Gamma$ be a finite subgroup of $\GL_n(L)$ that acts irreducibly on~$L^n$. Assume that $\Gamma/Z(\Gamma)$ is a direct product of finite simple groups of Lie type in characteristic $p$. Assume moreover that the map $\Gamma\to L$, $\gamma\mapsto f(\gamma^c)$ is not identically zero. Then there exist a finite subfield $k'$ of $L$ and a model $G'$ of $\SL_{n,L}$ over $k'$ such that $\Gamma^{\rm der}=G'(k').$
\end{thm}

\begin{proof}
Let $\overline{\Gamma}_i$ and $\Gamma_i \to \GL_{n_i}(L)$ be as in the proof of Theorem \ref{finite_main}. Suppose that some factor of $\Gamma/Z(\Gamma)$, say $\overline{\Gamma}_1$, is not regular. Then for every $\gamma\in\Gamma$, the definition of $c$ implies that $\gamma^c \in \Gamma_2\cdots\Gamma_m$. Each eigenvalue of $\gamma^c$ then has multiplicity $\geq n_1\geq2$; hence $f(\gamma^c)=0$ by Lemma \ref{Rprop} (a). This contradicts the given assumptions, and so each $\overline{\Gamma}_i$ is in fact regular.

The assumptions also imply that $\Gamma^{\rm der}$ is perfect and that $\Gamma = \Gamma^{\rm der}\cdot Z(\Gamma)$. Write any $\gamma\in\Gamma$ in the form $\gamma=\gamma'\zeta$ with $\gamma'\in\Gamma^{\rm der}$ and a scalar $\zeta\in Z(\Gamma)$. By construction $f(\alpha)$ is homogeneous of some degree $d$ in the coefficients of~$\alpha$; thus we have
$f(\gamma^c)\ =\ f(\gamma^{\prime c}\zeta^c) = f(\gamma^{\prime c})\cdot\zeta^{cd}$. Since this is not identically zero and $\gamma^{\prime c}\in\Gamma^{\rm der}$, it follows that $f$ is not identically zero on $\Gamma^{\rm der}$.

Together this shows that $\Gamma^{\rm der}$ satisfies the assumptions of Theorem \ref{finite_main}, and so the desired assertion follows.
\end{proof}

The following auxiliary results will help to determine the subfield $k'$ and the model $G'$ arising in Theorems \ref{finite_main} and \ref{finite_main1}:

\begin{prop}\label{fieldcontained}
Let $n\geq2$, and let $G$, $G'$ be models of $\SL_{n,L}$ over finite subfields $k, k'\subset L$, respectively. If $G'(k') \subset G(k)$, then $|k'|\le|k|$. 
\end{prop}

\begin{proof}
Let $q:=|k|$, and set $\epsilon:=1$ if $G$ is split and $\epsilon:=-1$ otherwise. Likewise, let $q':=|k'|$, and set $\epsilon':=1$ if $G'$ is split and $\epsilon':=-1$ otherwise. Then \cite{HumphMod}, Table 1.6.1, implies that
$$(q')^{\frac{n(n-1)}{2}} \prod_{i=2}^n (q^{\prime i}-\epsilon^{\prime i}) 
\ =\ |G'(k')| \ \le\ |G(k)| \ =\ 
   q  ^{\frac{n(n-1)}{2}} \prod_{i=2}^n (q^i-\epsilon^i).$$
Suppose that $q'>q$. Since both numbers are powers of the same prime $p$, it follows that $q'\geq pq\geq2q$. For each $2\leq i\leq n$ we then have $q^{\prime i}-\epsilon^{\prime i} \geq 4q^i-1 > q^i-\epsilon^i$, and so the left hand side of the above inequality is in fact greater than the right hand side, which is impossible. Therefore $q'\leq q$, as desired.
\end{proof}

\begin{prop}\label{same_model}
Let $n\geq2$, and let $G$, $G'$ be models of $\SL_{n,L}$ over the same finite subfield $k\subset L$. If $G'(k) \subset G(k)$, then the models are equal and $G'(k) = G(k)$.
\end{prop}

\begin{proof}
Let $F,F': \SL_{n,L}\to\SL_{n,L}$ be the Frobenius maps corresponding to the models $G,G'$, respectively. Since they belong to the same finite field, there exists an automorphism $\alpha$ of $\SL_{n,L}$ such that $F=\alpha\circ F'$. Then for any $g'\in G'(k)$ we have $g'\in G(k)$ and hence $g' = F(g') = \alpha(F'(g')) = \alpha(g')$. In other words $\alpha$ is the identity on $G'(k)$.

Suppose first that $\alpha$ is an inner automorphism. Then it is conjugation by some element of $\SL_n(L)$. This element commutes with $G'(k)$, and since the standard representation of $G'(k)$ is irreducible by Proposition \ref{Irred_fgolt}, it must be a scalar. Then $\alpha$ is the identity, and so $F=F'$ and $G=G'$, as desired.

If $\alpha$ is not an inner automorphism, we must have $n\geq3$. Let $\mathfrak{psl}_n(L)$ denote the image of the natural homomorphism of Lie algebras $\sll_n(L)\to\pgll_n(L)$, and let $\rho$ denote the representation on $\mathfrak{psl}_n(L)$ induced by the adjoint representation of $\SL_{n,L}$. Since $n\geq3$, we know that $\rho$ factors through a faithful irreducible representation of $\PGL_{n,L}$. Moreover, by Proposition \ref{Irred_fgolt} it remains irreducible when restricted to $G'(k)$. On the other hand $\alpha$ induces an automorphism $\bar\alpha$ of $\mathfrak{psl}_n(L)$ that commutes with $\rho(G'(k))$. Thus $\bar\alpha$ is multiplication by a scalar, and therefore it commutes with the algebraic representation $\rho$. It follows that $\alpha$ induces the identity on $\PGL_{n,L}$. But then it is really an inner automorphism, contrary to the assumption. 
\end{proof}

\begin{prop}\label{same_k_and_model}
The subfield $k'$ and the model $G'$ in Theorems \ref{finite_main} and \ref{finite_main1} are unique.
\end{prop}

\begin{proof}
Let $k$ be another finite subfield of $L$, and let $G$ be a model of $\SL_{n,L}$ over $k$, such that $\Gamma=G(k)$. Then applying Proposition \ref{fieldcontained} in both ways shows that $|k'|=|k|$. Thus $k'=k$, and then Proposition \ref{same_model} shows that $G'=G$, as desired.
\end{proof}

%


\subsection{Arbitrary finite groups} \label{irre_subgroups}

The following general result was established by Larsen and the second author in \cite{L-P}, Theorem 0.2:

\begin{thm} \label{L-P}
For any integer $n \geq 1$ there exists a constant $c_n$, such that for every field $L$, of arbitrary characteristic $p\geq0$, and every finite subgroup $\Gamma\subset \GL_n(L)$, there exist normal subgroups $\Gamma_3 \subset \Gamma_2 \subset \Gamma_1$ of $\Gamma$ satisfying:
\begin{enumerate}
\item[(a)] $[\Gamma : \Gamma_1]\leq c_n$,
\item[(b)] either $\Gamma_1=\Gamma_2$, or $p>0$ and $\Gamma_1/\Gamma_2$ is a direct product of finite simple groups of Lie type in characteristic $p$,
\item[(c)] $\Gamma_2/\Gamma_3$ is abelian of order not divisible by $p$, and
\item[(d)] either $\Gamma_3=\{1\}$, or  $p>0$  and $\Gamma_3$ is a $p$-group.
\end{enumerate}
\end{thm}

We are interested in the following special case:

\begin{thm} \label{L-P3}
For any integer $n \geq 1$ there exists a constant $c'_n$, such that for every algebraically closed field $L$, of arbitrary characteristic $p\geq0$, and every finite subgroup $\Gamma\subset \GL_n(L)$ acting irreducibly, there exists a normal subgroup $\Gamma' \triangleleft\Gamma$ satisfying:
\begin{enumerate}
\item[(a)] $[\Gamma : \Gamma']\leq c'_n$, and
\item[(b)] either $\Gamma'=Z(\Gamma')$, or $p>0$ and $\Gamma'/Z(\Gamma')$ is a direct product of finite simple groups of Lie type in characteristic $p$.
\end{enumerate}
\end{thm}

\begin{proof}
Let $\Gamma_3 \subset \Gamma_2 \subset \Gamma_1\subset\Gamma$ be the subgroups furnished by Theorem \ref{L-P}. First we show that $\Gamma_3$ is trivial. By assumption, this is a unipotent normal subgroup of~$\Gamma$. Set $V:=L^n$. Then the subspace of invariants $V^{\Gamma_3}$ is non-zero and stabilized by $\Gamma$. Since $V$ is an irreducible representation of $\Gamma$, it follows that $V^{\Gamma_3}=V$. This means that $\Gamma_3=\{ 1\}$, as desired.

The triviality of $\Gamma_3$ implies that $\Gamma_2$ is an abelian group of order not divisible by $p$. Let $V=V_1 \oplus \ldots \oplus V_m$ be the isotypic decomposition under $\Gamma_2$, with all summands non-zero. The number of summands then satisfies $m\le n$. Since $\Gamma_2$ is normal in $\Gamma$, the summands are permuted by $\Gamma$, and so the permutation action corresponds to a homomorphism $f$ from $\Gamma$ to the symmetric group $S_m$ on $m$ letters. Set $\Gamma' := \Gamma_1\cap\mathop{\rm ker}(f)$. By construction this is a normal subgroup of index ${[\Gamma:\Gamma']} \allowbreak \leq {[\Gamma:\Gamma_1]\cdot |S_m|} \allowbreak \leq c_n\cdot n! =:c'_n$, where $c_n$ is the constant from Theorem \ref{L-P}. 

On the other hand, the fact that $\Gamma_2$ acts by scalars on each $V_i$ and $\Gamma'$ stabilizes each $V_i$ implies that $\Gamma_2$ is contained in the center of~$\Gamma'$. Moreover $\Gamma'/\Gamma_2$ is the kernel of a homomorphism $\Gamma_1/\Gamma_2\to S_m$ induced by $f$. Since $\Gamma_1/\Gamma_2$ is a direct product of non-abelian finite simple groups, this kernel is simply a direct product of some of the factors. Thus either $\Gamma'=\Gamma_1=\Gamma_2$, or $p>0$ and $\Gamma'/\Gamma_2$ is a direct product of finite simple groups of Lie type in characteristic $p$. The last statement also implies that the inclusion $\Gamma_2\subset Z(\Gamma')$ must be an equality, and everything is proved.
\end{proof}


\newpage


\section{Subgroups of $\SL_n$ over a complete valuation ring} \label{approx}

Let $R$ be a complete discrete valuation ring of equal characteristic with finite residue field $k=R/\pp$ of characteristic $p$. Fix an integer $n\geq2$. In this section we consider closed subgroups of $\SL_n(R)$ for the topology induced by~$R$ and establish suitable conditions for such a closed subgroup to be equal to $\SL_n(R)$. We use successive approximation over the congruence filtration of $\SL_n(R)$, respectively of $\GL_n(R)$, whose subquotients are related to the adjoint representation. The case $p=n=2$ presents some special subtleties here, because the Lie bracket on $\sll_2$ in characteristic $2$ is not surjective. In Subsection \ref{trace_crit} we show how a certain non-triviality condition required earlier can be guaranteed using traces in the adjoint representation. The main results are Theorems \ref{strong_approx_SLn}, \ref{strong_approx_GLn}, \ref{trace_crit_1}, and \ref{trace_crit_2}.


\subsection{Adjoint representation}

We first collect a few general results on the cohomology and subgroups of the adjoint representation. Let $\gll_n$, $\sll_n$, $\pgll_n$ denote the Lie algebras of $\GL_n$, $\SL_n$, $\text{PGL}_n$, respectively. As usual we identify elements of $\gll_n$ with $n\times n$-matrices. Let $\mathfrak{c}$ denote the subspace of scalar matrices in $\gll_n$. For any field $k$ let $\mathfrak{psl}_n(k)$ denote the image of the natural homomorphism $\sll_n(k)\to\pgll_n(k)$. 

\begin{prop}\label{cohom1}
For any finite field $k$ with $|k|>9$ and any subgroup $H$ of $\GL_n(k)$ that contains $\SL_n(k)$, we have
$$H^1(H,\pgll_n(k))=0.$$
\end{prop}

\begin{proof}
Consider the short exact sequence
$$0 \longrightarrow \mathfrak{c}(k) \longrightarrow \gll_n(k) \longrightarrow \pgll_n(k) \longrightarrow 0.$$ 
Its associated long exact cohomology sequence contains the portion
$$H^1(\SL_n(k),\gll_n(k)) \longrightarrow H^1(\SL_n(k),\pgll_n(k)) \longrightarrow H^2(\SL_n(k),\mathfrak{c}(k)).$$
Here the group on the left is trivial by \cite{TaZa}, Theorem 9. The group on the right classifies central extensions of $\SL_n(k)$ by $\mathfrak{c}(k)$; but since $\SL_n(k)$ has no central extensions by \cite{Ste2}, Theorem 1.1, if $|k|>9$, this group is also trivial. Thus the group in the middle is trivial. Finally, since $[H:\SL_n(k)]$ divides $[\GL_n(k):\SL_n(k)]= |k|-1$, it is prime to the characteristic of $k$. By \cite{CPS1}, Proposition 2.3 (g), the restriction map 
$$H^1(H,\pgll_n(k)) \longrightarrow H^1(\SL_n(k),\pgll_n(k))$$
is therefore injective. Thus $H^1(H,\pgll_n(k))$ is trivial, as desired.
\end{proof}

The next proposition is an adaptation of \cite{PR2}, Proposition 2.1.

\begin{prop}\label{Egon2.1}
Let $n\geq 2$ and $k$ be a finite field with $|k|>9$. Let $H$ be an additive subgroup of $\mathfrak{gl}_n(k)$ that is invariant under conjugation by $\SL_n(k)$. Then either $H\subset\mathfrak{c}(k)$ or $\mathfrak{sl}_n(k)\subset H$. 
\end{prop}

\begin{proof}
Let $W_0\subset\gll_n(k)$ denote the subgroup of diagonal matrices. For each pair of distinct indices $1\leq i,j\leq n$, let $W_{i,j}\subset\gll_n(k)$ denote the group of matrices with all entries zero except, possibly, in the position $(i,j)$. Then we have the decomposition
$$\mathfrak{gl}_n(k)=W_0\oplus \bigoplus_{i\not=j}W_{i,j}.$$ 
This decomposition is invariant under conjugation by the group of diagonal matrices $T(k)\subset\GL_n(k)$. Indeed, an element $t=\diag(t_1,\ldots,t_n)\in T(k)$ acts trivially on $W_0$ and by multiplication with $\chi_{i,j}(t) := t_i/t_j$ on $W_{i,j}$. Set $T'(k) := T(k)\cap\SL_n(k)$, and let ${\mathbb F}_p$ denote the prime field of $k$.

\begin{lem}\label{tolstoi}
The $W_{i,j}$ are non-trivial and irreducible viewed as representations of $T'(k)$ over ${\mathbb F}_p$. If $|k|>9$, they are pairwise non-isomorphic. Furthermore, they are permuted transitively under conjugation by the normalizer of $T'(k)$ in $\SL_n(k)$.
\end{lem}

\begin{proof}
The last assertion follows from the fact that the $W_{i,j}$ are permuted transitively by the permutation matrices and that every permutation matrix can be moved into $\SL_n(k)$ by changing the sign of at most one entry. 

For the first assertion it thus suffices to consider $W_{1,2}$. The calculation $\chi_{1,2}(\diag(x,x^{-1},1,\ldots,1))=x^2$ shows that $(k^\times)^2 \subset \chi_{1,2}(T'(k)) \subset k^\times$. Since every element of a finite field $k$ can be written as a sum of two squares, this subgroup generates $k$ as an ${\mathbb F}_p$-algebra. As  $W_{1,2}$ is a one-dimensional vector space over $k$, it is therefore irreducible as a representation of $T'(k)$ over ${\mathbb F}_p$. Since $(k^\times)^2\not=\{1\}$ by assumption, this representation is non-trivial. 

For the remaining assertion suppose that two distinct $W_{i,j}$ and $W_{i',j'}$ are isomorphic as representations of $T'(k)$ over ${\mathbb F}_p$. This means that $\chi_{i,j}|T'(k) = \chi_{i',j'}^{p^m}|T'(k)$ for some $m\ge0$. Without loss of generality we may assume that $(i',j')=(1,2)$. Suppose first that $(i,j)=(2,1)$. Then by applying the equation to elements of the form $\diag(x,x^{-1},1,\ldots,1)$ we find that $x^{-2} = x^{2p^m}$ for all $x\in k^\times$. By an explicit calculation that we leave to the reader, this is not possible if $|k|>9$ (and this bound cannot be improved if $n=2$!). If $i,j>2$, the element $\diag(x,x^{-1},1,\ldots,1)$ acts as multiplication by $x^2$ on $W_{1,2}$ and trivially on $W_{i,j}$. Since $(k^\times)^2\not=\{1\}$ by assumption, again the representations cannot be isomorphic. If precisely one of $i,j$ is $\le 2$, we may assume that the other is~$3$. Then the element $\diag(x,x,x^{-2},1,\ldots,1)$ acts trivially on $W_{1,2}$ and as multiplication by $x^{\pm3}$ on $W_{i,j}$. Since $(k^\times)^3\not=\{1\}$ by assumption, we again obtain a contradiction.
\end{proof}

Now let $H$ be as in the proposition. Suppose first that $H\subset W_0$. Consider an arbitrary element $h=\diag(h_1,\ldots,h_n)\in H$. Take distinct indices $1\leq i,j\leq n$ and let $g \in \SL_n(k)$ be the matrix with entries $1$ on the diagonal and in the position $(i,j)$, and entries $0$ elsewhere. Then $ghg^{-1}-h$ is the matrix with entry $h_i-h_j$ in the position $(i,j)$ and entries $0$ elsewhere. Since $H\subset W_0$, it follows that $h_i=h_j$. Varying $i$ and $j$ we deduce that $h$ is a scalar matrix, i.e., that $H\subset\mathfrak{c}(k)$.

If $H\not\subset W_0$, Lemma \ref{tolstoi} implies that $H$ contains at least one, and hence all $W_{i,j}$. Consider the trace form 
$$\mathfrak{gl}_n(k) \times \mathfrak{gl}_n(k) \to {\mathbb F}_p, 
\quad (X,Y)\mapsto \Tr_{k/{\mathbb F}_p}\Tr(XY),$$ 
which is a perfect ${\mathbb F}_p$-bilinear pairing invariant under $\SL_n(k)$. Then $H$ contains the orthogonal complement $W_0^\perp$ of $W_0$, and since taking orthogonal complements reverses inclusion relations, the orthogonal complement $H^\perp$ of $H$ is contained in $W_0$. By construction $H^\perp$ is again an $\SL_n(k)$-invariant subgroup; hence the preceding arguments show that $H^\perp\subset\mathfrak{c}(k)$. Therefore 
$\mathfrak{sl}_n(k) = \mathfrak{c}(k)^\perp \subset H$, as desired.
\end{proof}

\begin{cor}\label{Egon2.1_cor}
Let $n\geq 2$ and $k$ be a finite field with $|k|>9$. Let $H$ be a non-zero additive subgroup of $\mathfrak{pgl}_n(k)$ that is invariant under conjugation by $\SL_n(k)$. Then $H$ contains $\mathfrak{psl}_n(k)$.
\end{cor}

\begin{proof}
Apply Proposition \ref{Egon2.1} to the inverse image $\tilde H\subset\mathfrak{gl}_n(k)$ of $H$. Since $H$ is non-trivial, we have $\tilde H\not\subset\mathfrak{c}(k)$ and hence $\mathfrak{sl}_n(k)\subset \tilde H$. Therefore $\mathfrak{psl}_n(k)\subset H$, as desired.
\end{proof}


\subsection{Successive approximation} \label{succc}

The congruence filtration of $\GL_n(R)$ consists of the subgroups
$$G^i := \{ g\in\GL_n(R)\mid g\equiv{\rm Id}_n\mathrel{\rm mod}\pp^i\}$$
for all $i\geq0$. Their successive subquotients possess natural isomorphisms
$$G^{[i]}\ :=\ G^i/G^{i+1}\ \cong\ 
\left\{\begin{array}{ll}
\GL_n(k) & \hbox{if $i=0$,}\\[5pt]
\gll_n(\pp^i/\pp^{i+1}) & \hbox{if $i>0$,}
\end{array}\right.$$
where the second isomorphism is given by $[{\rm Id}_n+X] \mapsto [X]$. For any subgroup $H$ of $\GL_n(R)$ we define $H^i:= H\cap G^i$ and $H^{[i]}:= H^i/H^{i+1}$ and identify the latter with a subgroup of $\GL_n(k)$ or $\gll_n(\pp^i/\pp^{i+1})$, respectively. For example, the induced congruence filtration of $G' := \SL_n(R)$ has the successive subquotients
$$G^{\prime[i]}\ \cong\ 
\left\{\begin{array}{ll}
\SL_n(k) & \hbox{if $i=0$,}\\[5pt]
\sll_n(\pp^i/\pp^{i+1}) & \hbox{if $i>0$.}
\end{array}\right.$$
As a preparation we show:

\begin{lem} \label{camilla}
Assume that $|k|>9$. Let $H$ be a subgroup of $\GL_n(R)$, and let $H'$ be a normal subgroup of $H$ that is contained in $\SL_n(R)$. Assume that $\SL_n(k) \subset H^{\prime[0]}$ and that $H^{[1]}$ contains a non-scalar matrix. Then we have $H^{\prime[1]}=\sll_n(\pp/\pp^2)$.
\end{lem}

\begin{proof}
Pick a non-scalar matrix $X\in H^{[1]} \subset \gll_n(\pp/\pp^2)$. Since $\pp/\pp^2\cong k$, Proposition \ref{Egon2.1} implies that the $k$-vector subspace generated by all $\SL_n(k)$-conjugates of $X$ contains $\sll_n(\pp/\pp^2)$. Thus there exists $\gamma\in\SL_n(k)$ such that $X$ and $\gamma X\gamma^{-1}$ and $\Id_n$ are $k$-linearly independent. Choose elements $h\in H^1$ and $h'\in H'$ that are mapped to $X$ and~$\gamma$, respectively, i.e., that satisfy $h\equiv \Id_n+X\mathrel{\rm mod}\pp^2$ and $h'\equiv\gamma\mathrel{\rm mod}\pp$. Then the commutator $h'hh^{\prime-1}h^{-1}$ lies in $H'$ and is congruent to $\Id_n + \gamma X\gamma^{-1}-X \mathrel{\rm mod}\pp^2$. By construction $\gamma X\gamma^{-1}-X \mathrel{\rm mod}\pp^2$ is not scalar; hence $H^{\prime[1]}\subset\sll_n(\pp/\pp^2)$ contains a non-scalar matrix.

On the other hand the isomorphism $G^{[1]}\stackrel{\sim}{\to}\gll_n(\pp/\pp^2)$, $[{\rm Id}_n{+}X] \mapsto [X]$ is equivariant under conjugation by $\GL_n(R)$. This conjugation action factors through an action of $\GL_n(k)$. In the present situation it follows that $H^{\prime[1]}\subset\sll_n(\pp/\pp^2)$ is an additive subgroup that is invariant under conjugation by $\SL_n(k)\subset H^{\prime[0]}$. Since by assumption it also contains a non-scalar matrix, Proposition \ref{Egon2.1} implies that $H^{\prime[1]} = \sll_n(\pp/\pp^2)$, as desired.
\end{proof}


\begin{thm}\label{strong_approx_SLn}
Assume that $|k|>9$. Let $H$ be a closed subgroup of $\SL_n(R)$ such that $H^{[0]}=\SL_n(k)$ and that $H^{[1]}$ contains a non-scalar matrix. Then $H=\SL_n(R)$.
\end{thm}

Before proving this we derive two consequences. For a closed subgroup $H$ of $\GL_n(R)$ we let $H^{\rm der}$ denote the closure of the derived group of $H$ for the topology induced by~$R$. (Probably the derived group is already closed, but we neither need nor want to worry about that.)

\begin{thm}\label{strong_approx_GLn}
Assume that $|k|>9$. Let $H$ be a closed subgroup of $\GL_n(R)$ such that $\SL_n(k)\subset H^{[0]}$ and that $H^{[1]}$ contains a non-scalar matrix. 
Then $H^{\rm der}=\SL_n(R)$.
\end{thm}

\begin{proof}
Since $|k|>3$, the group $\SL_n(k)$ is perfect, and so the assumption $\SL_n(k)\subset H^{[0]}$ implies that $(H^{\rm der})^{[0]} = \SL_n(k)$. Since $H^{[1]}$ contains a non-scalar matrix, applying Lemma \ref{camilla} with $H'=H^{\rm der}$ thus shows that $(H^{\rm der})^{[1]}=\sll_n(\pp/\pp^2)$; in particular it contains a non-scalar matrix. The desired assertion results by applying Theorem \ref{strong_approx_SLn} to $H^{\rm der}$.
\end{proof}

\begin{prop}\label{normal_sa}
Assume that $|k|>9$. Then every closed normal subgroup $H\subset\SL_n(R)$ satisfying $H^{[0]}=\SL_n(k)$ is equal to $\SL_n(R)$.
\end{prop}

\begin{proof}
Applying Lemma \ref{camilla} 
with $(H,\SL_n(R))$ in place of $(H',H)$ 
shows that $H^{[1]}$ contains a non-scalar matrix. The desired assertion now follows directly from Theorem \ref{strong_approx_SLn}.
\end{proof}

\begin{proof}[Proof of Theorem \ref{strong_approx_SLn}.]
The proof of this will extend to the end of the next subsection. Let $H$ satisfy the assumptions of Theorem \ref{strong_approx_SLn}. Since $H$ is a closed subgroup of $\SL_n(R)$, the desired assertion is equivalent to $H^{[i]}= \SL_n(R)^{[i]} $ for all $i\geq 0$. By assumption this holds already for $i=0$. Applying Lemma \ref{camilla} with $H'=H$ implies:

\begin{lem}\label{lemma12.6}
We have $H^{[1]}=\sll_n(\pp/\pp^2)$.
\end{lem}

\begin{lem}\label{lemma12.7}
If $(p,n)\neq (2,2)$, then $H^{[i]}=\sll_n(\pp^i/\pp^{i+1})$ for all $i\geq 1$.
\end{lem}

\begin{proof}
We use induction on $i$, the case $i=1$ being covered by Lemma \ref{lemma12.6}. So suppose that the assertion holds for some $i\geq1$. Consider the commutator map $(g,g')\mapsto gg'g^{-1}g^{\prime-1}$ on $\GL_n(R)$. Direct calculation shows that it induces a map $G^1\times G^i\to G^{i+1}$ and hence a map $G^{[1]}\times G^{[i]} \to G^{[i+1]}$, which under the above isomorphisms corresponds to the Lie bracket 
\begin{eqnarray*}
[\ \:,\ ] :\ \sll_n(\pp/\pp^2) \times \sll_n(\pp^i/\pp^{i+1}) 
        \!\!\!&\longrightarrow&\!\!\! \sll_n(\pp^{i+1}/\pp^{i+2}), \\
        (X,Y) &\mapsto& XY-YX.
\end{eqnarray*}
Since $(p,n)\neq (2,2)$, the image of this pairing generates $\sll_n(\pp^{i+1}/\pp^{i+2})$ as an additive group, for instance by \cite{PSTAX}, Proposition 1.2 (a). 

On the other hand, by construction the pairing sends the subset $H^{[1]} \times H^{[i]}$ to the subgroup $H^{[i+1]}$.
The equality in the source thus implies equality in the target, and so the assertion holds for $i+1$, as desired.
\end{proof}

This proves Theorem \ref{strong_approx_SLn} in the case $(p,n)\neq (2,2)$. The remaining case is more complicated, because the image of the Lie bracket on $\sll_2$ in characteristic $2$ does not generate $\sll_2$. We deal with this in the next subsection.


\subsection{Successive approximation in the case $p=n=2$} 

Keeping the assumptions of Theorem \ref{strong_approx_SLn}, we now consider the case $p=n=2$.

Let $\overline{G}{}^i$ for all $i\ge0$ denote the subgroups in the congruence filtration of $\PGL_2(R)$. Thus $\overline{G}{}^i$ consists of all elements of $\PGL_2(R)$ whose images in the adjoint representation are congruent to the identity modulo $\pp^i$. Their successive subquotients possess natural isomorphisms
$$\overline{G}{}^{[i]}\ :=\ \overline{G}{}^i/\overline{G}{}^{i+1}\ \cong\ 
\left\{\begin{array}{ll}
\PGL_2(k) & \hbox{if $i=0$,}\\[5pt]
\pgll_2(\pp^i/\pp^{i+1}) & \hbox{if $i>0$.}
\end{array}\right.$$
We will compare the congruence filtration of $\overline{G}$ with that of $G' := \SL_n(R)$. Let $\pi$ denote the projection $\SL_2\to\PGL_2$. For $i>0$ the induced map $G^{\prime[i]}\to\overline{G}{}^{[i]}$ corresponds to the derivative $d\pi: \sll_2(\pp^i/\pp^{i+1}) \to \pgll_2(\pp^i/\pp^{i+1})$. Let $\mathfrak{psl}_2(\pp^i/\pp^{i+1})$ denote its image. Being in characteristic $2$, we have short exact sequences
$$\xymatrix@R-18pt{
0\ar[r] & \mathfrak{c}(\pp^i/\pp^{i+1}) 
 \ar[r] & \sll_2(\pp^i/\pp^{i+1}) 
 \ar[r] & \psll_2(\pp^i/\pp^{i+1}) \ar[r] & 0, \\
0\ar[r] & \psll_2(\pp^i/\pp^{i+1}) 
 \ar[r] & \pgll_2(\pp^i/\pp^{i+1})
 \ar[r]^{(*)} & \mathfrak{c}^*(\pp^i/\pp^{i+1}) \ar[r] & 0, \\}$$
where $\mathfrak{c}(\pp^i/\pp^{i+1}) \cong \mathfrak{c}^*(\pp^i/\pp^{i+1}) \cong \pp^i/\pp^{i+1}$ and the homomorphism $(*)$ is induced by the trace on $\gll_2$.

The first few layers of $\SL_2(R)$ and $\PGL_2(R)$ are related as follows. Since $k$ has characteristic $2$, the homomorphism $\SL_2(k) \to \PGL_2(k)$ is an isomorphism of abstract groups. Therefore $\pi^{-1}(\overline{G}{}^1) = G^{\prime1}$. Next consider the subgroup $G^{\prime2-} := \pi^{-1}(\overline{G}{}^2) \subset \SL_2(R)$. What we have just said implies that $G^{\prime2} \subset G^{\prime2-} \subset G^{\prime1}$. By construction $\pi$ induces a natural homomorphism $G^{\prime2-}/G^{\prime3} \to \overline{G}{}^{[2]}$.

\begin{lem}\label{tempe}
There is an isomorphism $\mathfrak{c}(\pp/\pp^2) \stackrel{\sim}{\longrightarrow} \mathfrak{c}^*(\pp^2/\pp^3)$ such that the following diagram commutes:
$$\xymatrix@R-20pt{
0\ar[r] & G^{\prime2}/G^{\prime3} \ar[r] \ar@{=}[dd]^-\wr & 
G^{\prime2-}/G^{\prime3} \ar[r] \ar[ddd]_{[\pi]} & 
G^{\prime2-}/G^{\prime2} \ar[r] \ar@{=}[dd]^-\wr & 0 \\
&&&& \\
& \sll_2(\pp^2/\pp^3) \ar@{->>}[ddd]_{d\pi} && 
\mathfrak{c}(\pp/\pp^2) \ar[ddd]_\cong & \\
&& \overline{G}{}^{[2]} \ar@{=}[dd]^-\wr && \\
&&&& \\
0\ar[r] & \psll_2(\pp^2/\pp^3) \ar[r] & 
\pgll_2(\pp^2/\pp^3) \ar[r] & 
\mathfrak{c}^*(\pp^2/\pp^3) \ar[r] & 0\rlap{.} \\}$$
\end{lem}

\begin{proof}
The commutativity on the left hand side is already clear. Let $\varpi\in\pp$ be a uniformizer. An easy calculation shows that $G^{\prime2-}$ consists of the matrices 
$$g\ =\ \left(\!\begin{array}{cc}1+\varpi x & 0 \\[5pt] 
  0 & (1+\varpi x)^{-1} \end{array}\!\!\right) 
  \cdot g_2$$
for all $x\in R$ and $g_2\in\SL_2(R)$ which satisfy $g_2\equiv{\rm Id_2}\mathrel{\rm mod}\pp^2$. The residue class of $g$ in $\mathfrak{c}(\pp/\pp^2) \cong \pp/\pp^2$ is then $\varpi x\mathrel{\rm mod}\pp^2$. On the other hand $\pi(g) \in \PGL_2(R)$ is also the image of the matrix
$$\tilde g\ :=\ (1+\varpi x)\cdot g\ =\ 
  \left(\!\begin{array}{cc}1+\varpi^2 x^2 & 0 \\[5pt] 
  0 & 1 \end{array}\!\right) 
  \cdot g_2 \ \in\ \GL_2(R).$$
Since $\tilde g\equiv{\rm Id_2}\mathrel{\rm mod}\pp^2$, its image in $\mathfrak{c}^*(\pp^2/\pp^3) \cong \pp^2/\pp^3$ is simply $\Tr(\tilde g-\Id_2)$. But the assumptions $g_2\in\SL_2(R)$ and $g_2\equiv{\rm Id_2}\mathrel{\rm mod}\pp^2$ imply that $\Tr(g_2-\Id_2) \in \pp^3$, and so an easy calculation shows that
$$\Tr(\tilde g-\Id_2)\ \equiv\ \varpi^2 x^2 + \Tr(g_2-\Id_2)
\ \equiv\ \varpi^2 x^2 \mod \pp^3.$$
Thus the diagram in question commutes with the map 
$$\mathfrak{c}(\pp/\pp^2) \longrightarrow \mathfrak{c}^*(\pp^2/\pp^3),
\quad (\varpi x\mathrel{\rm mod}\pp^2) \mapsto (\varpi^2 x^2\mathrel{\rm mod}\pp^3).$$
Up to multiplication by $\varpi$, respectively $\varpi^2$, this is simply the Frobenius map $x\mapsto x^2$ on the finite field $k$. It is therefore an isomorphism, as desired.
\end{proof}

Now let $\overline{H}$ denote the image of $H$ in $\PGL_2(R)$. Define $\overline{H}{}^i:= \overline{H}\cap \overline{G}{}^i$ and $\overline{H}{}^{[i]}:= \overline{H}{}^i/\overline{H}{}^{i+1}$ and identify the latter with a subgroup of $\PGL_2(k)$ or $\pgll_2(\pp^i/\pp^{i+1})$, respectively. The projection $\pi: \SL_2(R)\twoheadrightarrow\PGL_2(R)$ induces homomorphisms $H^i\to\overline{H}{}^i$ and $H^{[i]}\to\overline{H}{}^{[i]}$. 
Consider the subgroup $H^{2-} := H\cap G^{\prime2-}$. Then by construction the natural homomorphisms $H^{2-} \to \overline{H}{}^2 \to \overline{H}{}^{[2]}$ are surjective.

\begin{lem}\label{brennan}
We have $\overline{H}{}^{[2]}=\pgll_2(\pp^2/\pp^3)$.
\end{lem}

\begin{proof}
By construction $H^{2-}$ can be described equivalently as the inverse image of $\mathfrak{c}(\pp/\pp^2) \subset \sll_2(\pp/\pp^2)$ in $H^1$. Thus Lemma \ref{lemma12.6} implies that $H^{2-}$ surjects to $\mathfrak{c}(\pp/\pp^2)$. 
Thus Lemma \ref{tempe} implies that $\overline{H}{}^{[2]}$ surjects to $\mathfrak{c}^*(\pp^2/\pp^3)$. In particular $\overline{H}{}^{[2]}$ is non-zero.

Furthermore, since the embedding $\overline{H}{}^{[2]} \hookrightarrow \pgll_2(\pp^2/\pp^3)$ is equivariant under conjugation by $H$, its image is invariant under conjugation by $H^{[0]}=\SL_2(k)$. Thus  Corollary \ref{Egon2.1_cor} implies that $\psll_2(\pp^2/\pp^3)\subset \overline{H}{}^{[2]}$. Combined with the surjection $\overline{H}{}^{[2]} \twoheadrightarrow \mathfrak{c}^*(\pp^2/\pp^3)$ this implies the desired equality.
\end{proof}

\begin{lem}\label{seeley}
We have $H^{[2]}=\sll_2(\pp^2/\pp^3).$
\end{lem}

\begin{proof}
Lemma \ref{brennan} says that $H^{2-}$ surjects to $\pgll_2(\pp^2/\pp^3)$. Combined with Lemma \ref{tempe} this implies that $H^2$ surjects to $\psll_2(\pp^2/\pp^3)$. In particular $H^{[2]}$ contains a non-scalar matrix. Being invariant under $H^{[0]}=\SL_2(k)$ and contained in $\sll_2(\pp^2/\pp^3)$, by Proposition \ref{Egon2.1} it is therefore equal to $\sll_2(\pp^2/\pp^3)$.
\end{proof}

\begin{lem}\label{lemma12.13}
We have $H^{[i]}=\sll_2(\pp^i/\pp^{i+1})$ for all $i\geq1$. 
\end{lem}

\begin{proof}
By Lemmas \ref{lemma12.6} and \ref{seeley} we already know this assertion for $i=1,2$. Suppose that the assertion holds for some $i\geq1$. The commutator map $(g,g')\mapsto gg'g^{-1}g^{\prime-1}$ on $\GL_2(R)$ induces a map $\PGL_2(R) \times \SL_2(R) \to \SL_2(R)$. Direct calculation shows that this in turn induces a map $\overline{G}{}^2\times G^{\prime i}\to (G')^{i+2}$ and hence a map $\overline{G}{}^{[2]}\times G^{\prime[i]}\to G^{\prime[i+2]}$, which under the given isomorphisms is obtained from the Lie bracket by
\begin{eqnarray*}
[\ \:,\ ] :\ \pgll_2(\pp^2/\pp^3) \times \sll_2(\pp^i/\pp^{i+1}) 
        \!\!\!&\longrightarrow&\!\!\! \sll_2(\pp^{i+2}/\pp^{i+3}), \\
        \bigl((X\mathrel{\rm mod}\mathfrak{c}),Y\bigr) &\mapsto& XY-YX.
\end{eqnarray*}
Another direct calculation, or looking up \cite{PSTAX}, Proposition 1.2 (b), shows that the image of this pairing generates $\sll_2(\pp^{i+2}/\pp^{i+3})$ as an additive group. 

On the other hand, by construction the pairing sends the subset $\overline{H}{}^{[2]} \times H^{[i]}$ to the subgroup $H^{[i+2]}$. Since $\overline{H}{}^{[2]}=\pgll_2(\pp^2/\pp^3)$ by Lemma \ref{brennan}, the equality in the source thus implies equality in the target, and so the assertion holds for $i+2$. By separate induction over all even, resp.\ odd integers, the assertion follows for all $i\geq1$.
\end{proof}

Lemma \ref{lemma12.13} finishes the proof of Theorem \ref{strong_approx_SLn} in the remaining case $p=n=2$.
\end{proof}


\subsection{Trace criteria} \label{trace_crit}

In this subsection we show how the assumption in Theorem \ref{strong_approx_GLn} that $H^{[1]}$ contain a non-scalar matrix can be guaranteed using traces in the adjoint representation.

We keep the notations of Subsection \ref{succc}, assume that $|k|>9$, and consider a closed subgroup $H$ of $\GL_n(R)$, such that $\SL_n(k)\subset H^{[0]}$. In other words, the remaining assumptions in Theorem \ref{strong_approx_GLn} are met. 

Recall that $R$ is a complete valuation ring of equal characteristic. Thus the projection $R\twoheadrightarrow k$ possesses a unique splitting $k\hookrightarrow R$. Via this splitting we can view $\GL_n(k)$ as a subgroup of $\GL_n(R)$.
Let $G^{2-}$ denote the group of all matrices $g\in\GL_n(R)$ that are congruent to the identity modulo $\pp$ and congruent to a scalar modulo $\pp^2$. Then $G^2\subset G^{2-}\subset G^1$, and $G^{2-}/G^2 \cong \mathfrak{c}(\pp/\pp^2)$.

\begin{lem} \label{booth}
If $H^{[1]}$ contains only scalar matrices, then up to conjugation by an element of $\GL_n(R)$ we have $H\subset\GL_n(k)\cdot G^{2-}$.
\end{lem}

\begin{proof}
Consider the commutative diagram with exact rows
$$\xymatrix@R-7pt{
0\ar[r] & \pgll_n(\pp/\pp^2) \ar[r]^-{1+(\ )} \ar@{=}[d]^-\wr & 
\GL_n(R/\pp^2)/(1+\mathfrak{c}(\pp/\pp^2)) \ar[r] \ar@{=}[d]^-\wr & 
\GL_n(k) \ar[r] \ar@{=}[d]^-\wr & 1 \\
1\ar[r] & G^1/G^{2-} \ar[r] & 
G^0/G^{2-} \ar[r] & 
G^{[0]} \ar[r] & 1 \\
1\ar[r] & (H\cap G^1)/(H\cap G^{2-}) \ar[r]  \ar@{^{ (}->}[u] & 
H/(H\cap G^{2-}) \ar[r] \ar@{^{ (}->}[u] & 
H^{[0]} \ar[r] \ar@{^{ (}->}[u] & 1\rlap{.} \\}$$
The assumption $H^{[1]} \subset\mathfrak{c}(\pp/\pp^2)$ means that $H^1 = H\cap G^1\subset G^{2-}$. Thus the group on the lower left is trivial, and so the group in the lower middle defines a splitting $H^{[0]} \to G^0/G^{2-}$. We compare this splitting with the splitting induced by the inclusions $H^{[0]} \subset \GL_n(k) \subset \GL_n(R)$. These two splittings differ by a $1$-cocycle $H^{[0]} \to \pgll_n(\pp/\pp^2)$. But since $\SL_n(k)\subset H^{[0]} \subset \GL_n(k)$, Proposition \ref{cohom1} shows that this cocycle is a coboundary. This means that the splittings are conjugate by an element coming from $\pgll_n(\pp/\pp^2)$, and the lemma follows.
\end{proof}

Let $\Ad$ denote the adjoint representation of $\GL_n$, and let $\Tr\Ad(H)$ denote the subset $\{\Tr\Ad(h)\mid h\in H\} \subset R$. Recall that $H^{\rm der}$ denotes the closure of the derived group of $H$. 

\begin{thm} \label{trace_crit_1}
Assume that $|k|>9$. Let $H$ be a closed subgroup of $\GL_n(R)$ such that $\SL_n(k)\subset H^{[0]}$ and $\Tr\Ad(H)$ topologically generates the ring $R$. Then $H^{\rm der}=\SL_n(R)$.
\end{thm}

\begin{proof}
By Theorem \ref{strong_approx_GLn} it suffices to show that $H^{[1]}$ contains a non-scalar matrix. If that is not the case, by Lemma \ref{booth} we may assume that $H \subset \GL_n(k)\cdot\nobreak G^{2-}$. 
Consider any element $h\in H$. Then by the definition of $G^{2-}$, its image $\Ad(h)$ is the product of an element of $\PGL_n(k)$ with a matrix that is congruent to the identity modulo $\pp^2$. Its trace is therefore congruent to an element of $k$ modulo $\pp^2$; in other words we have $\Tr\Ad(h) \in k+\pp^2$. This contradicts the assumption on $\Tr\Ad(H)$.
\end{proof}

If $p=n=2$, the assumption on traces in Theorem \ref{trace_crit_1} may fail in interesting cases (compare Proposition \ref{Rtradsmall}), although the conclusion is satisfied. This is due to the fact that the representation of $\GL_2$ on $\psll_2$ in characteristic $2$ is isomorphic to the pullback under $\Frob_2$ of the standard representation twisted with the inverse of the determinant, which implies that $\Tr\Ad(g) = \Tr(g)^2\cdot\det(g)^{-1} + 2$ for every $g\in \GL_2(R)$. Thus if $\det(H)$ consists of squares, which happens in particular for $H=\SL_2(R)$, the subset $\Tr\Ad(H)$ is entirely contained in the subring $R^2 := \{x^2\mid x\in R\}$. The following result provides a suitable substitute in that case:

\begin{thm} \label{trace_crit_2}
Assume that $|k|>9$ and $p=2$. Let $H$ be a closed subgroup of $\GL_2(R)$ such that $\SL_2(k)\subset H^{[0]}$. Let $H'\subset H$ denote the intersection of all closed subgroups of index $2$, and assume that $\Tr\Ad(H')$ topologically generates the subring $R^2 := \{x^2\mid x\in \nobreak R\}$. 
Then $H^{\rm der}=\SL_2(R)$.
\end{thm}

\begin{proof}
Again by Theorem \ref{strong_approx_GLn} it suffices to show that $H^{[1]}$ contains a non-scalar matrix. If that is not the case, by Lemma \ref{booth} we may assume that $H \subset \GL_n(k)\cdot\nobreak G^{2-}$. Let $\varpi\in\pp$ be a uniformizer. Then every element of $H$ can be written in the form $h=\gamma\cdot g_2\cdot (1+\varpi x)$ with $\gamma\in \GL_2(k)$ and $g_2\in G^2$ and $x\in R$. 

\begin{lem} \label{gormogon}
There exists a homomorphism $f:H\to\pp^2/\pp^3$ satisfying $f(h)= \Tr(g_2-\Id_2)\mathrel{\rm mod}\pp^3$ for any element $h=\gamma\cdot g_2\cdot (1+\varpi x)$ of the above form.
\end{lem}

\begin{proof}
Consider another element $h'=\gamma'\cdot g'_2\cdot (1+\varpi x')\in H$ with $\gamma'\in \GL_2(k)$ and $g'_2\in G^2$ and $x'\in R$. 

To show that $f$ is well-defined, we must prove that $\Tr(g'_2-\Id_2) \equiv \Tr(g_2-\nobreak\Id_2)\allowbreak \mathrel{\rm mod}\pp^3$ whenever $h'=h$. But $h'=h$ implies that $\gamma'=\gamma$ and hence $g'_2 = g_2\cdot(1+\varpi y)$ for some $y\in R$. Therefore $\Tr(g'_2-\Id_2) = \Tr(g_2-\Id_2) + \Tr(g_2)\cdot\varpi y$. Since $g_2$ is congruent to the identity matrix modulo $\pp^2$, its trace is congruent to $2\mathrel{\rm mod}\pp^2$, i.e., congruent to $0\mathrel{\rm mod}\pp^2$. Thus $\Tr(g_2)\cdot\varpi y\in\pp^3$, and so the map is well-defined.

To show that $f$ is a homomorphism, observe that 
\begin{eqnarray*}
h'h &=& \gamma'\cdot g'_2\cdot (1+\varpi x') \cdot \gamma\cdot g_2\cdot (1+\varpi x) \\
&=& (\gamma'\gamma) \cdot (\gamma^{-1}g'_2\gamma\cdot g_2)\cdot (1+\varpi x') \cdot (1+\varpi x)
\end{eqnarray*}
with $\gamma'\gamma \in \GL_2(k)$ and $\gamma^{-1}g'_2\gamma\cdot g_2 \in G^2$ and $(1+\varpi x') \cdot (1+\varpi x)= (1+\varpi y)$ for some $y\in R$. Thus 
$f(h'h) = \Tr(\gamma^{-1}g'_2\gamma\cdot g_2-\Id_2)\mathrel{\rm mod}\pp^3$. Write this trace in the form
$$\Tr\bigl((\gamma^{-1}g'_2\gamma-\Id_2)(g_2-\Id_2)\bigr)
+ \Tr(\gamma^{-1}g'_2\gamma-\Id_2)
+ \Tr(g_2-\Id_2).$$
Here the first summand lies in $\pp^4$, because $\gamma^{-1}g'_2\gamma-\Id_2 \equiv g_2-\Id_2 \equiv 0$ modulo $\pp^2$, and the second summand is equal to $\Tr(g'_2-\Id_2)$, because the trace is invariant under conjugation. Therefore
$$\Tr(\gamma^{-1}g'_2\gamma\cdot g_2-\Id_2)
\ \equiv\ \Tr(g'_2-\Id_2)+ \Tr(g_2-\Id_2) \mod\pp^3,$$
and so $f(h'h) = f(h')+f(h)$, as desired.
\end{proof}

Since $\pp^2/\pp^3$ is an elementary abelian $2$-group, Lemma \ref{gormogon} implies that the restriction of $f$ to the subgroup $H'$ is trivial. In other words, for every element $h=\gamma\cdot g_2\cdot (1+\varpi x) \in H'$  with $\gamma\in \GL_2(k)$ and $g_2\in G^2$ and $x\in R$ we have $\Tr(g_2-\Id_2)\in\pp^3$. But for any such element we have
\begin{eqnarray*}
\Tr\Ad(h) &=& \Tr\Ad(\gamma g_2) \ =\  
\Tr(\gamma g_2)^2\cdot\det(\gamma g_2)^{-1} + 2 \\
&=& \Tr(\gamma g_2)^2\cdot\det(\gamma)^{-1}\cdot\det(g_2)^{-1} + 2.
\end{eqnarray*}
Here the matrix $\gamma g_2$ has coefficients in $k+\pp^2$; hence its trace lies in $k+\pp^2$, and the first factor lies in $k+\pp^4$. Since $\gamma\in\GL_2(k)$, the second factor lies in $k^\times$. Moreover, the fact that $g_2\equiv\Id_2\mathrel{\rm mod}\pp^2$ implies that $\det(g_2) \equiv 1+\Tr(g_2-\Id_2) \mathrel{\rm mod}\pp^2$. But we have just seen that $\Tr(g_2-\Id_2)\in\pp^3$, and so $\det(g_2)$ and hence the third factor lies in $1+\pp^3$. Together we find that $\Tr\Ad(h)$ lies in $k+\pp^3$. This contradicts the assumption on $\Tr\Ad(H')$, and so Theorem \ref{trace_crit_2} is proved.
\end{proof}


\newpage


\section{Preliminary results on Drinfeld modules} \label{known_results}

In this section we list some known results about Drinfeld modules or adapt them slightly, and create the setup on which the proof in Section \ref{sect_surjective} is based. From Subsection \ref{tatess} onwards we will restrict ourselves to the case of special characteristic. For the general theory of Drinfeld modules see Drinfeld \cite{DriI}, Deligne and Husem\"oller \cite{DeHu}, Hayes \cite{HayEx}, or Goss \cite{GossFFA}.


\subsection{Endomorphisms rings} \label{endos}

Let $\mathbb{F}_p$ denote the finite field of prime order~$p$. Let $F$ be a finitely generated field of transcendence degree $1$ over~$\mathbb{F}_p$. Let $A$ be the ring of elements of $F$ which are regular outside a fixed place $\infty$ of~$F$. 

Let $K$ be another finitely generated field over $\mathbb{F}_p$ of arbitrary transcendence degree. 
Then the endomorphism ring of the algebraic additive group ${\mathbb G}_{a,K}$ over $K$ is the non-commutative polynomial ring in one variable $K\{\tau\}$, where $\tau$ represents the endomorphism $u \mapsto u^p$ and satisfies the commutation relation $\tau u =u^p \tau$ for all $u\in K$. 
Consider a Drinfeld $A$-module 
$$\phi: A \rightarrow 
K\{\tau\},\, a\mapsto \phi_a$$ 
of rank $r \geq 1$ over~$K$. 
Let $\pp_0$ denote the characteristic of~$\phi$, that is, the kernel of the homomorphism $A \rightarrow K$ determined by the lowest coefficient of~$\phi$. This is a prime ideal of~$A$ and hence either $(0)$ or a maximal ideal, and $\phi$ is called of generic resp.\ of special characteristic accordingly.
By definition, the endomorphism ring of $\phi$ over $K$ is the centralizer 
$$\End_K(\phi) := \{u\in K\{\tau\}\mid\forall a\in A: \phi_a\circ u = u\circ\phi_a\}.$$
This is a finitely generated projective $A$-module, and $\End_K(\phi)\otimes_AF$ is a\hyphenation{finite} finite dimensional division algebra over~$F$. In special characteristic this algebra can be non-commutative. We often identify $A$ with its image under the homomorphism $A\to\End_K(\phi)$, ${a\mapsto\phi_a}$.

It may happen that $\phi$ possesses endomorphisms over an overfield that are not defined over~$K$. But by \cite{GossFFA}, Proposition 4.7.4, Remark 4.7.5, we have:

\begin{prop}\label{def_field_end}
There exists a finite separable extension $K'$ of $K$ such that for every overfield $L$ of $K'$ we have $\End_L(\phi) = \End_{K'}(\phi)$.
\end{prop}

Consider any integrally closed infinite subring $B\subset A$. Then $A$ is a finitely generated projective $B$-module of some rank $m\ge1$, and the restriction $\phi|B$ is a Drinfeld $B$-module of rank $rm$ over~$K$. By definition there is a natural inclusion $\End_K(\phi) \subset \End_K(\phi|B)$ identifying $\End_K(\phi)$ with the commutant of $A$ in $\End_K(\phi|B)$. In special characteristic it is possible that the latter is non-commutative and $A$ is not contained in its center, in which case the inclusion is proper.

Dually consider any commutative $A$-subalgebra $A'\subset \End_K(\phi)$. Then $A'$ is a finitely generated projective $A$-module of some rank $m'\ge1$. If $A'$ is normal, i.e., integrally closed in its quotient field, the tautological embedding $\phi': A'\rightarrow K\{\tau\}$ is a Drinfeld $A'$-module of rank $r/m'$ over~$K$; in particular $m'$ is then a divisor of~$r$. One can prove the same fact for arbitrary $A'$ using the isogeny provided by the following subsection.

\subsection{Isogenies} \label{isogs}

Let $\phi'$ be a second Drinfeld $A$-module over~$K$. Let $f$ be an isogeny $\phi\to\phi'$ over~$K$, that is, a non-zero element $f\in K\{\tau\}$ satisfying $f\circ\phi_a = \phi'_a\circ f$ for all $a\in A$. Then $f$ induces an isomorphism of $F$-algebras
\begin{myequation}\label{EndIsog}
\End_K(\phi)\otimes_AF \ \stackrel{\sim}{\longrightarrow}\ \End_K(\phi')\otimes_AF
\end{myequation}%
which sends $e\otimes1$ to $e'\otimes1$ if $f\circ e=e'\circ f$.

The following proposition extends a result of \cite{HayEx}, Proposition 3.2, to the possibly non-commutative case and is established in a different way.

\begin{prop}\label{ConstIsog}
Let $\phi\colon A\to K\{\tau\}$ be any Drinfeld module, let $S$ be any $A$-subalgebra of $\End_K(\phi)$ and let $S'$ be any maximal $A$-order in $S\otimes_AF$ which contains~$S$. Then there exist a Drinfeld $A$-module $\phi'\colon A\to K\{\tau\}$ and an isogeny $f\colon\phi\to\phi'$ over $K$ such that $S'$ corresponds to $\End_K(\phi') \cap (S\otimes_AF)$ via the isomorphism (\ref{EndIsog}).
\end{prop}

\begin{proof}
To avoid confusing endomorphisms of $\phi$ with endomorphisms of the desired $\phi'$ we denote the tautological embedding $S\hookrightarrow K\{\tau\}$ by $s\mapsto \phi_s$. 
Fix any non-zero element $a\in A$ satisfying $S'a\subset S$. Let $H_a$ denote the kernel of $\phi_a$ as a finite subgroup scheme of ${\mathbb G}_{a,K}$. Observe that the action of any endomorphism $s\in S$ on $H_a$ depends only on the residue class of $s$ modulo~$Sa$, and that $Sa$ has finite index in~$S'a$. Thus the sum 
$$H := \sum_{s\in S'a} \phi_s(H_a)$$
is really finite and defines another finite subgroup scheme of ${\mathbb G}_{a,K}$. By construction $H$ is mapped to itself under $\phi_s$ for every $s\in S$. In particular it is therefore the scheme theoretic kernel of a non-zero element $f\in K\{\tau\}$. 
Also, for each $s\in S$ we have $f(\phi_s(H)) \subset f(H) = 0$; hence $f\circ\phi_s$ annihilates $H=\Ker(f)$, and thus we have $f\circ\phi_s = \phi'_s\circ f$ for a unique element $\phi'_s\in K\{\tau\}$. For any two elements $s_1$, $s_2\in S$ we have
\begin{myequation}\label{ConstIsog0}
\phi'_{s_1}\circ\phi'_{s_2}\circ f
\ =\ \phi'_{s_1}\circ f\circ\phi_{s_2}
\ =\ f\circ\phi_{s_1}\circ\phi_{s_2}
\ =\ f\circ\phi_{s_1s_2}
\ =\ \phi'_{s_1s_2}\circ f
\end{myequation}%
and therefore $\phi'_{s_1}\circ\phi'_{s_2} = \phi'_{s_1s_2}$, and a similar calculation shows that $\phi'_{s_1}+\phi'_{s_2} = \phi'_{s_1+s_2}$. The resulting map $S\to K\{\tau\}$, $s\mapsto \phi'_s$ is thus a ring homomorphism. In particular its restriction to $A$ is a Drinfeld $A$-module $\phi'$ such that $f$ defines an isogeny $\phi\to\nobreak\phi'$, and the full map $s\mapsto \phi'_s$ defines an embedding $S\hookrightarrow \End_K(\phi')$ compatible with the isomorphism (\ref{EndIsog}). To extend this map to the maximal order~$S'$ we need the following preparation:

\begin{lem}\label{ConstIsog1}
Let $H_{a^2} \subset {\mathbb G}_{a,K}$ denote the kernel of $\phi_{a^2}$. Then
$$\sum_{s\in S'a} \phi_s(H_{a^2}) \ =\ \Ker(f\circ\phi_a).$$
\end{lem}

\begin{proof}
The summand for $s=a$ on the left hand side is $\phi_a(H_{a^2}) = H_a = \Ker(\phi_a)$ and therefore also contained in the right hand side. Thus it suffices to prove that the images of both sides under $\phi_a$ coincide. But
\begin{eqnarray*}
\phi_a\biggl( \sum_{s\in S'a} \phi_s(H_{a^2}) \biggr)
&=& \sum_{s\in S'a} \phi_a(\phi_s(H_{a^2}))
\ =\ \sum_{s\in S'a} \phi_s(\phi_a(H_{a^2}))
\ =\ \sum_{s\in S'a} \phi_s(H_a) \\
&\stackrel{\rm def}{=}& H
\ =\ \Ker(f)
\ =\ \phi_a(\Ker(f\circ\phi_a)),
\end{eqnarray*}
as desired.
\end{proof}

Now consider any $s'\in S'$, and observe that we have already constructed $\phi'_a$ and $\phi'_{s'a}$ in $K\{\tau\}$.

\begin{lem}\label{ConstIsog2}
There exists an element $\phi'_{s'}\in \End_K(\phi')$ which satisfies
$\phi'_{s'}\circ\phi'_a = \phi'_{s'a}$.
\end{lem}

\begin{proof}
For each $s\in S'a$ we have $s's$, $s'as\in S'a$; hence $\phi_{s'a}$ and $\phi_{s's}$ and $\phi_{s'as}$ all exist and satisfy $\phi_{s'a}\circ\nobreak\phi_s\allowbreak = \phi_{s'as} = \phi_{s's}\circ\phi_a$. Also, we have $f(\phi_{s's}(H_a)) \subset f(H) = 0$ and so
$$(f\circ\phi_{s'a})(\phi_s(H_{a^2}))
\ =\ (f\circ\phi_{s's})(\phi_a(H_{a^2}))
\ =\ (f\circ\phi_{s's})(H_a) 
\ =\ 0.$$
Summing over all $s\in S'a$ and using Lemma \ref{ConstIsog1} we deduce that
$f\circ\phi_{s'a}$ annihilates $\Ker(f\circ\phi_a)$. Thus there exists a unique element $\phi'_{s'}\in K\{\tau\}$ satisfying 
$f\circ\phi_{s'a} = \phi'_{s'}\circ f\circ\phi_a$.
The calculation 
$$\phi'_{s'a}\circ f 
\ =\ f\circ\phi_{s'a}
\ =\ \phi'_{s'}\circ f\circ\phi_a
\ =\ \phi'_{s'}\circ \phi'_a\circ f$$
now implies that $\phi'_{s'a} = \phi'_{s'}\circ \phi'_a$. Finally, a calculation like that in (\ref{ConstIsog0}) shows that $\phi'_{s'}\circ\phi'_b = \phi'_b\circ\phi'_{s'}$ for all $b\in A$.
Thus $\phi'_{s'}\in\End_K(\phi')$, as desired.
\end{proof}

By a calculation as in (\ref{ConstIsog0}) one easily shows that the map $S'\to \End_K(\phi')$, $s'\mapsto \phi'_{s'}$ is a ring homomorphism extending the previous one on~$S$. By construction it factors through a homomorphism $S'\to\End_K(\phi') \cap (S\otimes_AF)$ which becomes an isomorphism after tensoring with $F$ over~$A$. Since both sides of the latter are finitely generated torsion free $A$-modules, that homomorphism must be an inclusion of finite index. But as $S'$ is already a maximal $A$-order in $S\otimes_AF$, it follows that 
$S'\to\End_K(\phi') \cap (S\otimes_AF)$ is an isomorphism. This finishes the proof of the proposition.
\end{proof}

\begin{prop}\label{numerics}
Let $\phi\colon A\to K\{\tau\}$ be a Drinfeld module of rank~$r$. Let $d^2$ be the dimension of $\End_K(\phi)\otimes_A F$ over its center~$Z$, and let $e$ denote the dimension of $Z$ over~$F$.
Then $de$ divides~$r$.
\end{prop}

\begin{proof}
Set $R := \End_K(\phi)$ and let $F'$ be any maximal commutative $F$-subalgebra of $R\otimes_A F$. Let $A'$ denote the integral closure of $A$ in~$F'$. Then by construction we have $\mathop{\rm rank}\nolimits_A(A') = [F'/F] = de$. Applying Proposition \ref{ConstIsog} to $S:= A'\cap R$ and $S':=A'$ yields a Drinfeld $A'$-module $\phi'\colon A'\to K\{\tau\}$ and an isogeny $f\colon\phi\to\phi'|A$. Then $\phi'|A$ has rank $r$, and the remarks at the end of Subsection \ref{endos} imply that $\phi'$ has rank $r/de$. Thus this quotient is an integer, as desired.
\end{proof}


\subsection{Tate modules} \label{tatess}

From now on we assume that $\phi$ has special characteristic. We abbreviate $R:=\End_K(\phi)$ and assume that $A$ is the center of~$R$. By Proposition \ref{numerics} we then have $\dim_F(R\otimes_A F) = d^2$ for some factorization in integers $r=nd$.

\medskip
Let $K^{\text{sep}}$ denote the separable closure of $K$ inside a fixed algebraic closure $\overline{K}$ of $K$. Let $\kappa$ denote the finite constant field of $K$ and $\overline{\kappa}$ its algebraic closure in $K^\text{sep}$. Then $G_K:=\Gal(K^{\text{sep}}/K)$ is the absolute Galois group and $G_K^{\geom}:=\Gal(K^{\text{sep}}/K \overline{\kappa})$ the geometric Galois group of~$K$. Moreover, the quotient $G_K/G_K^{\geom} \cong \Gal(\overline{\kappa}/\kappa)$ is the free pro-cyclic group topologically generated by the element $\Frob_\kappa$, which acts on $\overline{\kappa}$ by $u \mapsto u^{|\kappa|}$. 

By a prime $\pp$ of $A$ we will mean any maximal ideal of $A$. The $\pp$-adic completions of $A$ and $F$ are denoted $A_\pp$ and $F_\pp$, respectively. For any prime $\pp\not=\pp_0$ of $A$ and any positive integer $i$, the $\pp^i$-torsion points of $\phi$ 
$$\phi[\pp^i]:=\{x \in K^{\text{sep}} \mid \forall a\in \pp^i : \phi_a(x)=0\}$$ 
form a free $A/\pp^i$-module of rank $r$. The $\pp$-adic Tate module $T_\pp (\phi) := \varprojlim \phi[\pp^i]$ is therefore a free $A_\pp$-module of rank $r$. Choosing a basis, the natural action of the Galois group $G_K$ on $T_\pp (\phi)$ is described by a continuous homomorphism
$$\rho_\pp : G_K\longrightarrow \GL_r(A_\pp).$$

The action of endomorphisms turns $T_\pp(\phi)$ into a module over $R_\pp :=R\otimes_A\nobreak A_\pp$.
Let $D_\pp$ denote the commutant of $R_\pp$ in $\End_{A_\pp}(T_\pp(\phi))$.
Since the action of $G_K$ commutes with that of $R_\pp$, the homomorphism $\rho_\pp$ factors through the multiplicative group $D_\pp^\times$ of $D_\pp$. We can thus view $\rho_\pp$ as a homomorphism $G_K\to D_\pp^\times$. The associated adelic Galois representation then becomes a homomorphism
$$\rho_{\text{ad}} := (\rho_\pp)_\pp:\ 
  G_K \longrightarrow \prod_{\pp\not=\pp_0}  D_\pp^\times
    \ \subset\        \prod_{\pp\not=\pp_0}\GL_r(A_\pp).$$
Let $V_\pp(\phi) := T_\pp(\phi)\otimes_{A_\pp}F_\pp$ denote the rational Tate module of $\phi$ at $\pp$. Then by construction $D_\pp\otimes_{A_\pp}F_\pp$ is the commutant of $R\otimes_AF_\pp$ in $\End_{F_\pp}(V_\pp(\phi))$.


\medskip
For the next technical results we choose a maximal commutative $F$-subalgebra $F'\subset R\otimes_A F$, let $A'$ denote the integral closure of $A$ in~$F'$, and choose a Drinfeld $A'$-module $\phi'\colon A'\to K\{\tau\}$ and an isogeny $f\colon\phi\to\phi'|A$, as in the proof of Proposition \ref{numerics}. Then $\phi'$ has rank~$n$ and its characteristic $\pp'_0$ is a prime of $A'$ above~$\pp_0$. For any prime $\pp\neq\pp_0$ of $A$ the isogeny $f$ induces a $G_K$-equivariant isomorphism
\begin{myequation}\label{eq:tensisom}
V_\pp(\phi) \ \cong\ V_\pp(\phi'|A) \ \cong\ \prod_{\pp'|\pp}V_{\pp'}(\phi').
\end{myequation}%

\begin{prop}\label{boalo}
For any prime $\pp\not=\pp_0$ of $A$ and any prime $\pp'$ of $A'$ above $\pp$ we have:
\begin{enumerate}
\item[(a)] $D_\pp\otimes_{A_\pp}F_\pp$ is a central simple algebra of dimension $n^2$ over~$F_\pp$.
\item[(b)] There is a natural isomorphism
$D_\pp\otimes_{A_\pp}F'_{\pp'} \stackrel{\sim}{\longrightarrow}
\End_{F'_{\pp'}}(V_{\pp'}(\phi'))$.
\item[(c)] The action of $G_K$ on $V_{\pp'}(\phi')$ is induced by the homomorphism $\rho_\pp\!: {G_K\to D_\pp^\times}$ and the isomorphism (b).
\end{enumerate}
\end{prop}

\begin{proof}
By construction $R_\pp\otimes_{A_\pp}F_\pp$
is a central simple algebra of dimension $d^2$ over~$F_\pp$, and $D_\pp\otimes_{A_\pp}F_\pp$ is its commutant in the action on the $R_\pp\otimes_{A_\pp}F_\pp$-module $V_\pp(\phi)$ of dimension $r=nd$ over~$F_\pp$. With general facts on semisimple algebras this implies~(a).
Next the isomorphism (\ref{eq:tensisom}) is really the isotypic decomposition of $V_\pp(\phi)$ over $A'\otimes_AF_\pp \cong \prod_{\pp'|\pp}F'_{\pp'}$. Since the action of $A'\otimes_AF_\pp \subset R\otimes_AF_\pp$ commutes with $D_\pp$, the decomposition is $D_\pp$-invariant. Thus each $V_{\pp'}(\phi')$ is a $D_\pp$-module. The actions of both $D_\pp$ and $F'_{\pp'}$ agree on~$A_\pp$; hence they induce a non-zero homomorphism 
$$D_\pp\otimes_{A_\pp}F'_{\pp'} \longrightarrow
\End_{F'_{\pp'}}(V_{\pp'}(\phi')).$$
Here by (a) the left hand side is a central simple algebra of dimension $n^2$ over~$F'_{\pp'}$.  But since $V_{\pp'}(\phi')$ has dimension $n$ over $F'_{\pp'}$, the same is true for the right hand side as well. Thus the homomorphism must be an isomorphism, proving (b). Finally, the natural construction implies (c).
\end{proof}

For any prime $\pp\not=\pp_0$ of $A$, let $D^1_\pp$ denote the subgroup of all elements of $D_\pp^\times$ whose reduced norm over $F_\pp$ is~$1$.

\begin{prop} \label{Ggeom_in_SL} 
There exists a finite extension $K'\subset K^{\text{sep}}$ of $K$ such that 
$$\rho_{\text{ad}}(G_{K'}^{\geom}) \subset \prod_{\pp\not=\pp_0} D_\pp^1.$$
\end{prop}

\begin{proof}
Let $\phi'\colon A'\to K\{\tau\}$ be as above. By Anderson \cite{Anderson_t_Motives}, \S4.2, the determinant of the adelic Galois representation associated to $\phi'$ is the adelic Galois representation associated to some Drinfeld $A'$-module of rank $1$ of special characteristic~$\pp_0'$. By Proposition \ref{finite} below the image of $G_K^{\geom}$ in that representation is finite. Choose a finite extension $K'\subset K^{\text{sep}}$ of $K$ such that $G_{K'}^{\geom}$ lies in its kernel. Then for any prime $\pp\not=\pp_0$ of $A$ and any prime $\pp'$ of $A'$ above $\pp$, Proposition \ref{boalo} (b) and (c) implies that $\rho_\pp(G_{K'}^{\geom}) \subset D^1_\pp$, as desired.
\end{proof}

\begin{prop}\label{labacu}
For almost all primes $\pp\not=\pp_0$ of $A$, we have $D_\pp \cong {\rm Mat}_{n\times n}(A_\pp)$ and $D_\pp^\times \cong \GL_n(A_\pp)$ and $D_\pp^1 \cong \SL_n(A_\pp)$.
\end{prop}

\begin{proof}
For almost all $\pp$, the central simple algebra $R\otimes_AF_\pp$ is split and $R_\pp = R\otimes_AA_\pp$ is a maximal order therein. For these $\pp$ we have $R_\pp \cong {\rm Mat}_{d\times d}(A_\pp)$, and $T_\pp(\phi)$ is a direct sum of $n$ copies of the tautological representation $A_\pp^{d}$. Its commutant $D_\pp$ is then isomorphic to ${\rm Mat}_{n\times n}(A_\pp)$, and everything follows.
\end{proof}

\subsection{Non-singular model}\label{model} 

Any integral scheme of finite type over $\mathbb{F}_p$ with function field $K$ is called a \emph{model of $K$}. By de Jong's theorem on alterations \cite[Th.$\;$4.1]{deJong1996} we have:

\begin{thm}\label{deJong}
There exists a finite separable extension $K'$ of $K$ which possesses a smooth projective model.
\end{thm}

In the following we assume that $\overline{X}$ is a smooth projective model of~$K$. Then we have:

\begin{thm}\label{Pop}
For any finite group $H$ there exist only finitely many continuous homomorphisms $G_K\to H$ which are unramified at all points of codimension $1$ of $\overline{X}$.
\end{thm}

\begin{proof}
By the Zariski-Nagata purity theorem of the branch locus \cite{ZariskiPurity}, \cite{NagataPurity}, any such extension comes from a finite \'etale covering of $\overline{X}$. In other words it factors through the \'etale fundamental group $\pi_1^{\rm et}(\,\overline{X}\,)$. This group lies in a short exact sequence
$$\xymatrix{
1 \ar[r] & \pi_1^{\rm et}(\,\overline{X}_{\overline{\kappa}}) \ar[r] &
\pi_1^{\rm et}(\,\overline{X}\,) \ar[r] &
\Gal(\overline{\kappa}/\kappa)  \ar[r] & 1,\\}$$
where $\pi_1^{\rm et}(\,\overline{X}_{\overline{\kappa}}\,)$ is topologically finitely generated by Grothendieck \cite[Exp.$\,$X, Th.$\,$2.9]{SGA1}, and $\Gal(\overline{\kappa}/\kappa)$ is the free pro-cyclic group topologically generated by Frobenius. Thus $\pi_1^{\rm et}(\,\overline{X}\,)$ is topologically finitely generated and so possesses only finitely many continuous homomorphisms to $H$, as desired.
\end{proof}

We choose an open dense subscheme $X\subset\overline{X}$ such that $\phi$ extends to a family of Drinfeld $A$-modules of rank $r$ over $X$. Since $\phi$ has special characteristic $\pp_0$, the extended family has characteristic $\pp_0$ everywhere. For any $\pp\not=\pp_0$, the action of $G_K$ on $T_\pp(\phi)$ factors through the \'etale fundamental group $\pi_1^{\rm et}(X)$. In particular it is unramified at all points of codimension $1$ in~$X$.

In the next three subsections we look separately at information coming from points in $X$, respectively in $\overline{X}\smallsetminus X$.


\subsection{Frobenius action}\label{chapter_frob} 

Consider any closed point $x\in X$ with finite residue field $\kappa_x$. By a Frobenius element $\Frob_x \in G_K$ we mean any element whose image in $\pi_1^{\rm et}(X)$ lies in a decomposition group above $x$ and acts by $u \mapsto u^{|\kappa_x|}$ on an algebraic closure of~$\kappa_x$. The action of $\Frob_x$ on $T_\pp(\phi)$ corresponds to the action on the Tate module $T_\pp(\phi_x)$, where $\phi_x$ denotes the reduction of $\phi$ at $x$.

Let $\pp$ be any prime of $A$ for which Proposition \ref{labacu} holds. Then $\rho_\pp(\Frob_x)\in D_\pp^\times \cong\GL_n(A_\pp)$, and we can consider its characteristic polynomial
\begin{myequation}\label{eq:charpol}
f_x(T) := \det\bigl(T\cdot{\rm Id}_n-\rho_\pp(\Frob_x)\bigr) \in A_\pp[T].
\end{myequation}%

\begin{prop}\label{F_charpol}
The polynomial $f_x$ has coefficients in $A$ and is independent of $\pp$. 
\end{prop}

\begin{proof}
Let $F'$ and $\phi'\colon A'\to K\{\tau\}$ be as in Subsection \ref{tatess}. Then Proposition \ref{boalo} shows that, for every $\pp$ as above and every $\pp'|\pp$, the image of $f_x(T)$ in $F'_{\pp'}[T]$ is the characteristic polynomial of the image of $\Frob_x$ in its representation on $V_{\pp'}(\phi')$. Applying \cite{GossFFA}, Theorem 4.12.12 (b), to the Drinfeld $A'$-module $\phi'$ shows that this image has coefficients in $F'$ and is independent of $\pp'$. Fixing $\pp$ and varying $\pp'|\pp$ it follows that the coefficients of $f_x(T)$ lie in $\diag(F') \subset \prod_{\pp'|\pp}F'_{\pp'}$, in other words, in the subring $A'\otimes_AF \subset A'\otimes_AF_\pp$. But by definition they also lie in the subring $A_\pp \cong A\otimes_AA_\pp$, whose intersection with the former is just $A$. Varying both $\pp$ and $\pp'$ then shows that $f_x(T)$ is independent of $\pp$.
\end{proof}

%
%

\begin{prop}\label{F_eigenvalues}
Let $\alpha_1, \ldots, \alpha_n$ be the roots of $f_x$ in an algebraic closure $\overline{F}$ of $F$, with repetitions if necessary. Consider any normalized valuation $v$ of $F$ and an extension $\overline{v}$ of $v$ to $\overline{F}$. Let $k_v$ denote the residue field at $v$. 
\begin{enumerate}
\item[(a)] If $v$ does not correspond to $\pp_0$ or $\infty$, then for all $1 \leq i \leq n$ we have 
$$\overline{v}(\alpha_i)=0.$$
\item[(b)] If $v$ corresponds to $\infty$, then for all $1 \leq i \leq n$ we have 
$$\overline{v}(\alpha_i)=-\frac{1}{nd}\cdot \frac{[\kappa_x/\mathbb{F}_p]}{[k_v/\mathbb{F}_p]}.$$ 
\item[(c)] If $v$ corresponds to $\pp_0$, then there exists an integer $1\leq n_x \leq n$ such that 
$$\overline{v}(\alpha_i)=\left\{ \begin{array}{ll}
 \frac{1}{n_x d}\cdot \frac{[\kappa_x/\mathbb{F}_p]}{[k_v/\mathbb{F}_p]} & \text{ for precisely } n_x \text{ of the } \alpha_i, \text{ and} \\[7pt]
 0 & \text{ for the remaining } n-n_x \text{ of the } \alpha_i.
\end{array}\right.$$ 
\end{enumerate}
\end{prop}

\begin{proof}
By construction the $\alpha_i$ are the roots of the characteristic polynomial of $\rho_\pp(\Frob_x)$ associated to the Drinfeld module $\phi$ of rank $r=nd$, except that their multiplicities are divided by $d$. Thus the proposition is a direct consequence of \cite{DriII}, Proposition 2.1, to~$\phi$. 
\end{proof}


\subsection{Good reduction and lattices} \label{lettuce} 

In this subsection we briefly leave the current setting and consider the following general situation. 

Let $L$ be a field containing $\mathbb{F}_p$ with a non-trivial discrete valuation $v$. Let $R\subset L$ denote the associated discrete valuation ring and ${\mathfrak m}$ its maximal ideal. Let $\psi:A\to R\{\tau\}$, $a\mapsto\psi_a$ be a Drinfeld $A$-module of rank $s>0$ with good reduction, i.e., such that for every $a\in A\smallsetminus\{0\}$ the highest non-zero coefficient of $\psi_a$ is a unit in $R$.
We view $L$ as an $A$-module with respect to the action $a\cdot u :=\psi_a(u)$ for all $a\in A$ and $u\in L$. Then $R$ is a submodule for this action, and we are interested in the structure of the $A$-module $L/R$. 

To any $A$-module $M$ are associated the following notions. The \emph{rank of $M$} is the maximal number of $A$-linearly independent elements of~$M$, or $\infty$ if the maximum does not exist. Of course, any finitely generated $A$-module has finite rank. Next, the \emph{division hull} of an $A$-submodule $N\subset M$ is defined as
\begin{myequation}\label{DivHull}
\sqrt{N}\ :=\ \bigl\{ \overline{u}\in M \bigm| 
\exists\, a\in A\smallsetminus\{0\}: a\cdot\overline{u}\in N\bigr\},
\end{myequation}%
which is an $A$-module of the same rank as $N$. The $A$-module $M$ is called \emph{tame} if every finitely generated $A$-submodule $N\subset M$ satisfies $[\sqrt{N}:N] < \infty$.

The following result was obtained by Poonen in \cite[Lemma$\,$5]{PoonenMordellWeil} when $L$ is a global field and $\psi$ has generic characteristic, and by Wang \cite{WangMordellWeil} in general. 

\begin{prop}\label{lettprop1}
$L/R$ is a tame $A$-module.
\end{prop}

\subsection{Bad reduction} \label{bad_red} 

Now we return to the situation and the notations of Subsections \ref{endos} through \ref{chapter_frob}. We assume in addition that there exists a prime $\pp\not=\pp_0$ of $A$ such that all $\pp$-torsion points of $\phi$ are defined over~$K$. This can be achieved on replacing $K$ by the finite separable extension corresponding to the action of $G_K$ on $\phi[\pp]$.

Let $x$ be one of the finitely many generic points of $\overline{X}\smallsetminus X$. Let $K_x$ denote the completion of $K$ with respect to the valuation at $x$, and let $R_x\subset K_x$ denote the associated discrete valuation ring. Since $\phi$ possesses a full level structure of some level $\pp\not=\pp_0$ over~$K$, it is known to have semistable reduction over $K_x$. Its Tate uniformization at $x$ (see \cite{DriI}, \S 7) then consists of a Drinfeld $A$-module $\psi_x$ over $R_x$ of some rank $1\leq r_x\leq r$ with good reduction and an $A$-lattice $\Lambda_x \subset K_x^\text{sep}$ of rank $r-r_x$ for the action of $A$ on $K_x^\text{sep}$ via $\psi_x$. Here by definition an $A$-lattice is a finitely generated projective $A$-submodule whose intersection with any ball of finite radius is finite. This implies that any non-zero element of $\Lambda_x$ has valuation $<0$. Also, being finitely generated, the lattice is already contained in some finite Galois extension $K_x'$ of $K_x$. 

Let $I_x\subset D_x\subset G_K$ denote the inertia group, respectively the decomposition group, at a fixed place of $K^{\text{sep}}$ above $x$. Then $D_x$ is also the absolute Galois group of $K_x$. Let $D'_x\triangleleft D_x$ denote the absolute Galois group of $K'_x$, and set $I'_x := I_x\cap D'_x$. Then $D_x$ acts on $\Lambda_x$ through the finite quotient $D_x/D'_x$. 

For any prime $\pp\not=\pp_0$ of $A$ and any positive integer $i$ the Tate uniformization yields a $D_x$-equivariant isomorphism 
\begin{myequation}\label{TateSeq0}
\phi[\pp^i]\ \cong\ 
\bigl\{ u \in K_x^{\text{sep}} \bigm| 
        \forall a\in \pp^i : \psi_{x,a}(u)\in\Lambda_x \bigr\} \bigm/ \Lambda_x
\end{myequation}%
and hence a $D_x$-equivariant short exact sequence
$$0 \longrightarrow \psi_x[\pp^i] \longrightarrow \phi[\pp^i] \longrightarrow \Lambda_x\otimes_A(\pp^{-i}/A)\longrightarrow 0.$$
Taking the inverse limit over $i$ yields a $D_x$-equivariant short exact sequence
$$0 \longrightarrow T_\pp(\psi_x) \longrightarrow T_\pp(\phi) \longrightarrow \Lambda_x\otimes_AA_\pp \longrightarrow 0.$$
Here $I_x$ acts trivially on $T_\pp(\psi_x)$, and $D'_x$ acts trivially on $\Lambda_x\otimes_AA_\pp$. Thus in a suitable basis $\rho_\pp(D'_x)$ is contained in the group of block triangular matrices of the form
$$\left(\begin{array}{c|c} \ast & \ast \\ \hline 0 & 1 \\ \end{array}\right)
\ \subset\ \GL_r(A_\pp),$$
and $\rho_\pp(I'_x)$ is a $\rho_\pp(D_x)$-invariant subgroup of the group of block triangular matrices of the form
\begin{myequation}\label{inertunip}
\left(\begin{array}{c|c} 1 & \ast \\ \hline 0 & 1 \\ \end{array}\right)
\ \cong\ \Hom_A(\Lambda_x,T_\pp(\psi_x))
\ \cong\ T_\pp(\psi_x)^{r-r_x}.
\end{myequation}%
We are interested in the following three consequences:

\begin{lem}\label{no_inert}
Fix an integer $c\geq1$. Then for almost all primes $\pp\not=\pp_0$ of $A$, any continuous homomorphism from $\rho_\pp(D_x)$ to a finite group of order $\leq c$ is trivial on $\rho_\pp(I'_x)$.
\end{lem}

\begin{proof}
Fix a Drinfeld $A$-module $\psi_y$ of rank $r_x$ over a finite field that arises by good reduction from $\psi_x$. Let $\Frob_y$ be an associated Frobenius element in $D_x/I_x$, the absolute Galois group of the residue field at $x$. Then by \cite{GossFFA}, Theorem 4.12.12 (b), the characteristic polynomial of $\Frob_y$ on $T_\pp(\psi_x)$ has coefficients in $A$ and is independent of $\pp$. Moreover, \cite{DriII}, Proposition 2.1, implies that none of its eigenvalues $\beta_1,\ldots,\beta_{r_x}\in \overline{F}$ is a root of unity. Thus $a := \prod_{i=1}^{r_x}(\beta_i^{c!}-1)$ is a non-zero element of $A$. We claim that the assertion holds for all $\pp\not=\pp_0$ that do not divide $a$.

Indeed, let $f: \rho_\pp(D_x)\to H$ be a continuous homomorphism to a finite group of order $\leq c$, such that $f|\rho_\pp(I'_x)$ is non-trivial. Then $\mathop{\rm ker} f|\rho_\pp(I'_x)$ is a $\rho_\pp(D_x)$-invariant proper closed subgroup of $\rho_\pp(I'_x)$ of index $\leq c$. Thus $T_\pp(\psi_x)^{r-r_x}$ and hence $T_\pp(\psi_x)$, as a representation of $\rho_\pp(D_x)$, possesses a non-trivial finite subquotient of order $\leq c$. Then $\Frob_y^{c!}$ acts trivially on this subquotient. But this requires that some $\beta_i^{c!}$ is congruent to $1$ modulo a prime of $\overline{F}$ above $\pp$, or equivalently that $\pp|a$. This proves the claim.
\end{proof}

\begin{lem}\label{inert_big}
For almost all primes $\pp\not=\pp_0$ of $A$ we have $\psi_x[\pp] = \phi[\pp]^{I_x} = \phi[\pp]^{I'_x}$.
\end{lem}

\begin{proof}
The inclusions $\psi_x[\pp] \subset \phi[\pp]^{I_x} \subset \phi[\pp]^{I'_x}$ result from the fact that $I_x$ acts trivially on $\psi_x[\pp]$. To prove equality take any element of 
$\phi[\pp]^{I'_x}$. By (\ref{TateSeq0}) it corresponds to the residue class $u+\Lambda_x$ for some $u \in K_x^{\text{sep}}$ satisfying $\psi_{x,a}(u)\in\Lambda_x$ for all $a\in \pp$. That this residue class is $I'_x$-invariant means that $\sigma u-u \in \Lambda_x$ for all $\sigma\in I'_x$. But $\sigma\in I'_x$ acts trivially on $\psi_{x,a}(u)\in\Lambda_x$ for all $a\in\pp$; hence $\psi_{x,a}(\sigma u-u) = \sigma\psi_{x,a}(u) - \psi_{x,a}(u) = 0$. Since the homomorphism $\psi_{x,a}: \Lambda_x\to\Lambda_x$ is injective whenever $a\not=0$, it follows that $\sigma u-u=0$ and hence $u$ is $I'_x$-invariant.

Let $L$ denote the maximal unramified extension of $K'_x$, and $R\subset L$ its discrete valuation ring. As in Subsection \ref{lettuce} we denote the residue class in $L/R$ of an element $v\in L$ by $\overline{v}$ and abbreviate $a\cdot\overline{v} := \overline{\psi_{x,a}(v)}$ for all $a\in A$. Since every non-zero element of $\Lambda_x$ has valuation $< 0$, we have $\Lambda_x \cap R =\{0\}$ and thus the natural map $\Lambda_x\to L/R$ is injective; let $N_x$ denote its image.

The fact that $u$ is $I'_x$-invariant means that $u\in L$. On the other hand, the fact that $\psi_{x,a}(u)\in\Lambda_x$ for all $a\in \pp$ implies that $\pp\cdot\overline{u} \subset N_x$. In particular we have $\overline{u}\in\sqrt{N_x}$ in the notation of (\ref{DivHull}). But since $[\sqrt{N_x}:N_x] < \infty$ by Proposition \ref{lettprop1}, for almost all $\pp$ we can deduce that $\overline{u}\in N_x$. Then $u=v+\lambda$ for some $v\in R$ and $\lambda\in\Lambda_x$. For all $a\in \pp$ we then have $\psi_{x,a}(v)\in\Lambda_x\cap R=\{0\}$; in other words $v\in\psi_x[\pp]$. Thus the residue class in question $u+\Lambda_x$ comes from an element of $\psi_x[\pp]$, as desired.
\end{proof}

\begin{lem}\label{small_quotients1}
For any finite abelian group $H$ there exists a finite set $P'$ of primes of $A$, such that the number of continuous homomorphisms $G_K\to H$, which are trivial on $\mathop{\rm ker}(\rho_\pp)$ for some $\pp\not\in P'$, is finite.
\end{lem}

\begin{proof}
For each of the finitely many generic points $x$ of $\overline{X}\smallsetminus X$, let $P_x$ denote the finite set of primes of $A$ excluded by Lemma \ref{no_inert} with $c:=|H|$. We claim that the assertion holds with $P'$ the union of these sets $P_x$. 

Indeed, let $f: G_K \to H$ be a continuous homomorphism which is trivial on $\mathop{\rm ker}(\rho_\pp)$ for some $\pp\not\in P'$. From Subsection \ref{model} we know that $\rho_\pp$ and hence $f$ factors through the \'etale fundamental group $\pi_1^{\rm et}(X)$. Also, the restriction $f|I'_x$ is trivial for every generic point $x$ of $\overline{X}\smallsetminus X$ by Lemma \ref{no_inert}. There are therefore only finitely many possibilities for the restriction $f|I_x$. Since there are only finitely many $x$, it suffices to prove that the number of such $f$ with fixed restrictions $f|I_x$ for all $x$ is finite.

But since $H$ is abelian, any two such homomorphisms $f$ differ by a continuous homomorphism $g: G_K \to H$ which is unramified over $X$ and at all generic points of 
$\overline{X}\smallsetminus X$. By Theorem \ref{Pop} there are only finitely many possibilities for such $g$, and the desired finiteness follows.
\end{proof}


\subsection{Setup}\label{setup}

From here on we assume that $\phi$ satisfies the conditions of Theorem \ref{main_theorem}. Since we are only interested in the image of Galois groups up to commensurability, we may replace $K$ by a finite extension. We first replace it by the composite of the extensions provided by Propositions \ref{def_field_end} and \ref{Ggeom_in_SL} and the fields of definition of all $\pp$-torsion points of $\phi$ for some chosen prime $\pp\not=\pp_0$ of $A$. Thereafter we replace it by the extension from Theorem \ref{deJong}. 
By Proposition \ref{finite} below the assumption on $\End_{K^{\text{sep}}}(\phi|B)=R$ in Theorem \ref{main_theorem} implies that $n\geq2$.
Thus altogether we have the following assumptions:

\begin{asses}\label{asses}
\begin{enumerate}
\item[(a)] $R := \End_K(\phi) = \End_{K^{\text{sep}}}(\phi)$. 
\item[(b)] The center of $R$ is $A$.
\item[(c)] $n:=r/d \geq2$.
\item[(d)] For every integrally closed infinite subring $B \subset A$ we have $\End_{K^{\text{sep}}}(\phi|B)=\nobreak R$. 
\item[(e)] $\rho_{\text{ad}}(G_K^{\geom}) \subset \prod_{\pp\not=\pp_0} D_\pp^1$.
\item[(f)] There exists a prime $\pp\not=\pp_0$ of $A$ such that all $\pp$-torsion points of $\phi$ are defined over~$K$. 
\item[(g)] $K$ possesses a smooth projective model $\overline{X}$.
 \end{enumerate}
\end{asses}



\subsection{Images of Galois groups}\label{known}

Throughout the following we let $P_0$ denote the finite set of primes excluded by Proposition \ref{labacu}. For any $\pp\not\in P_0$ we set
$$\begin{array}{lll}
\Gamma_\pp & :=\ \rho_\pp(G_K) 
& \!\!\!\subset\ \GL_n(A_\pp), \quad\hbox{and} \\[5pt]
\Gamma_\pp^{\geom}\! & :=\ \rho_\pp(G_K^{\geom}) 
& \!\!\!\subset\ \SL_n(A_\pp).
\end{array}$$ 
By construction the latter is a closed normal subgroup of the former and the quotient is pro-cyclic. Combining Proposition \ref{boalo} (c) with \cite{PinI}, Theorem 1.1, and applying \cite{PinIII}, Lemma 3.7, we obtain:

\begin{thm}\label{Zariski_dense}
For any $\pp$ as above $\Gamma_\pp$ is Zariski dense in $\GL_{n,F_\pp}$, and $\Gamma_\pp^{\geom}$ is Zariski dense in $\SL_{n,F_\pp}$.
\end{thm}

The next result concerns the image of the group ring. By \cite{PinIII}, Theorem B, in the case that $K$ has transcendence degree $1$, and by \cite{PR1}, Theorem 0.2, in the general case, we know:

\begin{thm}\label{split_at_p}
For almost all primes $\pp$  of $A$ we have $A_\pp[\Gamma_\pp] ={\rm Mat}_{n\times n}(A_\pp)$.
\end{thm}

Let $k_\pp := A/\pp$ denote the residue field of $\pp$, and let $\overline{\rho}_\pp: G_K\to \GL_n(k_\pp)$ denote the reduction of $\rho_\pp$ modulo $\pp$. 
Theorem \ref{split_at_p} immediately implies:

\begin{cor}\label{abs_irred}
For almost all primes $\pp$  of $A$ the representation $\overline{\rho}_\pp$  on $k_\pp^n$ is absolutely irreducible.
\end{cor}

\begin{thm} \label{openness} 
For any finite set $P$ of primes $\not=\pp_0$ of $A$, consider the combined representation
$$\rho_{P} := (\rho_\pp)_\pp:\ 
  G_K \longrightarrow \prod_{\pp\in P} D_\pp^\times 
      \ \subset\ \prod_{\pp\in P} \GL_r(A_\pp).$$
Then $\rho_{P}(G_K^{\geom})$ has finite index in $\prod_{\pp\in P} D^1_\pp$.
\end{thm}

\begin{proof}
Since $n\ge2$, Proposition \ref{finite} below shows that $\phi$ is not isomorphic over $K^{\text{sep}}$ to a Drinfeld module defined over a finite field. We may thus apply \cite{PinII}, Theorems 6.1 and 6.2. The subfield $E$ given there is contained in the center $F$ of $R\otimes_AF$, such that $B := E\cap A$ is infinite and $\End_{K^{\text{sep}}}(\phi|B)\otimes_BE$ has center $E$. But by Assumption \ref{asses} (d) we have $\End_{K^{\text{sep}}}(\phi|B)=R$ with center $A$. Thus we must have $E=F$.

The group $G_Q(E_Q)$ described in \cite{PinII}, Theorem 6.2, is then the centralizer of $R\otimes_A\prod_{\pp\in P}F_\pp$ in $\prod_{\pp\in P}\Aut_{F_\pp}(V_\pp(\phi))$. In our situation it is therefore equal to $\prod_{\pp\in P}(D_\pp\otimes_{A_\pp}F_\pp)^\times$. 
The subgroup $G_Q^{\rm der}(E_Q)$ is the subgroup of elements of reduced norm $1$.
Theorem 6.1 of \cite{PinII} says that $\rho_{P}(G_K^{\geom})$ is commensurable to an open subgroup of $G_Q^{\rm der}(E_Q)$. Since $\rho_{P}(G_K^{\geom})$ is already contained in $\prod_{\pp\in P} D^1_\pp$ by Assumption \ref{asses} (e), which is an open compact subgroup of $G_Q^{\rm der}(E_Q)$, the index must be finite, as desired.
\end{proof}


\subsection{Ring of traces} \label{ring_of_trances} 

Let $\Ad$ denote the adjoint representation of $\GL_n$. Proposition \ref{F_charpol} implies that the trace $\Tr\Ad(\rho_\pp(\Frob_x))$ lies in $F$ and is independent of $\pp$. We let $R^{\rm trad}$ denote the subring of $F$ generated by $\Tr\Ad(\rho_\pp(\Frob_x))$ for all closed points $x\in X$, and let $F^{\rm trad} \subset F$ denote the quotient field of $R^{\rm trad}$.


\begin{thm} \label{field_trad}
Either $F^{\rm trad} = F$, or $n=p=2$ and $F^{\rm trad}=F^2 := \{x^2\mid x\in F\}$.
\end{thm}

\begin{proof}
Applying \cite{PinII}, Theorem 1.2, to the Drinfeld $A'$-module $\phi'$ from Subsection \ref{tatess} yields a subfield $E\subset F'$, which by Assumption \ref{asses} (d) turns out to be $F$. (One may equivalently combine \cite{PinII}, Theorem 1.1, for $\phi'$ with Theorem \ref{openness} above.) Thus by \cite{PinII}, Theorem 1.3, the subfield generated by the traces of Frobeniuses in the adjoint representation associated to $\phi'$ has the desired properties. But by Proposition \ref{boalo}, those traces are just $\Tr\Ad(\rho_\pp(\Frob_x))$; hence this subfield is $F^{\rm trad}$.
\end{proof}

As the following proposition shows, the second case in Theorem \ref{field_trad} really does occur:

\begin{prop}\label{Rtradsmall} 
Let $\kappa'\subset\overline{\kappa}$ denote the extension of degree $2$ of the constant field $\kappa$. If $n=p=2$, then after replacing $K$ by $K\kappa'$, we have $F^{\rm trad}=F^2$.
\end{prop}

\begin{proof}
In characteristic $p=2$, let ${\rm \std}^{(2)}$ denote the pullback under $\Frob_2$ of the standard representation of $\GL_2$, and let $\det:\GL_2\to{\mathbb G}_m$ denote the determinant. Then the adjoint representation of $\GL_2$ is an extension of ${\rm \std}^{(2)} \otimes \det^{-1}$ with two copies of the trivial representation of dimension~$1$. Thus for every $g\in \GL_2$ we have $\Tr\Ad(g) = \Tr(g)^2\cdot\det(g)^{-1} + 2$.

Recall from Assumption \ref{asses} (e) that $\rho_\pp(G_K^{\geom}) \subset \SL_n(A_\pp)$. Thus $\det\circ\rho_\pp$ factors through a homomorphism $\Gal(\overline{\kappa}/\kappa) \to A_\pp^\times$. Its value on any element of $\Gal(\overline{\kappa}/\kappa')$ is therefore a square. After replacing $K$ by $K\kappa'$ we find that $\Tr\Ad(\rho_\pp(\Frob_x)) \in F\cap F_\pp^2 = F^2$ for every closed point $x\in X$. Thus now only the second case in Theorem \ref{field_trad} is possible.
\end{proof}

\begin{prop}\label{Rtradprop} 
Let $A_0$ be the ring of elements of $F$ which are regular outside $\pp_0$. Then either $R^{\rm trad}$ is a subring of finite index of $A_0$, or $n=p=2$ and $R^{\rm trad}$ is a subring of finite index in $A_0^2 := \{x^2\mid x\in A_0\}$.
\end{prop}

\begin{proof}
Let $\alpha_1, \ldots, \alpha_n \in \overline{F}$ denote the eigenvalues of $\rho_\pp(\Frob_x)$. By Proposition \ref{F_eigenvalues} they have valuation $0$ at all places not above $\pp_0$ or $\infty$, and the same negative valuation at any place above $\infty$. Thus their ratios $\alpha_i/\alpha_j$ have trivial valuation at all places not above $\pp_0$. The sum over all $i,j$ of these ratios is therefore regular at all places $\not=\pp_0$. This sum is just $\Tr\Ad(\rho_\pp(\Frob_x))$, proving that $R^{\rm trad} \subset A_0$. 

By Theorem \ref{field_trad} the ring $R^{\rm trad}$ must contain some non-constant element $x$. Then $F$ is a finite field extension of ${\mathbb F}_p(x)$. Moreover, $x$ as an element of $F$ is regular outside $\pp_0$, and therefore $\pp_0$ is the unique place of $F$ above the place of ${\mathbb F}_p(x)$ where $x$ has a pole. This implies that $A_0$ is the integral closure of ${\mathbb F}_p[x]$ in~$F$. It is therefore a module of finite type over ${\mathbb F}_p[x]$, and so $R^{\rm trad}$ is a submodule that is again of finite type. In particular, $R^{\rm trad}$ is already generated by finitely many traces.

Also, it follows that $R^{\rm trad}$ is of finite index in its normalization. Depending on the case in Theorem \ref{field_trad}, this normalization is either $A_0$ or $A_0^2$, and we are done.
\end{proof}

By construction any prime $\pp\not=\pp_0$ of $A$ corresponds to a unique prime of $A_0$. Thus there are natural homomorphisms $R^{\rm trad} \hookrightarrow A_0 \hookrightarrow A_\pp  \twoheadrightarrow k_\pp$.

\begin{prop}\label{residue_field_trad1} 
There exists a finite set $P^{\rm trad}$ of primes of $A$, containing $\pp_0$, such that:
\begin{enumerate}
\item[(a)] For any prime $\pp \not\in P^{\rm trad}$ of $A$, the homomorphism $R^{\rm trad} \to k_\pp$ is surjective.
\item[(b)] For any two disctinct primes $\pp_1,\pp_2 \not\in P^{\rm trad}$ of $A$, the homomorphism $R^{\rm trad} \to k_{\pp_1}\times k_{\pp_2}$ is surjective.
\item[(c)] For any prime $\pp \not\in P^{\rm trad}$ of $A$, the image of the homomorphism $R^{\rm trad} \to A_\pp$ is dense in $A_\pp$ if $F^{\rm trad} = F$, respectively dense in $A_\pp^2 := \{a^2\mid a\in A_\pp\}$ if $F^{\rm trad}=F^2$.
\end{enumerate}
\end{prop}

\begin{proof}
Depending on the case, Proposition \ref{Rtradprop} implies that the annihilator of $A_0/R^{\rm trad}$, respectively the annihilator of $A_0^2/R^{\rm trad}$, as an $R^{\rm trad}$-module contains a non-zero element $x\in R^{\rm trad}$. Then $R^{\rm trad}[x^{-1}]$ is equal to $A_0[x^{-1}]$, respectively to $A_0^2[x^{-1}]$. Let $P^{\rm trad}$ be the finite set of primes of $A$ consisting of $\pp_0$ and all those dividing $x$ within $A_0$. Then $P^{\rm trad}$ has all the desired properties.
\end{proof}


\newpage


\section{Proof of the main result}\label{sect_surjective}

In this section we prove Theorem \ref{main_theorem}. Subsections \ref{51} through \ref{54} deal with the image of the geometric Galois group $G_K^{\geom}$, while Subsection \ref{55} finishes with the image of the absolute Galois group~$G_K$.
We keep all the notations from the preceding section and impose Assumptions \ref{asses}.


\subsection{Residual surjectivity at a single prime}\label{51}

Recall that $P_0$ denotes the finite set of primes excluded by Proposition \ref{labacu}. For any prime $\pp\not\in P_0$ of $A$, we let $\Delta_\pp^{\text{geom}} \triangleleft \Delta_\mathfrak{p}\subset\GL_n(k_\pp)$ denote the images of $G_K^{\text{geom}} \triangleleft G_K$ under the residual representation $\overline{\rho}_\pp$. Thus
$$\begin{array}{lll}
\Gamma_\pp 
& \twoheadrightarrow\ \Delta_\pp 
& \!\!\!\subset\ \GL_n(k_\pp), 
\\[5pt]
\Gamma_\pp^{\text{geom}}\!\!\! 
& \twoheadrightarrow\ \Delta_\pp^{\text{geom}}
& \!\!\!\subset\ \SL_n(k_\pp),
\end{array}$$ 
and the quotient $\Delta_\pp/\Delta_\pp^{\text{geom}}$ is cyclic. We will prove that $\Delta_\pp^{{\geom}} = \SL_n(k_\pp)$ for almost all $\pp$. 

\begin{lem}\label{fgh_on_delta1}
Fix any integer $c\geq 1$, and let $f$ denote the morphism from (\ref{Rdef}). Then for almost all primes $\pp\not\in P_0$ of $A$, the map $\Delta_\pp \to k_\pp$, $\delta \mapsto f(\delta^c)$ is not identically zero.
\end{lem}

\begin{proof}
Take any prime $\pp\not\in P_0$ of $A$. Then by Theorem \ref{Zariski_dense} together with Lemma \ref{R_nontrivial}, the map $\Gamma_\pp \to F_\pp$, $\gamma\mapsto f(\gamma^c)$ is not identically zero. Since this map is continuous and the images of Frobenius elements are dense in $\Gamma_\pp$, we may fix a closed point $x\in X$ such that $a := f(\rho_\pp(\Frob_x)^c)\not=0$. By the definition of $f$, this value is a polynomial with coefficients in ${\mathbb Z}$ in the coefficients of the characteristic polynomial of $\rho_\pp(\Frob_x)^c$. With Proposition \ref{F_charpol} it follows that $a$ lies in $A$ and is independent of $\pp$. In other words, having found $x$ and $a\in A\smallsetminus\{0\}$ with the help of \emph{some} auxiliary prime $\pp\not\in P_0$, we then have $f(\rho_\pp(\Frob_x)^c) = a$ for \emph{every} prime $\pp\not\in P_0$.

Thus for $\delta := \overline{\rho}_\pp(\Frob_x) \in \Delta_\pp$ we now deduce that $f(\delta^c) = a\mathop{\rm mod}\pp$. This is non-zero whenever $\pp\nmid a$; hence the desired assertion holds whenever $\pp\not\in P_0$ and $\pp\nmid a$.
\end{proof}

Let $\bar k_\pp$ denote an algebraic closure of $k_\pp$, and set $W_\pp := \phi[\pp]\otimes_{k_\pp}\bar k_\pp$. By Corollary \ref{abs_irred} this is an irreducible representation of $\Delta_\pp$ over $\bar k_\pp$ for all $\pp$ outside some finite set of primes $P^{\rm irr}$. By Theorem \ref{L-P3} there then exists a normal subgroup $\Delta'_\pp \triangleleft\Delta_\pp$ of index $\leq c'_n$, such that $\Delta'_\pp/Z(\Delta'_\pp)$ is a direct product of finite simple groups of Lie type in characteristic~$p$. We fix such a subgroup $\Delta'_\pp$ for every $\pp\not\in P^{\rm irr}$.

\begin{lem}\label{Wdecomp}
For almost all primes $\pp\not\in P^{\rm irr}$ of $A$, we have $W_\pp = W_{\pp,1}\oplus\ldots\oplus W_{\pp,m_\pp}$ for pairwise inequivalent irreducible representations $W_{\pp,i}$ of $\Delta'_\pp$.
\end{lem}

\begin{proof}
Let $W_{\pp,1}$ be any irreducible representation of $\Delta'_\pp$ contained in $W_\pp$. Then the sum of the conjugates $\delta W_{\pp,1}$ for all $\delta\in\Delta_\pp$ is a non-zero $\Delta_\pp$-invariant subspace. By irreducibility it is therefore equal to $W_\pp$ for all $\pp\not\in P^{\rm irr}$. Thus $W_\pp$ is the direct sum of certain conjugates $\delta W_{\pp,1}$. 

It remains to show that these summands are pairwise inequivalent. For this suppose that $\delta_1W_{\pp,1}$ and $\delta_2W_{\pp,1}$ are distinct but equivalent as representations of $\Delta'_\pp$ for some $\delta_1,\delta_2\in\Delta_\pp$. Then for every $\delta\in\Delta_\pp$, we have $\delta^{c'_n!} \in \Delta'_\pp$, and this element has the same eigenvalues on $\delta_1W_{\pp,1}$ and $\delta_2W_{\pp,1}$. By Lemma \ref{Rprop} (a) we thus have $f(\delta^{c'_n!})=0$. But since $\delta\in\Delta_\pp$ is arbitrary, by Lemma \ref{fgh_on_delta1} with $c=c'_n!$ this can happen only for finitely many primes~$\pp$, as desired.
\end{proof}

The stated properties imply that the decomposition in Lemma \ref{Wdecomp} is the isotypic decomposition of $W_\pp$ under $\Delta'_\pp$. It is therefore normalized by $\Delta_\pp$, and so the permutation action is given by a homomorphism from $\Delta_\pp$ to the symmetric group $S_{m_\pp}$ on $m_\pp$ letters. Let $\sigma_\pp$ denote the composite homomorphism $G_K\twoheadrightarrow \Delta_\pp \to S_{m_\pp}$.

\begin{lem}\label{PermRepUnram}
For almost all primes $\pp\not\in P^{\rm irr}$ of $A$, the homomorphism $\sigma_\pp$ is unramified at all points of codimension $1$ of $\overline{X}$.
\end{lem}

\begin{proof}
This is clear for points in $X$, because $\overline{\rho}_\pp$ is already unramified there. So let $x$ be one of the finitely many generic points of $\overline{X}\smallsetminus X$. Since $|S_{m_\pp}| \leq m_\pp! \leq n!$ is bounded, Lemma \ref{no_inert} implies that $\sigma_\pp|I'_x$ is trivial for almost all $\pp$. Then $I'_x$ stabilizes each summand $W_{\pp,i}$. Since $I'_x$ acts unipotently by (\ref{inertunip}), we deduce that $W_{\pp,i}^{I'_x}\not=0$ for every $i$. On the other hand Lemma \ref{inert_big} implies that $W_\pp^{I_x} = W_\pp^{I'_x}$ for almost all~$\pp$. This means that $I_x$ acts trivially on $W_\pp^{I'_x} = W_{\pp,1}^{I'_x}\oplus\ldots\oplus W_{\pp,m_\pp}^{I'_x}$. But as all these summands are non-zero, and $I_x$ permutes them according to the restriction of the homomorphism $\sigma_\pp$, it follows that $\sigma_\pp|I_x$ is trivial, as desired.
\end{proof}

\begin{lem}\label{DeltaPrimeIrred}
For almost all primes $\pp\not\in P^{\rm irr}$ of $A$, the group $\Delta'_\pp$ acts irreducibly on~$W_\pp$.
\end{lem}

\begin{proof}
Combining Lemma \ref{PermRepUnram}, the inequality $m_\pp\leq n$, and Theorem \ref{Pop}, we find that there are only finitely many possibilities for the homomorphism $\sigma_\pp$. The intersection of their kernels is therefore equal to $G_{K'}$ for some subextension $K'\subset K^{\text{sep}}$ that is finite over $K$. Applying Corollary \ref{abs_irred} with $K'$ in place of $K$ implies that $\overline{\rho}_\pp(G_{K'})$ acts irreducibly on $W_\pp$ for almost all $\pp$. But by construction $\overline{\rho}_\pp(G_{K'})$ stabilizes each summand of the decomposition in Lemma \ref{Wdecomp}; hence $m_\pp=1$ and $W_\pp = W_{\pp,1}$ for almost all $\pp$. Then $\Delta'_\pp$ acts irreducibly on $W_\pp$, as desired.
\end{proof}

\begin{lem}\label{almost_surjective2}
For almost all primes $\pp\not\in P^{\rm irr}$ of $A$, there exist a finite subfield $k_\pp'$ of $\bar k_\pp$ and a model $G'_\pp$ of $\SL_{n,\bar k_\pp}$ over $k'_\pp$, such that $\Delta_\pp^{\prime\rm der} = G'_\pp(k'_\pp)$.
\end{lem}

\begin{proof}
By Lemma \ref{DeltaPrimeIrred} the group $\Delta'_\pp$ acts irreducibly on $W_\pp$ for almost all~$\pp$. On the other hand let $c$ be the constant from Theorem \ref{finite_main1}. Then for almost all $\pp$, Lemma \ref{fgh_on_delta1} shows that the map $\Delta_\pp \to k_\pp$, $\delta \mapsto f(\delta^{c'_nc})$ is not identically zero. Since $\delta^{c'_n}\in\Delta'_\pp$ for all $\delta\in\Delta_\pp$, it follows that the map $\Delta'_\pp \to k_\pp$, $\delta' \mapsto f(\delta^{\prime c})$ is not identically zero. Together we find that $\Delta'_\pp$ satisfies the assumptions of Theorem \ref{finite_main1}, and so the desired assertion follows.
\end{proof}

\begin{lem}\label{kp_is_kp1}
For almost all primes $\mathfrak{p}$ of $A$ as in Lemma \ref{almost_surjective2} we have $k_\pp \subset k_\pp'$.
\end{lem}

\begin{proof}
Let $P'$ be the finite set of primes excluded by Lemma \ref{almost_surjective2}, and let $P^{\rm trad}$ be the finite set of primes from Proposition \ref{residue_field_trad1}. We claim that the assertion holds whenever $\pp\not\in P'\cup P^{\rm trad}$.

To prove this let $\Ad$ denote the adjoint representation of $\GL_n$. Take any element $\delta\in\Delta_\pp$, and let $\mathop{\rm int}(\delta)$ denote the corresponding inner automorphism of $\GL_{n,k_\pp}$. Then $\Ad(\delta)$ is the derivative $d(\mathop{\rm int}(\delta))$, and its trace is an element of $k_\pp$.

On the other hand $\mathop{\rm int}(\delta)$ induces an algebraic automorphism of $\SL_{n,\bar k_\pp}$ which normalizes $\Delta_\pp^{\prime\rm der} = G'_\pp(k'_\pp)$. By the uniqueness in Proposition \ref{same_k_and_model} it therefore induces an algebraic automorphism of the model $G'_\pp$ over $k_\pp'$. The derivative of this automorphism is an automorphism of the Lie algebra $\mathop{\rm Lie}G'_\pp$, whose trace is therefore an element of $k_\pp'$. But the fact that $G'_\pp$ is a model of $\SL_{n,\bar k_\pp}$ yields an equivariant isomorphism $\mathop{\rm Lie}G'_\pp \otimes_{k_\pp'} \nobreak \bar k_\pp \allowbreak \cong \sll_n(\bar k_\pp)$, and so the trace in question is equal to the trace of $d(\mathop{\rm int}(\delta))|\sll_n(\bar k_\pp)$. Together we deduce that
$$\Tr\Ad(\delta) \ =\ \Tr\bigl(d(\mathop{\rm int}(\delta))|\sll_n(\bar k_\pp)\bigr) + 1 
\ \in\ k'_\pp.$$

In particular, we can apply this to $\delta = \overline{\rho}_\pp(\Frob_x)$ for any closed point $x\in X$. Then $\Tr\Ad(\overline{\rho}_\pp(\Frob_x))$ is the image of $\Ad(\rho_\pp(\Frob_x))$ in the residue field $k_\pp$. Varying $x$, the elements $\Ad(\rho_\pp(\Frob_x))$ generate the ring of traces $R^{\text{trad}}$ from Subsection \ref{ring_of_trances}. Thus by Proposition \ref{residue_field_trad1} (a) their images generate the residue field $k_\pp$. Since these images also lie in $k'_\pp$, we deduce that $k_\pp \subset k_\pp'$, as desired.
\end{proof}

\begin{prop}\label{res_surj1}
For almost all primes $\pp\not\in P_0$ of $A$, we have $\Delta_\pp^{{\geom}} = \SL_n(k_\pp)$.
\end{prop}

\begin{proof}
We prove that this holds for all primes $\pp$ satisfying Lemmas \ref{almost_surjective2} and \ref{kp_is_kp1}. Indeed, Lemma \ref{almost_surjective2} shows that $G'_\pp(k'_\pp) = \Delta_\pp^{\prime\rm der} \subset \GL_n(k_\pp)^{\rm der} = \SL_n(k_\pp)$. Applying Proposition \ref{fieldcontained} with $G=\SL_{n,k_\pp}$, this implies that $|k_\pp'|\le|k_\pp|$. On the other hand we have $k_\pp \subset k_\pp'$ by Lemma \ref{kp_is_kp1}; hence together we deduce that $k_\pp = k_\pp'$. Applying Proposition \ref{same_model} with $G=\SL_{n,k_\pp}$ then shows that $G'(k'_\pp) = \SL_n(k_\pp)$. In particular we have $\SL_n(k_\pp) = \Delta_\pp^{\prime\rm der} \subset \Delta_\pp^{\rm der} \subset \Delta_\pp^{\text{geom}} \subset \SL_n(k_\pp)$, and so these inclusions are equalities, as desired.
\end{proof} 


\subsection{Surjectivity at a single prime}\label{52}

\begin{prop}\label{surj1} 
For almost all primes $\pp\not\in P_0$ of $A$, 
we have $\Gamma_\pp^{\geom}=\SL_n(A_\pp)$.
\end{prop}

\begin{proof}
Let $P'$ be the finite set of primes $\pp$ excluded by Proposition \ref{res_surj1} or satisfying $|k_\pp|\leq9$. For all $\pp\not\in P'$ we have a surjective homomorphism $\Gamma_\pp^{{\geom}}\to\SL_n(k_\pp)$.

Suppose first that $F^{\text{trad}} = F$. Let $P^{\rm trad}$ be the finite set of primes from Proposition \ref{residue_field_trad1}. Then for any prime $\pp\not\in P'\cup P^{\rm trad}$, the set of traces $\Tr\Ad(\Gamma_\pp)$ topologically generates $A_\pp$. Applying Theorem \ref{trace_crit_1} to $\Gamma_\pp \subset \GL_n(A_\pp)$ thus shows that $\Gamma_\pp^{\rm der}=\SL_n(A_\pp)$. 

Suppose now that $F^{\text{trad}}=F^2$. Then $p=n=2$ by Theorem \ref{field_trad}. By Lemma \ref{small_quotients1} there exists a finite set $P''\supset P_0$ of primes of $A$, such that the number of continuous homomorphisms from $G_K$ to a finite group of order $2$, which factor through the surjection $\rho_\pp:G_K\twoheadrightarrow\Gamma_\pp$ for some $\pp\not\in P''$, is finite. The intersection of the kernels of these homomorphisms is then $G_{K'}$ for some finite extension $K'\subset K^{\text{sep}}$ of $K$. 
If $\Gamma_\pp'$ denotes the intersection of all closed subgroups of index $2$ of $\Gamma_\pp$, it follows that for all primes $\pp\not\in P''$ of $A$ we have $\rho_\pp(G_{K'}) \subset \Gamma_\pp'$. 
Let $P^{\rm trad}$ be the finite set of primes obtained by applying Proposition \ref{residue_field_trad1} with $K'$ in place of $K$. Then for any prime $\pp\not\in P'\cup P''\cup P^{\rm trad}$, the set of traces $\Tr\Ad(\rho_\pp(G_{K'}))$, and hence a fortiori the set of traces $\Tr\Ad(\Gamma_\pp')$, topologically generates the subring $A_\pp^2 := \{a^2\mid a\in A_\pp\}$. Applying Theorem \ref{trace_crit_2} to $\Gamma_\pp \subset \GL_2(A_\pp)$ thus shows that $\Gamma_\pp^{\rm der}=\SL_2(A_\pp)$. 

Since $\Gamma_\pp^{\rm der} \subset \Gamma_\pp^{{\geom}} \subset \SL_n(A_\pp)$, the proposition follows in either case.
\end{proof}


\subsection{Residual surjectivity at several primes}\label{53}

For any finite set $P$ of primes $\pp\not=\pp_0$ of $A$, we let 
$$\Delta_P^{\geom}
\ \triangleleft\ \Delta_P\ \subset\ \prod_{\pp\in P} \;
(D_\pp/\pp D_\pp)^\times$$
denote the images of $G_K^{\geom} \triangleleft G_K$ under the combined representation induced by $\overline{\rho}_\pp$. Recall that $(D_\pp/\pp D_\pp)^\times \cong \GL_n(k_\pp)$ and $\Delta_\pp^{\geom} \subset \SL_n(k_\pp)$ whenever $\pp\not\in P_0$. Thus whenever $P\cap P_0=\varnothing$, we have 
$$\Delta_P^{\geom} \ \subset\ \prod_{\pp\in P} \SL_n(k_\pp).$$ 

\begin{prop}\label{res_surj_fin}
There exists a finite set $P_1$ of primes of $A$ containing $P_0$, such that for any finite set of primes $P$ of $A$ satisfying $P\cap P_1=\varnothing$,
we have $\Delta_P^{\geom}= \prod_{\pp\in P} \SL_n(k_\pp)$.
\end{prop}

\begin{proof}
Let $P'$ be the finite set of primes $\pp$ excluded by Proposition \ref{res_surj1} or satisfying $|k_\pp|\leq3$. Let $P^{\rm trad}$ be the finite set of primes from Proposition \ref{residue_field_trad1}, and set $P_1:=P'\cup P^{\rm trad}$. We claim that the assertion holds whenever $P\cap P_1=\varnothing$.

For any $\pp\in P$ abbreviate $\PSL(n,k_\pp) := \SL_n(k_\pp) / Z(\SL_n(k_\pp))$. The assumption $|k_\pp|>3$ implies that this is a non-abelian finite simple group and that $\SL_n(k_\pp)$ is perfect. Let
$$\overline{\Delta_P^{\geom}} \ \subset\ \prod_{\pp\in P} \PSL(n,k_\pp)$$ 
denote the image of $\Delta_P^{\geom}$. Then it suffices to prove that this inclusion is an equality. 

Assume otherwise. {}From Proposition \ref{res_surj1} we know that $\overline{\Delta_P^{\geom}}$ surjects to all factors. Since these factors are non-abelian simple groups, Goursat's Lemma implies that $\overline{\Delta_P^{\geom}}$ lies over the graph of an isomorphism between two factors, say associated to distinct primes $\pp_1,\pp_2\in P$. Then the situation persists after replacing $P$ by $\{\pp_1,\pp_2\}$; hence we may without loss of generality assume that $P=\{\pp_1,\pp_2\}$.

The isomorphism $\PSL(n,k_{\pp_1}) \stackrel{\sim}{\to} \PSL(n,k_{\pp_2})$ is induced by a field isomorphism  $\sigma: k_{\pp_1} \stackrel{\sim}{\to} k_{\pp_2}$ and a corresponding isomorphism of algebraic groups $\alpha: \sigma^* \PGL_{r,k_{\pp_1}} \stackrel{\sim}{\to} \PGL_{r,k_{\pp_2}}$ (see \cite{PSTAX}, Lemmas 9.4 and 9.5). Since $\SL_n(k_{\pp_1}) \times \SL_n(k_{\pp_2})$ is a central extension of $\PSL(n,k_{\pp_1}) \times \PSL(n,k_{\pp_2})$, the derived group $(\Delta_P^{\geom})^{\rm der}$ of $\Delta_P^{\geom}$ depends only on $\overline{\Delta_P^{\geom}}$. It is therefore the graph of the isomorphism $\SL_n(k_{\pp_1}) \stackrel{\sim}{\to} \SL_n(k_{\pp_2})$ induced by the unique isomorphism $\tilde\alpha: \sigma^* \SL_{r,k_{\pp_1}} \stackrel{\sim}{\to} \SL_{r,k_{\pp_2}}$ lifting $\alpha$. 

The uniqueness of the model from Proposition \ref{same_model} implies that the isomorphism $\tilde\alpha$ depends only on $(\Delta_P^{\geom})^{\rm der}$. Thus its graph depends only on $(\Delta_P^{\geom})^{\rm der}$. Since $\Delta_P$ normalizes $(\Delta_P^{\geom})^{\rm der}$ by construction, it thus also normalizes the graph of $\tilde\alpha$. In other words, for every $\delta=(\delta_1,\delta_2)\in\Delta_P$, the following diagram commutes:
$$\xymatrix{
\sigma^* \SL_{r,k_{\pp_1}} \ar[r]^-{\tilde\alpha} 
\ar[d]_{\sigma^*{\rm int}(\delta_1)}
& \SL_{r,k_{\pp_2}} \ar[d]^{{\rm int}(\delta_2)} \\
\sigma^* \SL_{r,k_{\pp_1}} \ar[r]^-{\tilde\alpha} 
& \SL_{r,k_{\pp_2}} \rlap{.}\\}$$
Taking traces and recalling that the trace on $\gll_n$ is the trace on $\sll_n$ plus $1$, we deduce that $\sigma(\Tr\Ad(\delta_1)) = \Tr\Ad(\delta_2)$.

In particular, we can apply this when $\delta$ is the image of $\Frob_x$ for any closed point $x\in X$. Then $\Tr\Ad(\delta_i) = \Tr\Ad(\overline{\rho}_{\pp_i}(\Frob_x))$ is the image of $\Ad(\rho_{\pp_i}(\Frob_x))$ in the residue field $k_{\pp_i}$, where $\Ad(\rho_{\pp_i}(\Frob_x)) \in R^{\text{trad}}$ is independent of $i$. Thus $\Tr\Ad(\delta) = \bigl( \Tr\Ad(\delta_1) , \sigma(\Tr\Ad(\delta_1)) \bigr)$ is the image of $\Ad(\rho_{\pp_1}(\Frob_x)) \in R^{\text{trad}}$ in the product of the residue fields $k_{\pp_1} \times k_{\pp_2}$. Since the elements $\Ad(\rho_{\pp_1}(\Frob_x))$ for all $x$ generate the ring of trances $R^{\text{trad}}$, it follows that the image of the reduction map $R^{\text{trad}} \to k_{\pp_1} \times k_{\pp_2}$ is contained in the graph of $\sigma$. But since $P
\cap P^{\rm trad}=\varnothing$, this contradicts Proposition \ref{residue_field_trad1} (b). Therefore $\overline{\Delta_P^{\geom}}$ cannot be a proper subgroup, and we are finished.
\end{proof}

\begin{lem} \label{res_indie}
There exists a finite set $P_2$ of primes $\pp\not=\pp_0$ of $A$ containing $P_0$, such that for every finite $P\supset P_2$ and every $\pp\not\in P$, we have 
$$\Delta^{\geom}_{P\cup\{\pp\}} \ =\ \Delta^{\geom}_P \times \SL_n(k_\pp).$$
\end{lem}

\begin{proof}
Let $P_1$ be the finite set of primes excluded by Proposition \ref{res_surj_fin}. 
Let $N$ be the maximum of the orders of all Jordan-H\"older constituents of the finite group $\Delta^\text{geom}_{P_1}$. Let $P_2$ be the union of $P_1$ with the set of primes $\pp$ for which $|\PSL(n,k_\pp)|\leq N$ or $|k_\pp|\leq9$. We will prove the assertion whenever $P\supset P_2$.

Consider the natural inclusion
$$\Delta^{\geom}_{P\cup\{\pp\}} \ \subset\ 
\Delta^{\geom}_{P_1} \times \Delta^{\geom}_{P\smallsetminus P_1} \times \SL_n(k_\pp).$$
By definition the image of $\Delta^{\geom}_{P\cup\{\pp\}}$ under the projection to the second and third factors is the subgroup
$$\Delta^{\geom}_{(P\smallsetminus P_1)\cup\{\pp\}} \ \subset\ 
\Delta^{\geom}_{P\smallsetminus P_1} \times \SL_n(k_\pp)
\ \subset\prod_{\pp'\in P\smallsetminus P_1}\!\! \SL_n(k_{\pp'}) \ \times\ \SL_n(k_\pp).$$
These inclusions are equalities by Proposition \ref{res_surj_fin}. Therefore the projection homomorphism $\Delta^{\geom}_{P\cup\{\pp\}} \to \Delta^{\geom}_{P\smallsetminus P_1} \times \SL_n(k_\pp)$ is surjective. From this it follows that 
$$E\ :=\ \Delta^{\geom}_{P\cup\{\pp\}} \cap \bigl(\Delta^{\geom}_{P_1} {\times} \{1\} {\times} \SL_n(k_\pp) \bigr)$$
surjects to $\SL_n(k_\pp)$. In particular $\PSL(n,k_\pp)$ is a Jordan-H\"older factor of $E$. The assumption $\pp\not\in P_1$ implies that the order of $\PSL(n,k_\pp)$ is greater than the order of any Jordan-H\"older constituent of $\Delta^\text{geom}_{P_1}$. Thus $\PSL(n,k_\pp)$ cannot be a Jordan-H\"older constituent of the image of $E$ in $\Delta^\text{geom}_{P_1}$. It must therefore be a Jordan-H\"older factor of $\Delta^{\geom}_{P\cup\{\pp\}} \cap \bigl(\{1\} {\times} \{1\} {\times} \SL_n(k_\pp) \bigr)$. Since $\SL_n(k_\pp)$ is perfect, it follows that 
$$\{1\} {\times} \{1\} {\times} \SL_n(k_\pp) \ \subset\ 
E  \ \subset\ \Delta^{\geom}_{P\cup\{\pp\}}.$$
The short exact sequence
$$\xymatrix{
1 \ar[r] &
\{1\} {\times} \{1\} {\times} \SL_n(k_\pp) \ar[r] &
\Delta^{\geom}_{P\cup\{\pp\}} \ar[r] &
\Delta^{\geom}_P \ar[r] & 1\\}$$
and the $5$-Lemma then show that $\Delta^{\geom}_{P\cup\{\pp\}} \ =\ \Delta^{\geom}_P \times \SL_n(k_\pp)$, as desired.
\end{proof}


\subsection{Adelic openness}\label{54}

For any finite set $P$ of primes $\pp\not=\pp_0$ of $A$, we let $\Gamma_P^{\text{geom}}$ denote the image of the combined homomorphism
$$(\rho_\pp)_{\pp\in P}:\ 
  G_K^{\text{geom}} \longrightarrow \prod_{\pp\in P} D_\pp^1.$$
Recall that $D_\pp^1 \cong \SL_n(A_\pp)$ whenever $\pp\not\in P_0$.

\begin{lem} \label{indie}
There exists a finite set $P_3$ of primes $\pp\not=\pp_0$ of $A$ containing $P_0$, such that for every finite $P\supset P_3$ and every $\pp\not\in P$, we have 
$$\Gamma^{\geom}_{P\cup\{\pp\}} \ =\ \Gamma^{\geom}_P \times \SL_n(A_\pp).$$
\end{lem}

\begin{proof}
Let $P'$ be the finite set of primes $\pp$ excluded by Proposition \ref{surj1} or satisfying $|k_\pp|\leq9$. Let $P_3$ be the union of $P'$ with the set of primes $P_2$ from Lemma \ref{res_indie}. We will prove the assertion whenever $P\supset P_3$.

For this we consider the commutative diagram
\def\twoheaddownarrow{\hbox{$\xymatrix@R-4pt{\ar@{->>}[d]\\ \\}$}}
$$\begin{array}{ccccc}
\Gamma^{\geom}_{P\cup\{\pp\}}
&\subset& \Gamma^{\geom}_P
&\!\!\!\times\!\!\!& \SL_n(A_\pp) \\[-7pt]
\twoheaddownarrow && \twoheaddownarrow && \twoheaddownarrow \\
\Delta^{\geom}_{P\cup\{\pp\}} 
&\subset& \Delta^{\geom}_P
&\!\!\!\times\!\!\!& \SL_n(k_\pp)\rlap{.}
\end{array}$$
The inclusion in the lower row is an equality by Lemma \ref{res_indie}. Thus if $H$ denotes the kernel of the surjection $\Gamma^{\geom}_P \twoheadrightarrow \Delta^{\geom}_P$, it follows that $\Gamma^{\geom}_{P\cup\{\pp\}} \cap \bigl(H {\times} \SL_n(A_\pp) \bigr)$ surjects to $\{1\}\times\SL_n(k_\pp)$. But by construction $H$ is a pro-$p$-group, and $\SL_n(k_\pp)$ has no Jordan-H\"older factor of order $p$. Since all groups in question are pro-finite, we deduce that
$$\Gamma_\pp'\ :=\ \Gamma^{\geom}_{P\cup\{\pp\}} 
  \cap \bigl(\{1\} {\times} \SL_n(A_\pp) \bigr)$$
also surjects to $\SL_n(k_\pp)$. 

By construction $\Gamma'_\pp$ is a closed normal subgroup of $\Gamma^{\geom}_{\{\pp\}\cup P}$, and the conjugation action of $\Gamma^{\geom}_{\{\pp\}\cup P}$ on it factors through the projection $\Gamma^{\geom}_{\{\pp\}\cup P} \twoheadrightarrow \Gamma^{\geom}_\pp \subset \SL_n(A_\pp)$. Since $\pp\not\in P'$, the last inclusion is an equality by Proposition \ref{surj1}. Together this implies that $\Gamma'_\pp$ is normalized by $\SL_n(A_\pp)$. 

Combining this with the assumption $|k_\pp|>9$ and the fact that $\Gamma_\pp'$ surjects to $\SL_n(k_\pp)$, Proposition \ref{normal_sa} now implies that $\Gamma'_\pp = \{1\}\times\SL_n(A_\pp)$. The short exact sequence
$$\xymatrix{
1 \ar[r] &
\{1\} {\times} \SL_n(A_\pp) \ar[r] &
\Gamma^{\geom}_{P\cup\{\pp\}} \ar[r] &
\Gamma^{\geom}_P \ar[r] & 1\\}$$
and the $5$-Lemma then show that $\Gamma^{\geom}_{P\cup\{\pp\}} \ =\ \Gamma^{\geom}_P \times \SL_n(A_\pp)$, as desired.
\end{proof}

\begin{proof}[Proof of Theorem \ref{main_theorem} (a).]
Let $P_3$ be as in Lemma \ref{indie}. Then induction on $P$ shows that for every finite $P\supset P_3$, we have 
$$\Gamma^{\geom}_P \ =\ \Gamma^{\geom}_{P_3} \times \prod_{\pp\in P\smallsetminus P_3} \SL_n(A_\pp).$$
In the limit this implies that
$$\rho_{\text{ad}}(G_K^{\text{geom}}) \ =\ \Gamma^{\geom}_{P_3} \times \prod_{\pp\not\in P_3} \SL_n(A_\pp).$$
But $\Gamma^{\text{geom}}_{P_3}$ has finite index in $\prod_{\pp\in P_3} D^1_\pp$ by Theorem \ref{openness}. Therefore $\rho_{\text{ad}}(G_K^{\text{geom}})$ has finite index in $\prod_{\pp\not=\pp_0} D^1_\pp$, as desired.
\end{proof} 


\subsection{Absolute Galois group}\label{55}

\begin{proof}[Proof of Theorem \ref{main_theorem} (b).]
Recall that $R := \End_K(\phi) = \End_{K^{\text{sep}}}(\phi)$ by Assumption \ref{asses} (a), and that $D_\pp$ was defined as the commutant of $R_\pp := R\otimes_AA_\pp$ in $\End_{A_\pp}(T_\pp(\phi))$. Thus $\rho_{\text{ad}}(G_K)$ is contained in $\prod_{\pp\not=\pp_0} D^\times_\pp$. We will look at its image under the determinant map.

Let $F'$ be a maximal commutative $F$-subalgebra of $R\otimes_A F$, let $A'$ denote the integral closure of $A$ in~$F'$, and choose a Drinfeld $A'$-module $\phi'\colon A'\to K\{\tau\}$ and an isogeny $f\colon\phi\to\phi'|A$, as in Subsection \ref{tatess}. The characteristic of $\phi'$ is then a prime $\pp_0'$ of $A'$ that divides~$\pp_0$. 
By Anderson \cite{Anderson_t_Motives}, \S4.2, there exists a Drinfeld $A'$-module $\psi'\colon A'\to K\{\tau\}$ of rank~$1$ and characteristic $\pp_0'$ whose adelic Galois representation is isomorphic to the determinant of the adelic Galois representation associated to~$\phi'$. With Proposition \ref{boalo} it follows that the composite homomorphism
$$\xymatrix{
\det\rho_{\text{ad}}\colon\ G_K \ar[r]^-{\rho_{\text{ad}}} &
\prod\limits_{\pp\not=\pp_0}\!\! D^\times_\pp \ar[r]^-{\det} &
\prod\limits_{\pp\not=\pp_0}\!\! A^\times_\pp \ar@{^{ (}->}[r] & 
\prod\limits_{\pp'\nmid\pp_0}\!\! A^{\prime\times}_{\pp'} \\ }$$
describes the Galois representation on the Tate modules $\prod_{\pp'\nmid\pp_0} T_{\pp'}(\psi')$.

Without loss of generality we may assume that $\psi'$ is defined over the finite field~$\kappa$. Let $m$ denote the degree of $\kappa$ over~$\mathbb{F}_p$. Then $\Frob_\kappa = \tau^m$ lies in the center of $\kappa\{\tau\}$. In particular it commutes with $\psi'_{a'}$ for all $a'\in A'$ and is therefore an endomorphism of~$\psi'$. As $\psi'$ has rank $1$, its endomorphism ring is equal to~$A'$; hence $\Frob_\kappa$ represents an element $a'\in A'$. The action of $\Frob_\kappa$ as an element of the Galois group $G_\kappa$ on all Tate modules of $\psi'$ is then just multiplication by~$a'$. Since $a'$ is the single eigenvalue of $\Frob_\kappa$ associated to $\psi'$, Proposition \ref{F_eigenvalues} implies that $a'$ is divisible by $\pp'_0$ but not by any other prime of~$A'$. 

For every element $\sigma\in G_K$ whose restriction to $\bar\kappa$ is $\Frob_\kappa$ we thus have $\det\rho_{\text{ad}}(\sigma) = a'$ diagonally embedded into $\prod_{\pp'\nmid\pp_0} A^{\prime\times}_{\pp'}$. But it also lies in the subgroup $\prod_{\pp\not=\pp_0} A^\times_\pp$, whose intersection with the diagonally embedded $A'$ is~$A$. Thus $a'$ is actually an element of~$A$, divisible by $\pp_0$ but not by any other prime of~$A$. Moreover, we have $\det\rho_{\text{ad}}(G_K) = \smash{\overline{\langle a'\rangle}}$, the pro-cyclic subgroup topologically generated by~$a'$.

Now both $a'$ and the $a_0$ in Theorem \ref{main_theorem} are elements of $A$ that are divisible by $\pp_0$ but not by any other prime of~$A$. Thus the corresponding ideals are $(a')=\pp_0^i$ and $(a_0)=\pp_0^j$ for some positive integers $i$ and~$j$. Together it follows that $(a^{\prime j}) = \pp_0^{ij} = (a_0^i)$, and so $a^{\prime j}/a_0^i$ is a unit in~$A^\times$. As the group of units is finite, we deduce that $a^{\prime j\ell} = a_0^{i\ell}$ for some positive integer~$\ell$. Thus the subgroup $\smash{\overline{\langle a'\rangle}}$ is commensurable to $\smash{\overline{\langle a_0\rangle}}$.

On adjoining to $K$ a suitable finite extension of the constant field~$\kappa$ we can replace $a'$ by any positive integral power. We can therefore reduce ourselves to the case that $\smash{\overline{\langle a'\rangle}} \subset \smash{\overline{\langle a_0^n\rangle}}$ with $n$ as in Assumption \ref{asses} (c). Then $\det(a_0) = a_0^n$, and from this we see that the middle row in the following commutative diagram is exact and the upper right rectangle is cartesian. This together with the inclusion $\smash{\overline{\langle a'\rangle}} \subset \smash{\overline{\langle a_0^n\rangle}}$ yields the inclusions in the lower half of the diagram:
$$\xymatrix@R-10pt@C+10pt{
1 \ar[r] &
\prod\limits_{\pp\not=\pp_0}\!\!D^1_\pp \ar[r] &
\prod\limits_{\pp\not=\pp_0}\!\!D^\times_\pp \ar[r]^-{\det} &
\prod\limits_{\pp\not=\pp_0}\!\!A^\times_\pp & \\
1 \ar[r] &
\prod\limits_{\pp\not=\pp_0}\!\!D^1_\pp \ar[r] \ar@{}[u]|<<<<{\big\Vert} &
\overline{\langle a_0\rangle} \cdot \!\prod\limits_{\pp\not=\pp_0}\!\!D^1_\pp \ar[r] \ar@{}[u]|<<<<{\bigcup} &
\overline{\langle a_0^n\rangle} \ar[r] \ar@{}[u]|<<<<{\bigcup} & 1 \\
1 \ar[r] &
\rho_{\text{ad}}(G_K) \cap \!\prod\limits_{\pp\not=\pp_0}\!\!D^1_\pp \ar[r] \ar@{}[u]|<<<<{\bigcup} &
\rho_{\text{ad}}(G_K) \ar[r] \ar@{}[u]|<<<<{\bigcup} &
\overline{\langle a'\rangle} \ar[r] \ar@{}[u]|<<<<{\bigcup} &1\rlap{.} \\}$$
Theorem \ref{main_theorem} (a) implies that the inclusion at the lower left is of finite index. By the above the same is true for the inclusion at the lower right. Since the bottom row is also exact, it follows that the inclusion at the lower middle is also of finite index. This shows that $\rho_{\text{ad}}( G_K)$ is commensurable to $\overline{\langle a_0\rangle} \cdot \prod_{\pp\not=\pp_0} D^1_\pp$, finishing the proof of Theorem \ref{main_theorem}~(b).
\end{proof} 


\newpage


\section{Arbitrary endomorphism ring}
\label{generalcase}


As in Section \ref{notation}, we let $K$ be a field that is finitely generated over a finite field $\kappa$ and let $\phi: A \rightarrow K\{\tau\}$ be a Drinfeld $A$-module of rank $r$ over $K$ of special characteristic~$\pp_0$. We keep the relevant notations of Section \ref{notation}, but do not impose any other restrictions.
Set $R:= \End_{K^{\text{sep}}}(\phi)$ and $F := \mathop{\rm Quot}(A)$. Then $R \otimes_A F$ is a division algebra of finite dimension over $F$. Let $Z$ denote its center and write
$$\dim_Z (R\otimes_A F) =d^2 \quad \text{and} \quad [Z/F]=e.$$ 
Then $de$ divides $r$ by Proposition \ref{numerics}.


\subsection{The isotrivial case}\label{61}

\begin{defn}\label{isotrivialdef}
We call $\phi$ \emph{isotrivial} if over some field extension it is isomorphic to a Drinfeld $A$-module defined over a finite field.
\end{defn}

Clearly this property is invariant under extending~$K$.

\begin{prop}\label{isotrivialprop}
\begin{enumerate}
\item[(a)] $\phi$ is isotrivial if and only if it is isomorphic over $K^{\text{sep}}$ to a Drinfeld $A$-module defined over a finite subfield of $K^{\text{sep}}$.
\item[(b)] Let $\phi'$ be another Drinfeld $A$-module over $K$ that is isogenous to~$\phi$. Then $\phi$ is isotrivial if and only if $\phi'$ is isotrivial.
\item[(c)] Let $B$ be any integrally closed infinite subring of~$A$. Then $\phi$ is isotrivial if and only if $\phi|B$ is isotrivial.
\end{enumerate}
\end{prop}

\begin{proof}
In (a) the `if' part is obvious. For the `only if' part assume that $L$ is a field extension of $K$ such that $\phi$ is isomorphic over $L$ to a Drinfeld $A$-module $\psi$ defined over a finite subfield $\ell\subset L$. By the definition of isomorphisms there is then an element $u\in L^\times$ such that $\phi_a=u\circ\psi_a\circ u^{-1}$ in $L\{\tau\}$ for all $a\in A$. Choose a prime $\pp\not=\pp_0$ of $A$ and, after replacing $L$ by a finite extension, a non-zero torsion point $t\in\phi[\pp](L)$. Then $t$ is separably algebraic over~$K$. On the other hand $ut$ is a non-zero torsion point of $\psi$ and therefore algebraic over~$\ell$. Since $\ell$ is finite, $ut$ is actually separable over~$\ell$. Thus the subfiend $K\ell(u,ut)\subset L$ is separably algebraic over~$K$ and can therefore be embedded into $K^{\text{sep}}$. Then $u=ut/t$ defines an isomorphism $\phi\cong\psi$ over $K^{\text{sep}}$, as desired.

In (b) by symmetry it suffices to prove the `if' part. So assume that $L$ is a field extension of $K$ such that $\phi$ is isomorphic over $L$ to a Drinfeld $A$-module $\psi$ defined over a finite subfield $\ell\subset L$. Then $\phi'$ is isogenous to $\psi$ over~$L$. By the definition of isogenies this means that there is a non-zero element $f\in L\{\tau\}$ such that $\phi'_a\circ f=f\circ\psi_a$ for all $a\in A$. Its scheme theoretic kernel $\mathop{\rm ker}(f)$ is then a finite subgroup scheme of ${\mathbb G}_{a,L}$ that is mapped to itself under $\psi_a$ for all $a\in A$. Its identity component is a finite infinitesimal subgroup scheme of ${\mathbb G}_{a,L}$ and therefore the kernel of some power of~$\tau$. On the other hand all its geometric points are torsion points of $\psi$ and therefore algebraic over~$\ell$. Together it follows that $\mathop{\rm ker}(f)$ is defined over some finite extension $\ell'\subset L$ of $\ell$ and is therefore the kernel of some non-zero element $g\in L\{\tau\}$. Since $\mathop{\rm ker}(f) = \mathop{\rm ker}(g)$, it now follows that $f=u\circ g$ for some element $u\in L^\times$. Consider the Drinfeld $A$-module $\psi': A\to L\{\tau\}$ defined by $\psi'_a := u^{-1}\circ\phi'_a\circ u$. Then the relation $\phi'_a\circ f=f\circ\psi_a$ implies that $\psi'_a\circ g=g\circ\psi_a$ for all $a\in A$. Since $g$ and $\psi_a$ have coefficients in~$\ell'$, this relation implies that $\psi'_a$ also has coefficients in~$\ell'$. In other words $\psi'$ is really defined over~$\ell'$, and since $\phi'\cong\psi'$, it follows that $\phi'$ is isotrivial, as desired.

In (c) the `only if' part is obvious. For the `if' part assume that $L$ is a field extension of $K$ such that $\phi|B$ is isomorphic over $L$ to a Drinfeld $B$-module $\psi'$ defined over a finite subfield $\ell\subset L$. By the definition of isomorphisms there is then an element $u\in L^\times$ such that $\phi_b=u\circ\psi'_b\circ u^{-1}$ in $L\{\tau\}$ for all $b\in B$. Consider the Drinfeld $A$-module $\psi: A\to L\{\tau\}$ defined by $\psi_a := u^{-1}\circ\phi_a\circ u$. By construction it satisfies $\psi|B=\psi'$; hence it defines an embedding $B\hookrightarrow \End_L(\psi')$. Thus by Proposition \ref{def_field_end} applied to $\psi'$ over $\ell$ the coefficients of $\psi_a$ for all $a\in A$ lie in some fixed finite extension $\ell'$ of~$\ell$. This means that $\psi$ is really defined over~$\ell'$, and since $\phi\cong\psi$, it follows that $\phi$ is isotrivial, as desired.
\end{proof}

\begin{prop}\label{finite}
The following assertions are equivalent:
\begin{enumerate}
\item[(a)] $\phi$ is isotrivial.
\item[(b)] $\rho_{\text{ad}}( G_K^{\geom})$ is finite. 
\item[(c)] $\rho_\pp( G_K^{\geom})$ is finite for every prime $\pp\not=\pp_0$ of $A$.
\item[(d)] $\rho_\pp( G_K^{\geom})$ is finite for some prime $\pp\not=\pp_0$ of $A$.
\item[(e)] $de=r$.
\end{enumerate}
\end{prop}

\begin{proof}
(Compare \cite{PinII}, Proposition 2.2.)
The implications (a)$\Rightarrow$(b)$\Rightarrow$(c)$\Rightarrow$(d) are obvious. 
For the rest of the proof we may assume that $\End_K(\phi)=R$ after replacing $K$ by a finite extension, using Proposition \ref{def_field_end}. Let $F'$ be a maximal commutative $F$-subalgebra of $R\otimes_A F$, let $A'$ denote the integral closure of $A$ in~$F'\!$, and choose a Drinfeld $A'$-module $\phi'\colon A'\to K\{\tau\}$ and an isogeny $f\colon\phi\to\phi'|A$, as in the proof of Proposition \ref{numerics}. Then $\phi'$ has rank~$r/de$ and endomorphism ring $\End_K(\phi') = A'$.

If (d) holds, there exist a prime $\pp\not=\pp_0$ of $A$ and a finite extension $K'\subset K^{\text{sep}}$ of $K$ such that $\rho_\pp( G_{K'}^{\geom})$ is trivial and hence $\rho_\pp(G_{K'})$ is abelian. After replacing $K$ by $K'$ we may therefore assume that $\rho_\pp(G_K)$ is abelian. Moreover, as in (\ref{eq:tensisom}) we have a $G_K$-equivariant isomorphism $V_\pp(\phi) \cong V_\pp(\phi'|A) \cong \prod_{\pp'|\pp} V_{\pp'}(\phi')$. Thus for any prime $\pp'|\pp$, the image $\rho_{\pp'}(G_K)$ of the Galois representation $\rho_{\pp'}$ on $V_{\pp'}(\phi')$ is abelian, and so the subring $F'_{\pp'}[\rho_{\pp'}(G_K)]$ of $\End_{F'_{\smash{\pp'}}}(V_{\pp'}(\phi'))$ is commutative. By the semisimplicity and Tate conjectures for Drinfeld modules (see \cite{Tagu}, \cite{Tama2}, \cite{Tama1}, \cite{Tama3}) this subring is the commutant of $\End_K(\phi')\otimes_{A'}F'_{\pp'}$. But as $\End_K(\phi') = A'$, this commutant is equal to $\End_{F'_{\smash{\pp'}}}(V_{\pp'}(\phi'))$. It is therefore commutative if and only if $r/de = \dim_{F'_{\smash{\pp'}}}(V_{\pp'}(\phi')) \leq 1$. Thus (d) implies (e).

If (e) holds, then $\phi'$ is a Drinfeld $A'$-module of rank~$1$ and of special characteristic. Since the moduli stack of Drinfeld $A'$-modules of rank $1$ is finite over $\Spec A'$, the Drinfeld module $\phi'$ is isomorphic to one defined over a finite field, i.e., isotrivial.
By Proposition \ref{isotrivialprop} the same then also follows for $\phi'|A$ and for~$\phi$. Thus (e) implies (a), and we are done.
\end{proof}

To determine the images of Galois up to commensurability for an isotrivial Drinfeld module we may reduce ourselves to the case of a Drinfeld module defined over a finite field. In that case the situation is as follows:

\begin{prop}\label{FinGal}
Suppose that $\phi$ is defined over a finite field~$\kappa$. Let $C$ denote the center of $\End_\kappa(\phi)$ and $C'$ the normalization of~$C$. Then there exists an element $c_0\in C$ with the properties:
\begin{enumerate}
\item[(a)] $c_0$ generates a positive power of a unique prime $\pp'_0$ of $C'$ above~$\pp_0$.
\item[(b)] $\rho_{\text{ad}}(\Frob_\kappa)$ coincides with the action of $c_0$ on $\prod_{\pp\not=\pp_0} T_\pp(\phi)$. 
\item[(c)] 
$\rho_{\text{ad}}(G_\kappa) = \smash{\overline{\langle c_0\rangle}}$, the pro-cyclic subgroup topologically generated by~$c_0$.
\end{enumerate}
\end{prop}

\begin{proof}
Let $m$ denote the degree of $\kappa$ over~$\mathbb{F}_p$. Then $\Frob_\kappa = \tau^m$ lies in the center of $\kappa\{\tau\}$. In particular it commutes with $\phi_a$ for all $a\in A$ and is therefore an endomorphism of~$\phi$, and more specifically it lies in the center $C$ of $\End_\kappa(\phi)$. As such let us denote it by~$c_0$. The action of $\Frob_\kappa$ as an element of the Galois group $G_\kappa$ on all Tate modules of $\phi$ is then the same as that obtained from the natural action of $c_0$ as an endomorphism. This directly implies (b) and (c). 

For (a) we apply Proposition \ref{ConstIsog} to $S:= C$ and $S':=C'$, obtaining a Drinfeld $C'$-module $\phi'\colon C'\to \kappa\{\tau\}$ and an isogeny $f\colon\phi\to\phi'|A$. The characteristic of $\phi'$ is then a prime $\pp'_0$ of $C'$ above~$\pp_0$. Also, the endomorphism $\Frob_\kappa$ of $\phi'$ still corresponds to the same element $c_0\in C$. Since $c_0$ acts as a scalar on the Tate modules of $\phi'$, it constitutes the single eigenvalue of $\Frob_\kappa$. Thus Proposition \ref{F_eigenvalues} implies that $c_0$ is divisible by $\pp'_0$ but not by any other prime of~$C'$. This shows (a), and we are done.
\end{proof}


\subsection{The non-isotrivial case}\label{62}

To determine the images of Galois in the general non-isotrivial case, we will use some reduction steps which end in the situation of Theorem \ref{main_theorem}. Recall that Theorem \ref{main_theorem} involves two conditions, namely that $A$ is the center of $R:=\End_{K^{\text{sep}}}(\phi)$ and that the endomorphism ring does not grow under restriction of~$A$. We will achieve the first condition by enlarging~$A$, and then the second condition by shrinking $A$ again until the endomorphism ring stops growing. That this process terminates is a non-trivial fact from \cite{PinII}.

\medskip
To enlarge $A$ we first choose a finite extension $K' \subset K^{\text{sep}}$ of $K$ such that $R=\End_{K'}(\phi)$. Recall that $Z$ denotes the center of $R \otimes_A F$; hence $C := Z\cap R$ is the center of~$R$. Let $C'$ denote the normalization of~$C$. Applying Proposition \ref{ConstIsog} to $S:= C$ and $S':=C'$ over~$K'$, we obtain a Drinfeld $C'$-module $\phi'\colon C'\to K'\{\tau\}$ and an isogeny $f\colon\phi\to\phi'|A$ over~$K'$. The characteristic of $\phi'$ is then a prime $\pp'_0$ of $C'$ above~$\pp_0$. Since $R\otimes_AF \cong \End_{K^{\text{sep}}}(\phi')\otimes_AF$, the construction implies that $C'$ is the center of $\End_{K^{\text{sep}}}(\phi')$. Also, the isogeny $f$ induces a $G_{K'}$-equivariant inclusion of finite index
\begin{myequation}\label{TateIsog}
  \prod_{\pp\not=\pp_0} T_\pp(\phi)
\ \hookrightarrow\ 
  \prod_{\pp\not=\pp_0} T_\pp(\phi'|A).
\end{myequation}%


Next we restrict $\phi'$ to suitable subrings $B$ of~$C'$. By Theorem 6.2 of \cite{PinII} there is a canonical choice for which the endomorphism ring of $\phi'|B$ is maximal:

\begin{prop}\label{ringB}
If $\phi$ is not isotrivial, there exists a unique integrally closed infinite subring $B$ of $C'$ with the following properties:
\begin{enumerate}
\item[(a)] The center of $\End_{K^{\text{sep}}}(\phi'|B)$ is $B$.
\item[(b)] For every integrally closed infinite subring $B'$ of $C'$ we have $\End_{K^{\text{sep}}}(\phi'|B') \allowbreak \subset\nobreak 
\End_{K^{\text{sep}}}(\phi'|B)$.
\end{enumerate}
\end{prop}

With $B$ as in Proposition \ref{ringB} we abbreviate $\psi := \phi'|B$ and $S:= \End_{K^{\text{sep}}}(\psi)$. The characteristic of $\psi$ is $\pq_0:=B\cap\pp_0'$ and hence a maximal ideal of~$B$. By Proposition 3.5 of \cite{PinII} we have:

\begin{prop}\label{uniqueprime}
$\pp_0'$ is the unique prime of $C'$ above $\pq_0$.
\end{prop}



Let $P_0'$ denote the finite set of primes of $C'$ lying above $\pp_0$; then in particular $\pp'_0\in P_0'$. Let $Q_0$ denote the finite set of primes $\pq$ of $B$ such that all primes of $C'$ above $\pq$ lie in~$P_0'$. Then Proposition \ref{uniqueprime} implies that $\pq_0\in Q_0$. Combining the natural isomorphisms $T_\pq(\psi) \cong \prod_{\pp'|\pq} T_{\pp'}(\phi')$ for all primes $\pq\not\in Q_0$ of $B$ and the natural isomorphisms $T_\pp(\phi'|A) \cong \prod_{\pp'|\pp} T_{\pp'}(\phi')$ for all primes $\pp\not=\pp_0$ of~$A$, we obtain a natural $G_{K'}$-equivariant surjection
\begin{myequation}\label{TateCompare}
\prod_{\pq\not\in Q_0} T_\pq(\psi) \ \cong\ 
  \prod_{\pq\not\in Q_0} \prod_{\pp'|\pq} T_{\pp'}(\phi') \ \twoheadrightarrow\ 
  \prod_{\pp'\not\in P'_0} T_{\pp'}(\phi') \ \cong\ 
  \prod_{\pp\not=\pp_0} T_\pp(\phi'|A).
\end{myequation}%
For every prime $\pq\not\in Q_0$ of $B$ let $D_\pq$ denote the commutant of $S\otimes_BB_\pq$ in $\End_{B_\pq}(T_\pq(\psi))$. As in Proposition \ref{boalo} (a) this is an order in a central simple algebra over the quotient field of~$B_\pq$. The product of these rings acts on the left hand side in (\ref{TateCompare}).

\begin{lem}\label{equivariance}
The kernel of the surjection (\ref{TateCompare}) is a $\vphantom{\Big|}\prod_{\pq\not\in Q_0} D_\pq$-submodule, and the induced action of\/ $\prod_{\pq\not\in Q_0} D_\pq$ on the quotient\/ $\prod_{\pp\not=\pp_0} T_\pp(\phi'|A)$ is faithful.
\end{lem}

\begin{proof}
For every prime $\pq\not\in Q_0$ of~$B$, the isomorphism $T_\pq(\psi) \cong \prod_{\pp'|\pq} T_{\pp'}(\phi')$ is the isotypic decomposition of $T_\pq(\psi)$ under $C'\otimes_BB_\pq$. Since $C'$ is contained in~$S$, the definition of $D_\pq$ shows that the actions of $C'\otimes_BB_\pq$ and $D_\pq$ commute; hence the decomposition is $D_\pq$-invariant. As the kernel of the surjection (\ref{TateCompare}) is a product of certain factors $T_{\pp'}(\phi')$, this implies the first assertion of the lemma. For the second note that, by the construction of~$Q_0$, for every prime $\pq\not\in Q_0$ of $B$ there exists a prime $\pp'\not\in P_0'$ of $C'$ with $\pp'|\pq$. Then $T_{\pp'}(\phi')$ is a non-trivial module over $D_\pq$, and it remains so after tensoring with the quotient field of~$B_\pq$; hence $D_\pq$ acts faithfully on it. Taking the product over all $\pq\not\in Q_0$ proves the second assertion.
\end{proof}

Let ${\mathcal D}$ denote the stabilizer in $\prod_{\pq\not\in Q_0} D_\pq$ of the image of the homomorphism (\ref{TateIsog}). By construction this is a closed subring of finite index, and Lemma \ref{equivariance} implies that ${\mathcal D}$ acts faithfully on $\prod_{\pp\not=\pp_0} T_\pp(\phi)$. For each $\pq\not\in Q_0$ let $D^1_\pq$ denote the multiplicative group of elements of $D_\pq$ of reduced norm~$1$. Then ${\mathcal D}^1 := {\mathcal D}^\times \cap \prod_{\pq\not\in Q_0} D_\pq^1$ is a closed subgroup of finite index of $\prod_{\pq\not\in Q_0} D_\pq^1$. We can identify ${\mathcal D}^\times$ and ${\mathcal D}^1$ with closed subgroups of $\prod_{\pp\not=\pp_0}\Aut_{A_\pp}(T_\pp(\phi))$.

Finally let $c_0$ be any element of $C'$ that generates a positive power of~$\pp'_0$. Let ${\mathfrak c}' \subset C'$ be the annihilator ideal of the cokernel of the inclusion (\ref{TateIsog}). Then ${\mathfrak c}'\not\subset\pp'_0$; hence it is relatively prime to~$c_0$. Thus after replacing $c_0$ by some positive power we may assume that $c_0\equiv1$ modulo~${\mathfrak c}'$. Then multiplication by $c_0$ is an automorphism of $\prod_{\pp\not=\pp_0} T_\pp(\phi'|A)$ that maps the image of (\ref{TateIsog}) to itself. We can thus view it as an element of $\prod_{\pp\not=\pp_0}\Aut_{A_\pp}(T_\pp(\phi))$. Let $\smash{\overline{\langle c_0\rangle}}$ denote the pro-cyclic subgroup of $\prod_{\pp\not=\pp_0}\Aut_{A_\pp}(T_\pp(\phi))$ that is topologically generated by it. Since $c_0\in C' \subset S$, this subgroup commutes with the action of ${\mathcal D}$ and hence with~${\mathcal D}^1$.

\begin{thm} \label{main_theorem_2}
Let $\phi$ be a non-isotrivial Drinfeld $A$-module over a finitely generated field $K$ of special characteristic~$\pp_0$. Let ${\mathcal D}^1$ and $\smash{\overline{\langle c_0\rangle}}$ denote the subgroups of\/ $\prod_{\pp\not=\pp_0}\Aut_{A_\pp}(T_\pp(\phi))$ defined above. Then
\begin{enumerate}
\item[(a)] $\rho_{\text{ad}}( G_K^{\geom})$ is commensurable to ${\mathcal D}^1$, and 
\item[(b)] $\rho_{\text{ad}}( G_K)$ is commensurable to $\overline{\langle c_0\rangle} \cdot {\mathcal D}^1$.
\end{enumerate}
\end{thm}

\begin{proof}
By Proposition \ref{ringB} the assumptions of Theorem \ref{main_theorem} are satisfied for the Drinfeld $B$-module $\psi$ over~$K'$. Let $\rho^\psi_{\text{ad}} :\ G_{K'} \to \prod_{\pq\not\in Q_0} D_\pq^\times$ denote the homomorphism describing the action of $G_{K'}$ on $\prod_{\pq\not\in Q_0} T_\pq(\psi)$, and let $b_0$ be any element of $B$ that is divisible by $\pq_0$ but not by any other prime of~$B$. Then Theorem \ref{main_theorem} implies that $\rho^\psi_{\text{ad}}( G_{K'}^{\geom})$ is commensurable to $\prod_{\pq\not\in Q_0} D^1_\pq$ and $\rho^\psi_{\text{ad}}(G_{K'})$ is commensurable to $\overline{\langle b_0\rangle} \cdot \prod_{\pq\not\in Q_0} D^1_\pq$.

Viewing $b_0$ as an element of~$C'$, Proposition \ref{uniqueprime} implies that $b_0$ is divisible by~$\pp'_0$ but not by any other prime of~$C'$. The same argument as in Section \ref{55} for $a'$ and~$a_0$ shows here that some positive power of $b_0$ is equal to some positive power of~$c_0$. Thus $\rho^\psi_{\text{ad}}(G_{K'})$ is commensurable to $\overline{\langle c_0\rangle} \cdot \prod_{\pq\not\in Q_0} D^1_\pq$.

By (\ref{TateCompare}) and Lemma \ref{equivariance} the group $G_{K'}$ acts on $\prod_{\pp\not=\pp_0} T_\pp(\phi'|A)$ through the composite of $\rho^\psi_{\text{ad}}$ with the faithful action of\/ $\prod_{\pq\not\in Q_0} D_\pq$. Combining this with (\ref{TateIsog}) and the construction of ${\mathcal D}^1$ we deduce that $\rho_{\text{ad}}( G_{K'}^{\geom})$ is commensurable to ${\mathcal D}^1$ and $\rho_{\text{ad}}( G_{K'})$ is commensurable to $\overline{\langle c_0\rangle} \cdot {\mathcal D}^1$. Since $K'$ is a finite extension of~$K$, the same then follows with $K$ in place of $K'$, as desired.
\end{proof}


\newpage

\nocite{*}
\bibliographystyle{alpha}
\bibliography{biblio}


\end{document}